%% file: QuasiIsoB.tex
\documentclass[12pt]{amsart}
\usepackage{amsfonts,amsmath,amsthm,amssymb,wasysym,rotating,mathrsfs}
\usepackage{tikz,tikz-cd}
\usepackage{graphicx,enumitem}

\usepackage[all]{xy}
\usepackage[vcentermath,enableskew]{youngtab}
\usepackage{hyperref}

\usepackage[margin=1.1in]{geometry}

\usepackage[capitalize,nameinlink,noabbrev,nosort]{cleveref}

\definecolor{darkred}{rgb}{1,0,0}

\newtheorem{maintheorem}{Theorem}	
 
\newtheorem{theorem}{Theorem}[section]
\newtheorem{thm}[theorem]{Theorem}
\crefname{thm}{Theorem}{Theorems}
\newtheorem{lemma}[theorem]{Lemma}
\newtheorem{lem}[theorem]{Lemma}
\crefname{lem}{Lemma}{Lemmas}

\newtheorem{prop}[theorem]{Proposition}

\newtheorem{cor}[theorem]{Corollary}
\newtheorem{conj}[theorem]{Conjecture}
\newtheorem*{conj:samepattern}{Conjecture \ref{conj:samepattern}}
\crefname{conj}{Conjecture}{Conjectures}

\theoremstyle{remark}

\theoremstyle{definition}
\newtheorem{rmk}[theorem]{Remark}
\newtheorem{example}[theorem]{Example}
\newtheorem{defn}[theorem]{Definition}
\crefname{defn}{Definition}{Definitions}

\crefname{question}{Question}{Questions}
\newtheorem*{claim*}{Claim}
\newtheorem*{defn*}{Definition}
\newtheorem*{thm*}{Theorem}

\def\Mat{\operatorname{Mat}}

\def\bd{\operatorname{Bound}}
\def\sep{\operatorname{Sep}}

\newcommand{\bbc}{\mathbb{C}}
\newcommand{\bbd}{\mathbb{D}}

\newcommand{\bbf}{\mathbb{F}}

\newcommand{\bbm}{\mathbb{M}}
\newcommand{\bbn}{\mathbb{N}}

\newcommand{\bbp}{\mathbb{P}}

\newcommand{\bbr}{\mathbb{R}}

\newcommand{\bbz}{\mathbb{Z}}

\newcommand{\mca}{\mathcal{A}}
\newcommand{\mcb}{\mathcal{B}}
\newcommand{\mcc}{\mathcal{C}}

\newcommand{\mcf}{\mathcal{F}}

\newcommand{\mci}{\mathcal{I}}
\newcommand{\mcj}{\mathcal{J}}

\newcommand{\mcm}{\mathcal{M}}

\newcommand{\mcr}{\mathcal{R}}
\newcommand{\mcs}{\mathcal{S}}
\newcommand{\mct}{\mathcal{T}}

\newcommand{\mcw}{\mathcal{W}}

\renewcommand{\aa}{\alpha}
\newcommand{\bb}{\beta}

\newcommand{\eps}{\epsilon}

\newcommand{\GL}{\textnormal{GL}}

\newcommand{\Gr}{{\rm Gr}}

\def\tGr{\widetilde \Gr}
\def\tPi{\widetilde \Pi}

\newcommand{\cev}[1]{\reflectbox{\ensuremath{\vec{\reflectbox{\ensuremath{#1}}}}}}

\DeclareMathOperator{\Aut}{Aut}

\newcommand{\tPio}{\tPi^{\circ}}
\newcommand{\Pio}{\Pi^{\circ}}

\newcommand{\Mato}{\Mat^{\circ}}

\newcommand{\gnec}{\mathscr{I}} 

\newcommand{\rmci}{\vec{\mci}}
\newcommand{\rightI}{\vec{I}}
\newcommand{\lmci}{\cev{\mci}}
\newcommand{\leftI}{\cev{I}}
\newcommand{\sI}{\overset{ \leftarrow \bullet}{I}}
\newcommand{\tI}{\overset{ \rightarrow \bullet}{I}}
\newcommand{\sourceF}{\overset{ \leftarrow \bullet}{\bbf}}
\newcommand{\targetF}{\overset{ \rightarrow \bullet}{\bbf}}
\newcommand{\rt}{\vec{\tau}}
\newcommand{\lt}{\cev{\tau}}

\newcommand{\ler}{\leq_\circ}

\newcommand{\meas}{\tilde{\mathbb{D}}}

\font\co=lcircle10

\def\jr{\ \smash{\raise2pt\hbox{\co \rlap{\rlap{\char'005} \char'007}}
               \raise6pt\hbox{\rlap{\vrule height6.5pt}}
                \raise2pt\hbox{\rlap{\hskip4pt \vrule
          height0.4pt depth0pt
                width7.7pt}}}}
\def\textcross{\ \smash{\lower4pt\hbox{\rlap{\hskip4.15pt\vrule height14pt}}
                \raise2.8pt\hbox{\rlap{\hskip-3pt \vrule height.4pt depth0pt
                                width14.7pt}}}\hskip12.7pt}

\def\textelbow{\ \hskip.1pt\smash{\raise2.75pt%
                \hbox{\co \hskip 4.15pt\rlap{\rlap{\char'004} \char'006}
                \lower6.8pt\rlap{\vrule height3.5pt}
                \raise3.6pt\rlap{\vrule height3.5pt}}
                \raise2.8pt\hbox{%
                  \rlap{\hskip-7.15pt \vrule height.4pt depth0pt
width3.5pt}%
                  \rlap{\hskip4.05pt \vrule height.4pt depth0pt
width3.5pt}}}
                \hskip8.7pt}

\begin{document}
\title{Positroid cluster structures from relabeled plabic graphs}
\date{\today}
\author{C. Fraser}%
\address{School of Mathematics, University of Minnesota, Minneapolis, USA
}%
\email{cfraser@umn.edu}%
\author{M. Sherman-Bennett}
\address{Department of Mathematics,
            University of California at Berkeley,
            Berkeley, CA
            USA
}%
\email{m\_shermanbennett@berkeley.edu}%


\thanks{}
\date{\today}

\begin{abstract}
The Grassmannian is a disjoint union of open positroid varieties $\Pio_\mu$, certain smooth irreducible subvarieties whose definition is motivated by total positivity. The coordinate ring $\bbc[\Pio_\mu]$ is a cluster algebra, and each reduced plabic graph $G$ for $\Pio_\mu$ determines a cluster. We study the effect of relabeling the boundary vertices of $G$ by a permutation $\rho$. Under suitable hypotheses on the permutation, we show that the relabeled graph $G^\rho$ determines a cluster for a different open positroid variety $\Pio_{\pi}$. As a key step in the proof, we show that $\Pio_\pi$ and $\Pio_\mu$ are isomorphic by a nontrivial twist isomorphism. Our constructions yield a family of cluster structures on each open positroid variety, given by plabic graphs with appropriately permuted boundary labels. We conjecture that the seeds in all of these cluster structures are related by a combination of mutations and rescalings by Laurent monomials in frozen variables. We establish this conjecture for (open) Schubert and opposite Schubert varieties. As an application, we also show that for certain reduced plabic graphs $G$, the ``source" cluster and the ``target" cluster are related by mutation and Laurent monomial rescalings.
\end{abstract}

\maketitle
\setcounter{tocdepth}{1}
\tableofcontents

\section{Introduction}\label{sec:intro}

\input{introB.tex}

\vskip .2cm

\noindent{\bf Acknowledgements:~} We thank David Speyer and Lauren Williams for helpful conversations on this topic and thank the anonymous referees for careful reading and helpful comments improving the exposition. We thank Pavel Galashin for pointing us towards \cite{FG} and for his plabic tilings applet, which was instrumental for computing many examples.
M.S.B acknowledges support by an NSF Graduate Research Fellowship No. DGE-1752814.
Any opinions, findings and conclusions or recommendations expressed in this material are those of the authors and do not necessarily reflect the views of the National Science Foundation. C.F. is supported by the NSF grant DMS-1745638 and a Simons Travel Grant.

\section{Background on cluster algebras and positroids}\label{sec:background}

\input{backgroundB.tex}

\section{Relabeled plabic graphs and Grassmannlike necklaces} 
\label{sec:toggles}
\input{togglesB.tex}

\section{Twist isomorphisms from necklaces}\label{sec:twist}

\input{twistsecnB.tex}

\section{Quasi-equivalence and cluster structures from relabeled plabic graphs}
\label{sec:reachable}
\input{reachableB.tex}

\section{Proofs}
\label{secn:proofs}
\input{proofsC.tex}

\bibliographystyle{alpha}
\bibliography{bibliography}

\end{document}

%% file: introB.tex
This paper investigates the coordinate rings of open positroid varieties in the Grassmannian, and various ways these coordinate rings can be identified with cluster algebras whose initial seeds are given by plabic graphs with permuted boundary vertices. 

\emph{Positroid varieties} are irreducible projective subvarieties of the Grassmannian introduced by Knutson, Lam, and Speyer \cite{KLS} as the algebro-geometric counterparts to Postnikov's {\sl positroid cells} \cite{Postnikov}. They can be defined as the images of Richardson subvarieties of the full flag variety under the projection from the flag variety to the Grassmannian, and were studied in this guise by Lusztig \cite{Lus98} and Rietsch \cite{RieClosure}. In particular, Schubert varieties in the Grassmannian are positroid varieties. Associated to each positroid variety $\Pi$ is its {\sl open positroid variety} $\Pio$, a smooth Zariski-open subset of the positroid variety defined by the non-vanishing of certain Pl\"ucker coordinates. From the perspective of cluster algebras, the natural object to study is the affine cone $\tPio$ over the open positroid variety in the Pl\"ucker embedding.

Positroid varieties in $\Gr(k,n)$ are indexed by \emph{permutations of type} $(k,n)$\footnote{Positroid varieties are usually indexed by {\sl decorated} permutations; to simplify exposition we only work with decorated permutations in which all fixed points are colored white (see \cref{subsecn:positroidcombo}).}; we write $\tPio_\pi$ for the cone over the open positroid variety indexed by permutation $\pi$. Our results concern affine cones over open positroid variaties only, so we frequently drop ``open'' and ``cone'' in what follows. For each positroid variety $\tPio_\pi$, Postnikov introduced a family of \emph{reduced plabic graphs} with \emph{trip permutation} $\pi$. These graphs are planar bicolored graphs drawn in the disk with boundary vertices $1, \dots, n$ (cf. \cref{fig:fourplabics}) satisfying certain conditions. Any two graphs with the same trip permutation are connected to each other by certain explicit local moves.

Fomin and Zelevinsky introduced cluster algebras as an algebraic and combinatorial framework for studying the dual canonical basis and total positivity in Lie theory \cite{CAI}. The definition has subsequently found connections with myriad other fields. Let $V$ be an affine variety with coordinate ring $\mathbb{C}[V]$ and field of functions $\mathbb{C}(V)$. A choice of seed $\Sigma$ determines a {\sl cluster structure} on~$V$ provided we have the equality of algebras $\mca(\Sigma) = \mathbb{C}[V]$, where $\mca(\Sigma)$ is the cluster algebra with frozen variables inverted. Thus frozen variables must be nonvanishing functions on~$V$. A choice of cluster structure determines the {\sl positive part} of $V$: the locus where all mutable and frozen variables are positive. Note that cluster structures are far from canonical. There may be many ways to identify $\bbc[V]$ with a cluster algebra, as indeed will be the case with positroid varieties.

Let $G$ be a plabic graph for the top-dimensional positroid variety in $\Gr(k, n)$. Scott \cite{Scott} gave a recipe to produce a ``source" seed $\Sigma^S_G=(\sourceF(G), Q(G))$ endowing the top-dimensional positroid variety in $\Gr(k,n)$ 
with a cluster structure\footnote{Scott's convention was that frozen variables are not inverted, so this was thought of as a cluster structure on the Grassmannian, rather than on the top-dimensional positroid variety. Also, Scott used the equivalent formalism of alternating strand diagrams.}. The cluster $\sourceF(G)$ consists of Pl\"ucker coordinates, with one cluster variable for each face of the graph~$G$. All such seeds $\Sigma^S_G$ are related by mutation, so that plabic graph seeds give rise to a single cluster structure on $\Gr(k,n)$.

The combinatorics underlying Scott's recipe works for arbitrary reduced plabic graphs~$G$, yielding a candidate seed $\Sigma^S_G=(\sourceF(G), Q(G))$. It was long expected that this candidate seed $\Sigma^S_G$ would determine a cluster structure on $\tPio_G$, see \cite[Conjecture 3.4]{MSAcyclic}. This expectation was recently confirmed by Galashin and Lam \cite{GalashinLam}, building on work of Muller and Speyer \cite{MSAcyclic, MSTwist}, Leclerc \cite{Leclerc}, and of the second author with  Serhiyenko and Williams \cite{SBSW} (cf. \cref{rmk:history} for a history). Again, all seeds from plabic graphs are related by mutations, so the results of \cite{GalashinLam} give rise to only one cluster structure on $\tPio_\pi$.

However, as Muller and Speyer \cite{MSAcyclic} note (and as was prevalent in the literature), there is another equally natural way to assign a collection of Pl\"ucker coordinates to a plabic graph, using the ``target" convention rather than the ``source.'' The source seed $\Sigma_G^S$ and the target seed $\Sigma_G^T$ are typically not related by mutations. Moreover, $\mca(\Sigma_G^S)$ and $\mca(\Sigma_G^T)$ have {\sl different} sets of cluster variables, even though $\mca(\Sigma_G^S) = \mca(\Sigma_G^T)=\bbc[\tPio_\pi]$. Muller and Speyer conjectured the following, in slightly different language.

\begin{conj}[\protect{\cite[Remark 4.7]{MSTwist}}]\label{conj:sourceistarget}
	Let $G$ be a reduced plabic graph. Then $\Sigma_G^S$ and $\Sigma_G^T$ are related by a \emph{quasi-cluster transformation}.
\end{conj}

A quasi-cluster transformation \cite{Fra} is a sequence of mutations and well-behaved rescalings of cluster variables by Laurent monomials in frozen variables (cf.~\cref{secn:seedorbits}). If a seed $\Sigma$ determines a cluster structure on~$V$, then any seed related to $\Sigma$ by a quasi-cluster transformation also determines a cluster structure on $V$. Moreover the two cluster structures have the same sets of cluster monomials and the same notion of positive part of $V$.

Our main theorem establishes many different cluster structures on $\tPio_\pi$ via seeds from \emph{relabeled plabic graphs} (cf. \cref{defn:permutedgraph}). Relabeled plabic graph seeds appeared previously in \cite{SBSW} in the course of analyzing Leclerc's cluster structure on positroid varieties \cite{Leclerc}.

If $G$ is a reduced plabic graph with boundary vertices $1, \dots, n$ and $\rho \in S_n$, the relabeled plabic graph $G^\rho$ is the same planar graph but with boundary vertices relabeled to be $\rho(1), \dots, \rho(n)$ (cf. \cref{fig:fourplabics}). The trip permutation, target face labels, etc. of $G^\rho$ are computed using these permuted boundary labels, giving rise to a pair 
$\Sigma_{G ^\rho}^T=(\targetF(G^\rho), Q_{G^\rho})$ as in Scott's recipe. Our next theorem characterizes when a relabeled plabic graph with trip permutation $\pi$ determines a cluster structure on $\tPio_\pi$. We refer the reader to \cref{defn:weakSep} for the important notion of {\sl weak separation} of Pl\"ucker coordinates appearing in (3) of the below theorem, and to \cref{defn:rightweakAffine,defn:rightweakCircular} for the notion $\leq_\circ$ of {\sl circular weak order} on permutations, a ``weak order analogue'' of Postnikov's circular Bruhat order.  
\begin{maintheorem}[{\cref{thm:main}, \cref{cor:posPartsCoincide}}]\label{thm:introSeeds}
	Suppose that $\pi,\rho$ are permutations such that $\pi\rho \ler \pi$. Put $\mu =\rho^{-1}\pi\rho$. Let $G$ be a reduced plabic graph with trip permutation $\mu$ so that $G^\rho$ has trip permutation $\pi$. Then the following four conditions are equivalent:
	\begin{enumerate}
			\item $\Sigma^T_{G^\rho}$ is a seed in $\bbc(\tPio_\pi)$ and $\mca(\Sigma^T_{G^\rho}) = \bbc[\tPio_\pi]$.  
		\item The number of faces of $G^{\rho}$ is $\dim \tPio_\pi$. Equivalently, $\dim \tPio_\pi = \dim \tPio_\mu$. 
		\item The Pl\"ucker coordinates $\targetF(G^\rho)$ associated to the boundary faces (equivalently, to all faces) of $G_{\rho}$ are a weakly separated collection. 	\item The open positroid varieties $\tPio_\pi$ and $\tPio_{\mu}$ are isomorphic.
	\end{enumerate}

Moreover, if any (hence, all) of the above conditions hold, then the positive part of $\tPio_\pi$ determined by the seed $\Sigma^T_{G^\rho}$ is the positroid cell $\tPio_{\pi,>0}$.
\end{maintheorem}

\cref{fig:fourplabics} gives an illustration of \cref{thm:introSeeds}.  The assumption $\pi \rho \ler \pi$ is a natural sufficient condition which ensures that the frozen variables in $\Sigma^T_{G^\rho}$ are nonvanishing on $\tPio_\pi$ (c.f. \cref{thm:rightweak}).

In stating (2), we have used the well known fact that $\dim \tPio_G$ is the number of faces of~$G$. This number of faces can in turn be computed in terms of lengths of affine permutations, so we think of (2) as a Coxeter-theoretic compatibility condition between $\rho$ and $\pi$. Condition (3) is a combinatorial interpretation of this Coxeter-theoretic condition in terms of weak separation; the fact that it is enough to check only the boundary faces was proved by Farber and Galashin \cite[Theorem 6.3]{FG}. The isomorphism $\tPio_\pi \to \tPio_\mu$ asserted in (4) is a generalization of the Muller-Speyer 
{\sl twist automorphism} of an open positroid variety \cite{MSTwist}. These generalized twist isomorphisms are nontrivial from the perspective of matroids (cf. \cref{eg:differentbases}). 

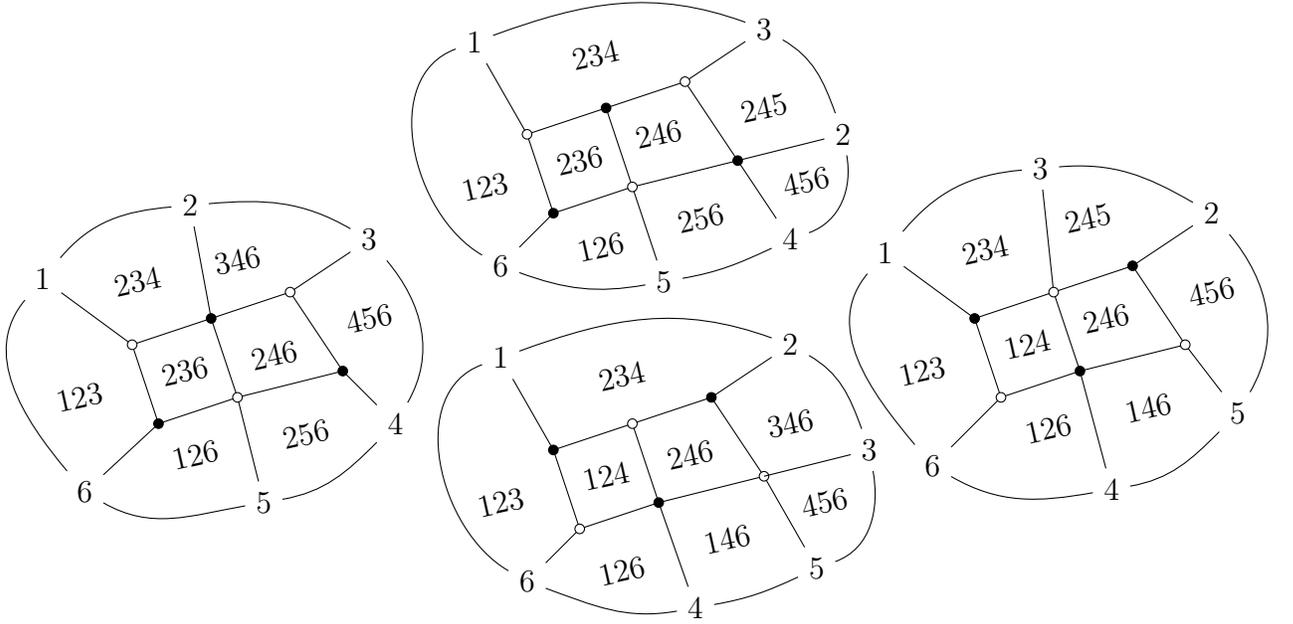
\begin{figure}
	\begin{tikzpicture}[scale = .7]
	\begin{scope}[xshift = -1cm,yshift = -1cm] 
	\node (v10) at (1.1,-2.3) {5};
	\node (v8) at (-2.5,2.25) {1};
	\node (v12) at (3,2.5) {3};
	\node (v9) at (-2,-2) {6};
	\node (v11) at (3.5,-1.5) {4};
	\node (v1) at (-1.5,0.5) {};
	\node (v2) at (0,1) {};
	\node (v5) at (1.5,1.5) {};
	\node (v4) at (-1,-1) {};
	\node (v3) at (0.5,-0.5) {};
	\node (v6) at (2.5,0) {};
	\draw  (v1.center) edge (v2.center);
	\draw  (v2.center) edge (v3.center);
	\draw  (v3.center) edge (v4.center);
	\draw  (v4.center) edge (v1.center);
	\draw  (v2.center) edge (v5.center);
	\draw  (v5.center) edge (v6.center);
	\draw  (v6.center) edge (v3.center);
	\draw  (v1.center) edge (v8);
	\draw  (v4.center) edge (v9);
	\draw  (v3.center) edge (v10);
	\draw  (v6.center) edge (v11);
	\draw  (v5.center) edge (v12);
	\draw[black, fill = white] (v1) circle (.5ex);
	\draw[black, fill = white] (v3) circle (.5ex);
	\draw[black, fill = white] (v5) circle (.5ex);
	\draw[black, fill = black] (v2) circle (.5ex);
	\draw[black, fill = black] (v4) circle (.5ex);
	\draw[black, fill = black] (v6) circle (.5ex);
	\node [rotate = 12.5] at (-.1,-1.65) {$126$};
	\node [rotate = 12.5] at (-2.3,-0.5) {$123$};
	\node [rotate = 12.5] at (-.2,2.0) {$234$};
	\node [rotate = 12.5] at (3.8,-0.4) {$456$};
	\node [rotate = 12.5] at (3,1) {$245$};
	\node [rotate = 12.5] at (1.8,-1.1) {$256$};
	\node [rotate = 12.5] at (-0.5,0) {$236$};
	\node [rotate = 12.5] at (1,0.5) {$246$};
	\node (v7) at (4.5,0.5) {2};
	\draw  (v6.center) edge (v7);
    \draw [out = 20, in = 160] (v8) to (v12);
    \draw [out = -30, in = 110] (v12) to (v7);    
    \draw [out = -80, in = 20] (v7) to (v11);    
    \draw [out = 205, in = 10] (v11) to (v10);    
    \draw [out = 190, in = -20] (v10) to (v9);
    \draw [out = 150, in = 200] (v9) to (v8);        
	\end{scope}
	\begin{scope}[yshift = -7cm,xshift = -.5cm] 
	\node (v10) at (1.2,-2.5) {4};
	\node (v8) at (-2.5,2.25) {1};
	\node (v12) at (3,2.5) {2};
	\node (v9) at (-2,-2) {6};
	\node (v11) at (3.5,-1.75) {5};
	\node (v1) at (-1.5,0.5) {};
	\node (v2) at (0,1) {};
	\node (v5) at (1.5,1.5) {};
	\node (v4) at (-1,-1) {};
	\node (v3) at (0.5,-0.5) {};
	\node (v6) at (2.5,0) {};
	\draw  (v1.center) edge (v2.center);
	\draw  (v2.center) edge (v3.center);
	\draw  (v3.center) edge (v4.center);
	\draw  (v4.center) edge (v1.center);
	\draw  (v2.center) edge (v5.center);
	\draw  (v5.center) edge (v6.center);
	\draw  (v6.center) edge (v3.center);
	\draw  (v1.center) edge (v8);
	\draw  (v4.center) edge (v9);
	\draw  (v3.center) edge (v10);
	\draw  (v6.center) edge (v11);
	\draw  (v5.center) edge (v12);
	\draw[black, fill = white] (v2) circle (.5ex);
	\draw[black, fill = white] (v4) circle (.5ex);
	\draw[black, fill = white] (v6) circle (.5ex);
	\draw[black, fill = black] (v1) circle (.5ex);
	\draw[black, fill = black] (v3) circle (.5ex);
	\draw[black, fill = black] (v5) circle (.5ex);
	\node [rotate = 12.5] at (-.2,-1.8) {126};
	\node [rotate = 12.5] at (-2.5,-0.5) {123};
	\node [rotate = 12.5] at (-.2,1.9) {234};
	\node [rotate = 12.5] at (3.65,-0.5) {456};
	\node [rotate = 12.5] at (3,1) {346};
	\node [rotate = 12.5] at (1.8,-1.2) {146};
	\node [rotate = 12.5] at (-0.5,0) {124};
	\node [rotate = 12.5] at (1.1,0.4) {246};
	\node (v7) at (4.5,0.5) {3};
	\draw  (v6.center) edge (v7);
    \draw [out = 20, in = 160] (v8) to (v12);
    \draw [out = -30, in = 110] (v12) to (v7);    
    \draw [out = -80, in = 20] (v7) to (v11);    
    \draw [out = 205, in = 10] (v11) to (v10);    
    \draw [out = 190, in = -20] (v10) to (v9);
    \draw [out = 150, in = 200] (v9) to (v8);        
	\end{scope}
	\begin{scope}[xshift = -8.5cm, yshift = -5cm] 
	\node (av7) at (-0.4,3.15) {2};
	\node (av10) at (1,-2.5) {5};
	\node (av8) at (-3.2,1.75) {1};
	\node (av12) at (3,2.5) {3};
	\node (av9) at (-2.4,-2.3) {6};
	\node (av11) at (3.5,-1) {4};
	\node (av1) at (-1.5,0.5) {};
	\node (av2) at (0,1) {};
	\node (av5) at (1.5,1.5) {};
	\node (av4) at (-1,-1) {};
	\node (av3) at (0.5,-0.5) {};
	\node (av6) at (2.5,0) {};
	\draw  (av1.center) edge (av2.center);
	\draw  (av2.center) edge (av3.center);
	\draw  (av3.center) edge (av4.center);
	\draw  (av4.center) edge (av1.center);
	\draw  (av2.center) edge (av5.center);
	\draw  (av5.center) edge (av6.center);
	\draw  (av6.center) edge (av3.center);
	\draw  (av2.center) edge (av7);
	\draw  (av1.center) edge (av8);
	\draw  (av4.center) edge (av9);
	\draw  (av3.center) edge (av10);
	\draw  (av6.center) edge (av11);
	\draw  (av5.center) edge (av12);
	\draw[black, fill = white] (av1) circle (.5ex);
	\draw[black, fill = white] (av3) circle (.5ex);
	\draw[black, fill = white] (av5) circle (.5ex);
	\draw[black, fill = black] (av2) circle (.5ex);
	\draw[black, fill = black] (av4) circle (.5ex);
	\draw[black, fill = black] (av6) circle (.5ex);
	\node [rotate = 12.5] at (-.3,-1.6) {$126$};
	\node [rotate = 12.5]  at (-2.5,-0.5) {$123$};
	\node [rotate = 12.5] at (-1.4,1.7) {$234$};
	\node [rotate = 12.5] at (0.5,2.1) {$346$};
	\node [rotate = 12.5] at (3,1) {$456$};
	\node [rotate = 12.5] at (1.8,-1.2) {$256$};
	\node [rotate = 12.5] at (-0.5,0) {$236$};
	\node [rotate = 12.5] at (1.2,0.3) {$246$};
    \draw [out = 50, in = 185] (av8) to (av7);
    \draw [out = 5, in = 150] (av7) to (av12);    
    \draw [out = -50, in = 60] (av12) to (av11);    
    \draw [out = 225, in = 10] (av11) to (av10);    
    \draw [out = 190, in = -30] (av10) to (av9);
    \draw [out = 130, in = 230] (av9) to (av8);        
	\end{scope}
	\begin{scope}[xshift = 7.5cm, yshift = -4.5cm] 
	\node (av7) at (-0.25,3.35) {3};
	\node (av10) at (1.1,-2.75) {4};
	\node (av8) at (-3.2,1.75) {1};
	\node (av12) at (3,2.5) {2};
	\node (av9) at (-2.3,-2.3) {6};
	\node (av11) at (3.5,-1.3) {5};
	\node (av1) at (-1.5,0.5) {};
	\node (av2) at (0,1) {};
	\node (av5) at (1.5,1.5) {};
	\node (av4) at (-1,-1) {};
	\node (av3) at (0.5,-0.5) {};
	\node (av6) at (2.5,0) {};
	\draw  (av1.center) edge (av2.center);
	\draw  (av2.center) edge (av3.center);
	\draw  (av3.center) edge (av4.center);
	\draw  (av4.center) edge (av1.center);
	\draw  (av2.center) edge (av5.center);
	\draw  (av5.center) edge (av6.center);
	\draw  (av6.center) edge (av3.center);
	\draw  (av2.center) edge (av7);
	\draw  (av1.center) edge (av8);
	\draw  (av4.center) edge (av9);
	\draw  (av3.center) edge (av10);
	\draw  (av6.center) edge (av11);
	\draw  (av5.center) edge (av12);
	\draw[black, fill = white] (av2) circle (.5ex);
	\draw[black, fill = white] (av4) circle (.5ex);
	\draw[black, fill = white] (av6) circle (.5ex);
	\draw[black, fill = black] (av1) circle (.5ex);
	\draw[black, fill = black] (av3) circle (.5ex);
	\draw[black, fill = black] (av5) circle (.5ex);
	\node [rotate = 12.5] at (-.1,-1.6) {$126$};
	\node [rotate = 12.5] at (-2.5,-0.5) {$123$};
	\node [rotate = 12.5] at (-1.3,1.8) {$234$};
	\node [rotate = 12.5] at (0.65,2.4) {$245$};
	\node [rotate = 12.5] at (3,1) {$456$};
	\node [rotate = 12.5] at (1.8,-1.2) {$146$};
	\node [rotate = 12.5] at (-0.5,0) {$124$};
	\node [rotate = 12.5] at (1,0.5) {$246$};
    \draw [out = 50, in = 185] (av8) to (av7);
    \draw [out = 5, in = 150] (av7) to (av12);    
    \draw [out = -50, in = 60] (av12) to (av11);    
    \draw [out = 225, in = 10] (av11) to (av10);    
    \draw [out = 190, in = -30] (av10) to (av9);
    \draw [out = 130, in = 230] (av9) to (av8);        
	\end{scope}
	\end{tikzpicture}
	\caption{The left graph is a reduced plabic graph with trip permutation $\pi  = 465213$. It encodes an open positroid variety $\tPio_\pi$. The three graphs on the right are relabeled plabic graphs with trip permutation~$\pi$. Ignoring the permuted boundary labels, the four plabic graphs represent four {\sl different} open positroid varieties, isomorphic to each other by \cref{thm:main}. The face labels of each graph together with the dual quiver give seeds which determine 4 different cluster structures on $\tPio_\pi$. These four seeds are related by quasi-cluster transformations.}\label{fig:fourplabics}
\end{figure}

\cref{thm:introSeeds} provides many seeds $\Sigma_{G ^\rho}^T$ which give a cluster structure on $\tPio_\pi$. When $\tPio_\pi$ is a Schubert or opposite Schubert variety, the boundary vertex relabelings $\rho$ giving seeds for $\tPio_\pi$ form the $\ler$-order ideal below $\pi^{-1}$ (\cref{prop:SchubWholeInterval}). For other $\pi$, the $\rho$ giving seeds are a subset of this order ideal satisfying an explicit but subtle length condition (\cref{defn:TransitiveClosure}).

If $G^\rho$ is a relabeled graph as in the above theorem and $H$ is a usual plabic graph with trip permutation $\pi$, then the seeds $\Sigma^T_{G^\rho}$ and 
$\Sigma^T_H$ are typically mutation-inequivalent. Nonetheless we expect: 
\begin{conj}[Quasi-equivalence conjecture]\label{conj:samepattern}
	If $H$ is a plabic graph with trip permutation $\pi$, and if $G^\rho$ is a relabeled plabic graph with trip permutation $\pi$ satisfying the conditions of \cref{thm:introSeeds}, then $\Sigma^T_{H}$ and $\Sigma^T_{G^\rho}$ are related by {\sl quasi-cluster transformations}. 
\end{conj}

As we explain in \cref{rmk:sourceisrelabeled}, one can view the source seed $\Sigma_H^S$ as a relabeled seed $\Sigma_{H^\rho}^{T}$ for an appropriate relabeling $\rho$ as in the above theorem. Thus \cref{conj:samepattern} generalizes the Muller-Speyer conjecture (\cref{conj:sourceistarget}) and puts it in a broader framework. We note that in general, seeds $\Sigma$ and $\Sigma'$ giving two different cluster structures on a variety may not be related by quasi-cluster transformations. Zhou \cite{Zhou} gives an example of this for the cluster algebra of the Markov quiver.

We verify \cref{conj:samepattern} for Schubert and opposite Schubert varieties.

\begin{maintheorem}[\cref{thm:SchubGood}]
	Suppose $\tPio_\pi$ is an open Schubert or opposite open Schubert variety. Then \cref{conj:samepattern} holds. In particular, source seeds $\Sigma_H^S$ and target seeds $\Sigma_H^T$ are related by quasi-cluster transformations.
\end{maintheorem}

We also give partial results towards \cref{conj:samepattern} for arbitrary open positroid varieties in \cref{thm:alignedIsQuasiEquiv}. From these results, we obtain a positive answer to \cref{conj:sourceistarget} for $\tPio_\pi$ where $\pi$ is \emph{toggle-connected} (cf. \cref{defn:TransitiveClosure}).

\begin{maintheorem}[\cref{cor:quasiCoinciding}]
	Suppose $\pi \in S_n$ is toggle-connected, and let $H$ be a reduced plabic graph with trip permutation $\pi$. Then the source seed $\Sigma_H^S$ and the target seed $\Sigma_H^T$ are related by a quasi-cluster transformation.
\end{maintheorem}

\textbf{Motivation.} In the special case of the top-dimensional positroid variety (i.e., for the cluster  algebra associated to the Grassmannian itself), the relationship between Pl\"ucker coordinates, cluster variables, plabic graphs, and clusters, is very clean: all Pl\"ucker coordinates are cluster variables and the clusters consisting entirely of Pl\"ucker coordinates are precisely those that come from plabic graphs \cite[Theorem 7.1]{OPS}. Source and target seeds give rise to the same cluster structure. Things are much murkier for lower-dimensional positroid varieties. There are Pl\"ucker coordinates which do not identically vanish on $\tPio_\pi$ but which are not cluster variables in $\mathcal{A}(\Sigma_H^T)$. Moreover, for certain trip permutations $\pi$, there is a unique plabic graph~$H$ with trip permutation~$\pi$, hence the combinatorics of plabic graphs gives us only one seed in $\mathcal{A}(\Sigma_H^T)$.

\cref{thm:main} and ~\cref{conj:samepattern} are meant to rectify this. \cref{thm:main} provides us with a much larger family of seeds for $\mathbb{C}[\tPio_\mu]$ consisting of Pl\"ucker coordinates. \cref{conj:samepattern} asserts that each of these seeds can be rescaled by frozens to give a seed in the ``usual'' cluster structure $\mathcal{A}(\Sigma_H^T)$, thereby providing many more Pl\"ucker clusters in $\mathcal{A}(\Sigma_H^T)=\mathbb{C}[\tPio_\mu]$ (working up to Laurent monomials in the frozen variables).

\medskip

{\bf Outline.} In \cref{sec:background}, we provide background on open positroid varieties and their cluster structures, as well as bounded affine permutations and the partial order $\leq_\circ$, following \cite{Postnikov,KLS,Scott}. We also recall basic notions of cluster algebras and quasi-cluster transformations \cite{CAI,Fra}. \cref{sec:toggles} introduces the main players: relabeled plabic graphs and Grassmannlike necklaces. \cref{sec:thms} introduces the the operation of {\sl toggling} a Grassmannlike necklace, the Unit Necklace \cref{thm:rightweak}, and our main theorem. \cref{sec:twist} establishes isomorphisms of open positroid varieties via twist maps and proves an analogue of the main commutative diagram relating the twist map to the boundary measurement map. It assumes some familiarity with the main constructions in \cite{MSTwist}. \cref{sec:reachable} explores the quasi-equivalence conjecture.  \cref{secn:proofs} collects some longer proofs.

%% file: backgroundB.tex
We review the positroid stratification of the Grassmannian and the combinatorial structures which underlie this stratification. We recall the source and target cluster structures on open positroid varieties, whose construction we will generalize to the setting of {\sl relabeled plabic graphs} in later sections. We discuss quasi-cluster transformations. We also give background on affine permutations. 

For a $k \times n$ matrix $M$ and a $k$-subset $I \in \binom {[n]} k$, let $\Delta_I(M)$ denote the $I$th {\sl Pl\"ucker coordinate} of $M$, i.e. the determinant of the $k \times k$ submatrix of $M$ in columns $I$. If two matrices $M$ and $M'$ have the same row span, then they are related by left multiplication by an element of $\GL_k$. Therefore, their Pl\"ucker coordinates coincide up to a common scalar multiple. 
If $\mcs \subset \binom{[n]}k$ we sometimes use the notation $\Delta(\mcs) = \{\Delta_S \colon S \in \mcs\}$.

Let $\Gr(k,n)$ denote the Grassmannian of $k$-dimensional subspaces of $\bbc^n$. The {\sl Pl\"ucker embedding} $\Gr(k,n) \hookrightarrow \bbp^{\binom n k-1}$ realizes $\Gr(k,n)$ as a projective variety. This embedding sends a point $x \in \Gr(k,n)$ to the list of Pl\"ucker coordinates $(\Delta_I(x))_{I \in \binom{[n]}k} \in \bbp^{\binom n k-1}$, where $\Delta_I(x)$ is defined as the Pl\"ucker coordinate of any $k \times n$ matrix representative for~$x$. We will be more concerned with the affine cone $\tGr(k,n) \subset \bbc^{\binom n k}$ over the Pl\"ucker embedding, i.e. the set of points in affine space whose coordinates satisfy the Pl\"ucker relations.

A real $k \times n$ matrix $M$ is {\sl totally nonnegative} if $\Delta_I(M) \geq 0$ for all $I \in \binom{[n]}k$. 
A collection of $k$-subsets $\mcm \subset \binom {[n]} k$ is a {\sl positroid} if it is the column matroid of a totally nonnegative matrix; i.e., if there exists 
a real matrix $M$ such that $\Delta_I(M) >0$ for $I \in \mcm $ and $\Delta_I(M) = 0$ for $I \notin \mcm$. The {\sl closed positroid variety} $\Pi_\mcm$ is the subvariety of $\Gr(k,n)$ whose homogeneous ideal is generated by $\{\Delta_I  \colon I \notin \mcm\}$ \cite{KLS}. That is, 
$$ \Pi_{\mcm}=\{x \in \Gr(k, n): \Delta_I(x)=0 \text{ for all } I \notin \mcm\}.$$

This paper is concerned with the \emph{open positroid variety} $\Pi^\circ \subset \Pi$, a Zariski-open subset of $\Pi$ whose definition we delay to \cref{subsecn:positroidgeometry}.

\subsection{Combinatorial objects that index positroids}\label{subsecn:positroidcombo}
Positroid are naturally labeled by several families of combinatorial objects. We focus on {\sl Grassmann necklaces} and {\sl decorated permutations} here. These objects and the results in this section are due to Postnikov \cite{Postnikov} unless otherwise noted.

We make the following expositional choice: we will give definitions only for \emph{loopless} positroids. A positroid $\mcm \subset \binom{[n]}k$ is {\sl loopless} if for every $i \in [n]$, there exists an $I \in \mcm$ with $i \in I$. If a positroid is not loopless, than one can work over the smaller ground set $[n] \setminus \{i\}$ without affecting any of the combinatorial or algebraic structures below in a significant way. Geometrically, if a positroid $\mcm$ has a loop $i$, then $k \times n$ matrix representatives for points in $\Pio_\mcm$ will have the zero vector in column~$i$. One can project away the $i$th column and work instead with an isomorphic positroid subvariety of $\Gr(k,n-1)$.

The first combinatorial object indexing positroids gives rise to frozen variables in the cluster structure(s) on $\tPi^\circ$. 

\begin{defn} A \emph{forward Grassmann necklace} of type $(k,n)$ is an $n$-tuple  $\rmci~=~(\rightI_1,\dots,\rightI_n)$ drawn from $\binom {[n]} k$ such that for all $a \in [n]$, 
\[a \in \rightI_a \quad \text{ and } \quad \rightI_{a+1} = \rightI_a \setminus a \cup \pi(a) \text{ for some }\pi(a) \in [n].\] Dually, a {\it reverse Grassmann necklace}\footnote{We use the terminology of \cite{MSTwist}, but different conventions. Their reverse Grassmann necklace is the tuple $(\leftI_2, \leftI_3, \dots, \leftI_n, \leftI_1)$.} of type $(k,n)$ is an $n$-tuple $\lmci = (\leftI_1,\dots,\leftI_n)$ in $\binom {[n]} k$ such that for all $a \in [n]$,
\[a \in \leftI_{a+1}\quad\text{ and }\quad\leftI_{a} = \leftI_{a+1} \setminus a \cup \sigma(a)\text{ for some }\sigma(a) \in n.\]   
\end{defn}

\begin{rmk}
The objects just defined might more properly be called {\sl loopless Grassmann necklaces}, because they correspond bijectively with loopless positroids. We will drop the adjective loopless. 
\end{rmk}

\begin{defn}
A permutation $\pi \in S_n$ has {\sl type} $(k, n)$ if $\#\{a \in [n] \colon a \leq \pi^{-1}(a)\}=k$.\end{defn}

If $\rmci$ is a forward Grassmann necklace, then it follows from the definition that the map $a \mapsto \pi(a)$ is a permutation of $[n]$. Moreover, the permutation $\pi$ determines the necklace $\rmci$. One has 
\begin{equation}\label{eq:I1}
\rightI_1 = \{a \in [n] \colon a \leq \pi^{-1}(a)\}
\end{equation}
and the remainder of the necklace can be computed from the data of $I_1$ and $\pi$ using the necklace condition. This establishes a bijection between (loopless) forward Grassmann necklaces of type $(k,n)$ and permutations of type $(k,n)$.

Dually, for a reverse Grassmann necklace $\lmci$, the map $i \mapsto \sigma(i)$ is a permutation of $[n]$, and a similar recipe allows one to recover $\lmci$ from the permutation $\sigma$.

Now we explain how (loopless) positroid subvarieties of $\Gr(k,n)$ are in bijection with Grassmann necklaces of type $(k,n)$, hence also with permutations of type $(k,n)$

For any $i \in [n]$, let $<_i$ denote the order on $[n]$ in which $i$ is smallest and $i-1$ is largest, i.e. $i <_i  i+1 <_i \cdots <_i i-1$. For a pair of subsets $S = \{s_1 <_i\dots <_i s_k\},T = \{t_1 <_i \dots <_i t_k\}$, we say that $S \leq_i T$ if $s_j \leq_i t_j$ for all $j$.

\begin{defn}
Let $\mcm \subset \binom{[n]}k$ be a positroid. For $i \in [n]$, let $\rightI_i$ be the $<_i$-minimal subset of $\mcm$, and let $\leftI_i$ be the $<_i$-maximal subset of $\mcm$. The Grassmann necklace of $\mcm$ is defined as $\rmci_\mcm:= (\rightI_1,\dots,\rightI_n)$ and the reverse Grassmann necklace is $\lmci_\mcm := (\leftI_1,\dots,\leftI_n)$. 
\end{defn}

Postnikov and Oh \cite{Suho}, respectively, proved that $\rmci_\mcm$ and $\lmci_\mcm$ are in fact forward and reverse Grassmann necklaces. Moreover, the permutations $\pi$ and $\sigma$ encoding $\rmci_\mcm$ and $\lmci_\mcm$ are related via $\sigma = \pi^{-1}$.

In our proofs, we frequently use the fact that we can read off the positroid $\mcm$ from either of its necklaces $\rmci$ and $\lmci$. This construction is known as Oh's Theorem \cite{Suho}: 
\begin{equation}\label{eq:recoverpositroidfromnecklace}
\mcm = \left\{S \in \binom{[n]} k \colon \, S \geq_i \rightI_i \text{ for $i \in [n]$}\right\} 
= \left\{S \in \binom{[n]} k \colon \, S \leq_i \leftI_i \text{ for $i  \in [n]$}\right\}.
\end{equation}

In summary, (loopless) open positroid varieties $\tPio \subset \tGr(k,n)$ can be bijectively labeled by a forward Grassmann necklace $\rmci$ of type $(k,n)$, or equivalently by a permutation $\pi$ of type $(k,n)$, or equivalently by a reverse Grassmann necklace $\lmci$ of type $(k,n)$\footnote{whose permutation is $\pi^{-1}$}. We write $\tPio_\pi,$ $\rmci_\pi,$ and  $\lmci_\pi$ to indicate the open positroid variety, Grassmann necklace, and reverse Grassmann necklace indexed by $\pi$.

\subsection{Open positroid varieties}\label{subsecn:positroidgeometry}

Let $\mcm$ be a loopless positroid of type $(k,n)$. It corresponds to forward Grassmann necklace $\rmci$ of type $(k,n)$. Then the {\sl open positroid variety} $\Pio_\mcm$
is the Zariski-open subvariety of $\Pi_\mcm$ on which the necklace variables are non-vanishing: 
$$\Pio_{\mcm} = \{x \in \Gr(k, n) \colon \Delta_I(x)= 0 \text{ for all } I \notin \mcm \text{ and }\Delta_I(x) \neq 0 \text{ for } I \in \rmci\}.$$

We let $\tPi,\tPio \subset \tGr(k,n) \subset \bbc^{\binom n k}$ be the affine cone  over $\Pi,\Pio \subset \bbp^{\binom n k -1}$. The remainder of this paper studies cluster structure(s) on $\tPio$. We remind the reader that, in an abuse of terminology, we will usually call $\tPio$ simply a ``positroid variety."

Algebraically, the coordinate ring of $\tPi$ is the quotient $\bbc[\tPi] = \bbc[\Gr(k,n)] / \mcj$ where $\mcj$ is the ideal $\langle \Delta_I \colon I \notin \mcm\rangle$. The coordinate ring $\bbc[\tPio]$ is the localization of $\bbc[\tPi]$ at the Pl\"ucker coordinates $\Delta(\rmci)$.

\begin{rmk}Open positroid varieties can also be obtained by intersecting $n$ cyclically shifted Schubert cells, i.e. by intersecting Schubert cells with respect to the standard ordered basis  $(e_1,\dots,e_n)$ and each of its cyclic shifts \cite{KLS}. 

\end{rmk}

\subsection{Plabic graphs and weak separation}\label{secn:weaksep}
We review Scott's recipe for obtaining a seed in an open positroid variety from a plabic graph and its connection to weakly separated collections. We will apply this recipe to plabic graphs with permuted boundary vertices in what follows.

We assume some familiarity with the notion of a {\sl plabic graph} $G$ in the disk with $n$ boundary vertices $1,\dots,n$ in clockwise order, and also with the technical condition that such a plabic graph is {\sl reduced}. 
See \cite{Postnikov} for details. We only work with reduced plabic graphs in this paper, and will omit the adjective. We also assume that $G$ has no isolated boundary vertices. Since we work with loopless positroids, we also assume henceforth that 
$G$ has no {\sl black lollipops} (interior black vertices of degree one, connected to a boundary vertex). The leftmost graph in \cref{fig:fourplabics} is a reduced plabic graph.  

\begin{defn}[Trip permutation]
Let $G$ be a reduced plabic graph with boundary vertices labeled $1,\dots,n$. For a boundary vertex with label $a$, the \emph{trip} starting at $a$ is the walk along the edges of $G$ which starts at $a$, turns maximally left at every white vertex and maximally right at every black vertex, and ends at a boundary vertex with label $b$. The {\sl 
trip permutation} of $G$, $\pi = \pi(G) \in S_n$, is defined via $\pi(a) = b$. The boundary vertices $a$ and $b$ are referred to as the  \emph{source} and \emph{target} of the trip, respectively. \end{defn} 

For example, the leftmost graph in \cref{fig:fourplabics} has trip permutation $465213 \in S_6$. See \cref{fig:badRelabeled} for an example of a trip.

Reduced plabic graphs are associated to positroid varieties via their trip permutations. That is, $G$ corresponds to the positroid variety $\tPio_{\pi(G)}$. Every permutation $\pi \in S_n$ arises as a trip permutation for some reduced $G$ \cite[Corollary 20.1]{Postnikov}, so each positroid variety is associated to (at least one) reduced plabic graph. If $G$ is reduced, then the number of faces of $G$ is $\dim \tPio_{\pi(G)}$.

Any trip in a reduced graph $G$ is non-self-intersecting and so any face of $G$ is either to the left or to the right of the trip.

\begin{defn}[Collections from plabic graphs]
Let $F$ be a face of a reduced plabic graph $G$. The \emph{target label} $\tI(F) \in \binom{[n]}k$ of a face $F$ is defined by $j \in \tI(F)$ if and only if $F$ is to the left of the trip with target $j$. The set $\targetF(G) = \{\tI(F) \colon \text{$F$ a face of $G$}\}$ is the \emph{target collection} of $G$.  

Dually, the \emph{source label} $\sI(F) \in \binom{[n]}k$ of a face $F$ of $G$ is defined by $j \in \tI(F)$ if and only if $F$ is ito the left of the trip with {\sl source} $j$\footnote{The location of the dot in the notations $\tI(F)$ versus $\sI(F)$ indicates that this labeling records the source vs target of trips; the orientation of the arrow is meant as a reminder that the boundary faces are labeled by the forward or reverse Grassmann necklace}. The {\sl source collection} of $G$ is $\sourceF(G) := \{\sI(F)\colon \text{$F$ a face of $G$}\}$. 
\end{defn}

If the trip permutation $\pi(G)$ has type $(k,n)$, it is a fact that every face of $G$ is on the left of exactly $k$ trips so that $\tI(F), \sI(F) \in \binom{[n]}k$ as claimed. 

The graph $G$ has $n$ boundary faces. The target labels of these boundary faces are the forward Grassmann necklace, $\rmci_{\pi(G)}=(\rightI_1, \dots, \rightI_n)$. In particular, if $F_a$ is the face just before boundary vertex $a$ in clockwise order, then the target label $\tI(F_a)$ is $\rightI_a$. Likewise, source labels of the boundary faces are the reverse Grassmann necklace $\lmci$ of $\mcm$: the source label $\sI(F_a)$ is $\leftI_a$.

The following definition allows us to describe target collections $\targetF(G) \subset \binom{[n]}k$ intrinsically (i.e., without references to graphs $G$).

\begin{defn}[Weak separation]\label{defn:weakSep}
A pair of subsets $I,J \in \binom {[n]} k$ is \emph{weakly separated} if there is no cyclically ordered quadruple $a < b < c < d$ where $a,c \in I\setminus J$ and $b,d \in J\setminus I$. A {\it weakly separated collection} $\mcc \subset \binom{[n]} k$ is a collection whose members are pairwise weakly separated. 

For a positroid $\mcm$, a weakly separated collection $\mcc \subset \mcm$ is called {\sl maximal} if $I \in \mcm \setminus \mcc$ implies that $\{I  \} \cup \mcc$ is not weakly separated.
\end{defn}

Maximal weakly separated collections which contain Grassmann necklaces are exactly the target face labels of plabic graphs.

\begin{thm}[\cite{OPS}]\label{thm:OPS0}
	Let $\mcm$ be a positroid with Grassmann necklace $\rmci$ and decorated permutation $\pi$. The following are equivalent.
	
	\begin{enumerate}
		\item The collection $\mcc \subseteq \mcm$ is a maximal weakly separated collection containing $\rmci$.
		\item The collection $\mcc$ is equal to $\targetF(G)$ for a plabic graph $G$ with trip permutation $\pi$.
	\end{enumerate}
	
\end{thm}

The following operation on weakly separated collections is crucial.
\begin{defn}[Square move]
Let $\mcc \subset \binom{[n]}k$ be a weakly separated collection and $I \in \mcc$. Suppose there are  cyclically ordered $a<b<c<d \in [n]$, and a subset $S \in \binom{[n]}{k-2}$, such that $I = Sac$, and moreover each of $Sab,Sbc,Scd,Sad \in \mcc$. (We abbreviate $Sac := S\cup \{a,c\}$.) Then $\mcc':= \mcc \setminus I \cup (Sbd)$ is again a weakly separated collection. The passage $\mcc \to \mcc'$ is referred to as a {\sl square move} on $\mcc$. 
\end{defn}

Thinking of a maximal weakly separated collection $\rmci \subset \mcc \subset \mcm$ as a target collection $\mcc = \targetF(G)$, a square move can be performed at $I = \tI(F) \in \mcc$  when $F$ is a square face of $G$ (whose vertices alternate in color). Performing the square move to $\mcc$ amounts to swapping the colors of the vertices in the face $F$.

\begin{thm}[{\cite{OPS}}]\label{thm:OPS}
Let $\mcm$ be a positroid with Grassmann necklace $\rmci$. 
Let $\mcc_1, \mcc_2 \subset \mcm$ be maximal weakly separated collections satisfying $\rmci \subset \mcc_i$ for $i= 1,2$. Then $\mcc_2$ can be obtained from $\mcc_1$ by a finite sequence of square moves (with each intermediate collection $\mcc$ satisfying $\rmci \subset \mcc \subset \mcm$).
\end{thm}

\subsection{Cluster algebras}\label{secn:clusters}
We assume familiarity with cluster algebras. We summarize our notation here with a view towards introducing quasi-cluster transformations in \cref{secn:seedorbits}.

Let $V$ be a rational affine algebraic variety with algebra of regular functions $\bbc[V]$ and field of rational functions $\bbc(V)$. Our case of interest is when $V = \tPi^\circ$ is the affine cone over an open positroid variety. 

\begin{defn}\label{defn:seeds}
Let $m = \dim V$. A \emph{seed} of rank $r$ in $\bbc(V)$ is a pair $(\textbf{x}, Q)$ satisfying the following. 

\begin{itemize}
	\item $\textbf{x} = (x_1,\dots,x_{m}) \subset \bbc(V)$ is a transcendence basis for $\bbc(V)$.
	\item $Q$ is a quiver with vertices $1,\dots,m$, the first $r$ of which are designated \emph{mutable}, and the last $m-r$ of which are designated \emph{frozen}. 
	\item Each {\sl frozen variable} $x_i$ for $i = r+1,\dots,m$ is a unit in $\bbc[V]$, i.e. $\frac{1}{x_i} \in \bbc[V] \subset \bbc(V)$. 
	\end{itemize}

The set $\textbf{x}$ is the \emph{cluster} of the seed, and $x_1, \dots, x_r$ are the \emph{cluster variables}.
\end{defn} 

We denote by $\bbp \subset \bbc(V)$ the abelian group of Laurent monomials in the frozen variables. By our assumptions we in fact have $\bbp \subset \bbc[V]$. 

A seed can be {\sl mutated} in any mutable direction $p \in [r]$ to produce a new seed. Seed mutation produces a new quiver $\mu_p(Q)$ and replaces the cluster variable $x_p \in {\bf x}$ with a new cluster variable $x_p'$.  Mutation at $p$ is an involution. If two seeds are related by a sequence of mutations, we call them \emph{mutation-equivalent}.

The {\sl cluster algebra} $\mca(\Sigma)$ is the $\bbc$-algebra generated by the frozen variables, their inverses, and the cluster variables of all seeds mutation-equivalent to $\Sigma$. 

A seed $\Sigma$ in $\bbc(V)$ determines a \emph{cluster structure} on $V$ if $\mca(\Sigma)=\bbc[V]$. If this is true of $\Sigma$ then it is true of any seed mutation-equivalent to $\Sigma$. 

If $\Sigma$ determines a cluster structure on $V$, then $V$ inherits the following structures:
\begin{itemize}
\item A set of {\sl cluster monomials} in $\bbc[V]$. These are elements $f \in \bbc[V]$ that can be expressed in the form $f = \frac{M_1}{M_2}$ where $M_1$ is a monomial in the variables of some cluster in the seed pattern, and $M_2 \in \bbp$ is a monomial in the frozen variables. Thus, our definition of cluster monomial allows frozen variables in the denominator. 
\item A {\sl totally positive part} $V_{>0} \subset V$. This is the subset where all cluster variables (equivalently, all variables in any particular cluster) evaluate positively.
\item For each seed $\Sigma'$ mutation-equivalent to $\Sigma$,  a rational map $V \dashrightarrow (\bbc^*)^{m}$ given by evaluating the cluster variables. We call its domain of definition the {\sl cluster torus} $V_{\Sigma'} \subset V$. By the Laurent phenomenon for cluster algebras, there is
an inverse map $(\bbc^*)^{m} \hookrightarrow V_\Sigma'$, an open embedding we refer to as the {\sl cluster chart}.
\end{itemize}

\subsection{Source and target cluster structures from plabic graphs}\label{secn:seeds}

Let $G$ be a reduced plabic graph. Its {\sl dual quiver} $Q(G)$ has vertex set identified with the faces of $G$, where boundary faces are frozen and internal faces are mutable. There is an arrow $F \to F'$ in $Q(G)$ if the edge $e$ separating $F$ and $F'$ has vertices of opposite color and, when moving from face $F$ to face $F'$ across $e$, one sees the white vertex of $e$ on the left\footnote{If 2-cycles occur in $Q(G)$, delete them.}.

\begin{defn}[Target and source seeds]
Let $G$ be a reduced plabic graph with trip permutation $\pi$ and open positroid variety $\tPio_\pi$. Then the {\sl target seed} is defined as $\Sigma^T_G := (\Delta(\targetF(G),Q(G)) \subset \bbc(\tPio_\pi)$. The forward Grassmann necklace 
$\Delta(\rmci_\pi) \subset \Delta(\targetF(G))$ are the frozen variables.

Dually, one has the {\sl source seed} $\Sigma^S_G = (\Delta(\sourceF(G)),Q(G))\subset \bbc(\tPio_\pi)$ with frozen variables
$\Delta(\lmci) \subset \Delta(\sourceF(G))$.
\end{defn}

Performing a square move at a face Pl\"ucker $\tI(F) \in \targetF(G)$ amounts to performing a mutation at the variable $\tI(F)$ in the seed $\Sigma^T_G$. Thus, all seeds $\{\Sigma_G^T \colon \pi(G) = \pi\}$ are mutation-equivalent in $\bbc[\tPio_\pi]$. We also  have the dual statement for the source seeds. 

\begin{thm}[\cite{GalashinLam}]\label{thm:GL}
If $G$ has trip permutation $\pi$, then the source seed $\Sigma^S_G \subset \bbc(\tPio_\pi)$ determines a cluster structure on $\tPio_\pi$. The positive part $\tPio_{\pi, >0}$ determined by this cluster structure is the positroid cell $\{x \in \tPio_\pi: \Delta_I(x)>0 \text{ for } I \in \mcm_\pi\}$.
\end{thm}

We call the cluster structure on $\tPio_\pi$ given by source seeds $\Sigma^S_G$ the \emph{source cluster structure}.

\begin{rmk}\label{rmk:history}
Leclerc \cite{Leclerc} established that for open Richardson varieties $\mcr_{v, w} \subset \mcf\ell_n$, there exists a seed $\Sigma \subset \bbc(\mcr_{v,w})$ such that the inclusion $\mca(\Sigma) \subseteq \bbc[\mcr_{v, w}]$ holds. In some cases, he showed that in fact $\mca(\Sigma)=\bbc[\mcr_{v, w}]$. Applying any isomorphism $\phi:\bbc[\mcr_{v, w}]\to \bbc[\tPio]$, Leclerc's results imply that $\phi(\mca(\Sigma))$ is equal to $\bbc[\tPio]$ for some positroid varieties, including Schubert and ``skew Schubert" varieties, and is a cluster subalgebra of $\bbc[\tPio]$ in general. For a particular choice of $\phi$, Serhiyenko, Williams, and the second author showed that for Schubert varieties, $\phi(\Sigma)$ is a target seed $\Sigma_G^T$; for skew-Schubert varieties, $\phi(\Sigma)$ is $\Sigma_{G^\rho}^T$ for $G^\rho$ a relabeled plabic graph with a particular boundary (c.f. \cref{defn:permutedgraph}) \cite{SBSW}. Galashin and Lam later showed that, under a different isomorphism $\psi:~\bbc[\mcr_{v, w}]\to \bbc[\tPio]$, $\psi(\Sigma)$ is a source seed $\Sigma_G^S$. They also showed that $\psi(\mca(\Sigma))$ is the entire coordinate ring $\bbc[\tPio]$.
\end{rmk}

\begin{rmk}\label{rmk:targetpreference} 
Our aesthetic preference is for forward Grassmann necklaces, so we choose to work with target-labeled seeds $\Sigma^T_G$ rather than source-labeled ones as in \cref{thm:GL}. Using twist maps, one can deduce from \cref{thm:GL} that the target-labeled seeds also determine a cluster structure on $\tPio_\pi$, which we call the \emph{target cluster structure}. We give a more general version of this style of argument in \cref{thm:genPlabicClusterStruc}.  
\end{rmk}

A motivating fact for us (observed by Muller-Speyer and Leclerc) is that the seeds $\Sigma_G^T$ and $\Sigma_G^S$ are typically {\sl not} mutation-equivalent, i.e. they do not lie in the same seed pattern. The next section discusses a conjectural remedy.

\subsection{Quasi-equivalent seeds and cluster algebras}\label{secn:seedorbits}
Muller and Speyer conjectured that the source and target cluster structures are ``the same" for an appropriate notion of equivalence in which one is allowed suitable Laurent monomial transformations involving frozen variables. Such a notion was systematized by the first author in \cite{Fra}).

For a seed $\Sigma$ of rank $r$ and a mutable index $i \in [r]$, consider the \emph{exchange ratio}
\begin{equation}\label{eq:yhat}
\hat{y}_i = \prod_{j \in [m]} x_j^{(\textrm{no. arrows } j\to i) - (\textrm{no. arrows } i\to j)}.
\end{equation}
This is the ratio of the two terms on the right hand side of the exchange relation when one mutates at $x_i$. 

\begin{defn}[{\cite{Fra}}]\label{defn:seedorbit} Let $\Sigma$ and 
$\Sigma'$ be seeds of rank $r$ in $\bbc(V)$. Let $\mathbf{x},\tilde{Q},x_i,\hat{y}_i$ denote the cluster, quiver, etc. in $\Sigma$ and use primes to denote these quantities in $\Sigma'$. Then $\Sigma$ and $\Sigma'$ are {\sl quasi-equivalent}, denoted $\Sigma \sim \Sigma'$, if 
the following hold: 
\begin{itemize}
\item The groups $\bbp,\bbp'\subset \bbc[V]$ of Laurent monomials in frozen variables coincide. That is, each frozen variable $x'_i$ is a Laurent monomial in  
$\{x_{r+1},\dots,x_m\}$ and vice versa. 
\item Corresponding mutable variables coincide up to multiplication by an element of~$\bbp$: for $i\in [r]$, there is a Laurent monomial $M_i \in \bbp$ such that $x_i = M_i x'_i \in \bbc(V)$. 
\item The exchange ratios \eqref{eq:yhat} coincide: 
$\hat{y}_i =\hat{y}'_i$ for $i\in [r]$.
\end{itemize} 

Quasi-equivalence is an equivalence relation on seeds.  

Seeds $\Sigma, \Sigma'$ are related by a \emph{quasi-cluster transformation} if there exists a finite sequence $\underline{\mu}$ of mutations such that $\underline{\mu}(\Sigma) \sim \Sigma'$.
\end{defn}

\begin{defn}
We say that cluster algebras $\mca(\Sigma_1)$ and $\mca(\Sigma_2)$ are quasi-equivalent if $\Sigma_1$ and $\Sigma_2$ are related by a quasi-cluster transformation or, equivalently, if any seed in $\mca(\Sigma_1)$ is related to any seed in $\mca(\Sigma_2)$ by a quasi-cluster transformation. 
\end{defn}

By \cite[Section 2]{Fra}, if $p \in [r]$ is a mutable vertex, then seeds $\Sigma \sim \Sigma'$ if and only if $\mu_p(\Sigma) \sim \mu_p(\Sigma')$. This justifies the equivalent formulations in the above definition. 

Geometrically, replacing a seed $\Sigma$ by a quasi-equivalent seed $\Sigma'$ amounts to reparameterizing the domain of the cluster chart $(\bbc^*)^m \hookrightarrow V$ by a Laurent monomial transformation. This does not change the image of this chart (i.e., the cluster torus).

The following lemma is immediate from the definitions.

\begin{lem}
If $\mca(\mcs_1)$ and $\mca(\mcs_2)$ are quasi-equivalent cluster algebras, both determining a cluster structure on $V$, then $\mca(\mcs_1)$ and $\mca(\mcs_2)$ have the same sets of cluster monomials and give rise to the same totally positive part $V_{>0} \subset V$. Each cluster of $\mca(\mcs_1)$ can be obtained from 
a cluster of $\mca(\mcs_2)$ by rescaling its cluster variables by appropriate Laurent monomials in frozen variables.
\end{lem} 

We remind the reader of \cref{conj:sourceistarget}, which states that the source and target cluster structures on $\tPio_\pi$ are quasi-equivalent.

\begin{rmk} \label{rmk:twistsquared}
The target and source collections of a plabic graph $G$ are related by a permutation of indices: we have 
$\tI(F) = \pi(\sI(F))$ for any face $F$. The 
permutation $\pi$ determines the automorphism of $\Gr(k,n)$ by column permutation. We warn that this automorphism {\sl does not} preserve the subvariety $\tPio_\pi$. On the other hand, Muller and Speyer defined a more subtle automorphism $\rt_\pi \in \Aut(\tPio_\pi)$, the {\sl right twist map}. By straightforward calculation using \cite[Proposition 7.13]{MSTwist}, the pullback of a source seed along $\rt_\pi^2$ is quasi-equivalent to a target seed: one has $(\rt^2_\pi)^*(\Sigma^S_G) \sim \Sigma^T_G$. Thus, 
establishing \cref{conj:sourceistarget} is the same as establishing that $\rt_\pi^2 \in \Aut(\tPio_\pi)$ is a quasi-cluster transformation, or that $\rt_\pi^2$ is a {\sl quasi-automorphism} in the language of \cite{Fra}. It is widely expected that $\rt_\pi$, not merely its square $\rt_\pi^2$, is a quasi-cluster automorphism (more specifically that it is the Donaldson-Thomas transformation of $\tPio_\pi$). The methods introduced here do not seem to help in proving this stronger statement. 
\end{rmk}

\subsection{Affine permutations}
The notions of positroid and Grasmmann necklace bear cyclic symmetry that is hidden when we label them by permutations $\pi$. To make this cyclic symmetry more apparent, we also index positroids by certain affine permutations following \cite{KLS}. We collect here the basic notions concerning affine permutations for use in our constructions and proofs.  

{\bf Convention:} we use Greek letters $\pi,\rho,\iota,\mu\dots$ for ordinary permutations and use Roman letters $f,r,i,m,\dots,$ for affine permutations.

\begin{defn}
Let $\tilde{S}_n$ denote the group of bijections $f \colon \bbz \to \bbz$ which are $n$-periodic: $f(a+n) = f(a)+n$ for all $a \in \bbz$. There is a group homomorphism ${\rm av} \colon \tilde{S}_n \twoheadrightarrow \bbz$ sending $f \mapsto \frac{1}{n}\sum_{a=1}^n (f(a)-a)$. We denote by $\tilde{S}_n^k : = \{f  \in \tilde{S}_n\colon {\rm av}(f) = k\}$. 
We say that $f \in \tilde{S}_n$ is {\sl bounded}\footnote{Our definition of bounded differs slightly from the standard definitions since we work with loopless positroids.} if $a <\pi(a) \leq a+n$ for all $a \in \bbz$. We denote by 
$\bd(k, n) \subset \tilde{S}_n^k$ those bounded $f$ with av$(f)=k$.
\end{defn}

By $n$-periodicity, any $f \in \tilde{S}_n$ is determined by its {\sl window notation} $[f(1),\dots,f(n)]$, i.e. its values on $[n] \subset \bbz$. 

For $f \in \tilde{S}_n$, the \emph{length} of $f$ is
$$\ell(f):= \# \{i \in [n], j \in \bbz \colon i < j \text{ and } f(i) > f(j)\}.$$

We have a group homomorphism $\tilde{S}_n \twoheadrightarrow S_n$ sending $f$ to the permutation 
$$\overline{f} \colon a \mapsto f(a)~\mod n.$$ The restriction of this map to $\bd(k, n)$ gives a bijection $$\bd(k, n)\to \{\text{permutations of type }(k, n)\}.$$ We say that $f \in \bd(k, n)$ is the {\sl lift} of its associated permutation $\overline{f} \in S_n$.

One advantage of working with affine permutations is the following dimension formula. If $f \in {\rm Bound}(k,n)$ is the lift of a permutation $\pi$, then 
\begin{equation}\label{eq:lengthiscodim}
\dim \tPio_\pi = \text{ \# of faces in a graph G with trip perm. $\pi$ } = k(n-k)+1-\ell(f).
\end{equation}
The first of these equalities was already discussed in \cref{secn:weaksep}.

\subsection{Right weak order on $\tilde{S}_n^k$ and circular weak order}

The kernel of the map av, $\tilde{S}^0_n$, is a Coxeter group of type $\tilde{A}_{n-1}$ (cf.~\cite[Section 8.3]{BjornerBrenti}). The Coxeter generators are the simple transpositions $s_i = [1,\dots,i+1,i,\dots,n]$  for $i=1,\dots,n-1$, together with $s_0 = [0,2,\dots,n-1,n+1]$. The transpositions $T \subset \tilde{S}^0_n$ are the affine permutations $t_{ab}$ swapping values $a+jn \leftrightarrow b+jn$ for all $j \in \bbz$. The Coxeter length function is the restrition of the length function defined above to $\tilde{S}^0_n$.

\begin{defn}
	Let $f, u, v \in \tilde{S}_n$ satisfying $f=uv$. The factorization $f=uv$ is {\sl length-additive} if 
$\ell(f)=\ell(u)+\ell(v)$.
\end{defn}

The Coxeter group $\tilde{S}^0_n$ is partially ordered by the
{\sl right weak order} $\leq_R$. For $f, u \in \tilde{S}_n^0$ and $v = u^{-1}f$, one has $u \leq_R f$ if and only if $f=uv$ is length-additive. Cover relations in the right weak order on $\tilde{S}_n^0$ correspond to (n-periodically) sorting adjacent values of~$f$. Each such cover relation amounts to right multiplication by an appropriate Coxeter generator $s_i$. 

The cosets of $\tilde{S}_n / \tilde{S}^0_n$ are $\{\tilde{S}_n^k: k \in \bbz\}$. We choose 
$$e_k:a \mapsto a+k \text{ for } a \in \mathbb{Z}$$ as the distinguished coset representative for $\tilde{S}_n^k$. The map $\tilde{S}_n^0 \to \tilde{S}_n^k$ given by $w \mapsto e_k w$ is a length-preserving bijection.

\begin{defn}\label{defn:rightweakAffine}
	Suppose $u,f \in \tilde{S}_n^k$, and let $v:=u^{-1}f$. Then $u \leq_R f$ if and only if $e_k^{-1} u \leq_R e_k^{-1} f$ in $\tilde{S}_n^0$, or, equivalently, if and only if $f=uv$ is length-additive.
\end{defn}

The equivalence of the two definitions follows immediately from the fact that $v \in \tilde{S}_n^{0}$ and that multiplying by $e_k$ does not change length.

Moving down in the right weak order on $\tilde{S}_n^k$ corresponds to (n-periodically) sorting the values of $f$. The minimal element in $(\tilde{S}_n^k, \leq_R)$ is the permutation $e_k$. Following our convention on Greek letters, we set
$$\eps_k := \overline{e_k}=k+1 \dots n 1 \dots k.$$ 

By (a similar argument to) \cite[Lemma 3.6]{KLS}, the subset $\bd(k, n) \subset \tilde{S}^k_n$ is a lower order ideal of the poset $(\tilde{S}_n^k, \leq_R)$. That is, if $f \in \bd(k, n)$ and $g \in \tilde{S}_n^k$ with $g \leq_R f$, then in fact $g \in \bd(k, n)$.

To streamline theorem statements, we also consider the partial order on permutations of type $(k, n)$ induced by $(\bd(k, n), \leq_R)$.

\begin{defn}[Circular weak order]\label{defn:rightweakCircular}
	Suppose $\iota, \pi$ are permutations of type $(k, n)$, with lifts $i, f \in \bd(k, n)$, respectively. Then we define $\iota \ler \pi$ if and only if $i \leq_R f$. The partial order $\ler$ is the \emph{circular weak order} on permutations of type $(k, n)$.
\end{defn}

\begin{rmk} Postnikov defined a circular Bruhat order on permutations of type $(k, n)$ \cite[Section 17]{Postnikov}. Its cover relations involve turning alignments in chord diagrams into crossings. There is a natural ``weak order'' version of Postnikov's order in which one only turns ``simple alignments'' into ``simple crossings''; our circular weak order is the dual of that order.
\end{rmk}

Finally, we give more details on length-additivity and the right weak order on $\tilde{S}_n^k$. It has a characterization in terms of left and right associated reflections, as in the $\tilde{S}_n^0$ case.

For $f \in \tilde{S}_n$, the set of \emph{right associated reflections} of $f$ is

\[T_R(f):= \{t_{a, b}: \ell(f t) < \ell(f)\}.
\]

The set of left associated reflections $T_L(f)$ is defined similarly. It is not hard to see that if $i<j$ with $i \in [n]$ and $j \in \bbz$ satisfies $f(i)>f(j)$, then $t_{i, j} \in T_R(f)$, and vice versa. We have $|T_R(f)|=|T_L(f)|= \ell(f)$.

\begin{lem} \label{lem:lengthAddAndReflections}
	Let $x, y \in \tilde{S}_n$. Then $\ell(xy)=\ell(x)+\ell(y)$ if and only if $T_R(x) \cap T_L(y)= \emptyset$.
\end{lem}

\begin{proof}
	Suppose av$(x)=p$ and av$(y)=q$. Then there exist $w, v \in \tilde{S}_n^0$ such that $x=e_p w$ and $y=v e_q$. Because right and left multiplying by $e_b$ does not change length, $\ell(xy)=\ell(x)+\ell(y)$ if and only if $\ell(wv)=\ell(w)+\ell(v)$. It is a standard fact from Coxeter theory that $\ell(wv)=\ell(w)+\ell(v)$ if and only if $T_R(w) \cap T_L(v)= \emptyset$ \cite[Exercise 1.13]{BjornerBrenti}. Since $T_R(w)=T_R(x)$ and $T_L(v)=T_L(y)$, we are done.
\end{proof}

%% file: togglesB.tex
We introduce relabeled plabic graphs and Grassmannlike necklaces, the combinatorial objects which will give rise to seeds and frozen variables, respectively, for cluster structures on open positroid varieties. We explain that the two constructs are related: Grassmannlike necklaces are exactly the target labels of boundary faces in a relabeled plabic graph. We introduce three conditions (P0),(P1),(P2) which are necessary for a relabeled plabic graph seed to determine a cluster structure.

\subsection{Plabic graphs with relabeled boundary}
Recall that every reduced plabic $G$ for an open positroid variety $\tPio$ gives rise to two seeds, $\Sigma_G^S$ and $\Sigma_G^T$, both of which determine cluster structures on $\tPio$.

 \begin{defn}[Relabeled plabic graph]\label{defn:permutedgraph}
Let $G$ be a reduced plabic graph of type $(k,n)$ and $\rho \in S_n$ a permutation. (Thus, $G$ has boundary vertices $1,\dots,n$ in clockwise order.) The {\sl relabeled plabic graph $G^\rho$ with boundary} $\rho$ is the graph obtained by relabeling the boundary vertex $i$ in $G$ with $\rho(i)$. The plabic graph $G$ is the \emph{underlyling graph} of $G^\rho$. 

	The {\sl trip permutation} $\pi$ of $G^\rho$, {\sl target labels} $\tI(F)$ for $F \in G^\rho$, and {\sl target collection} $\targetF(G^\rho) \subset \binom{[n]}k$	are defined in the same way as in \cref{sec:background}, taking into account the relabeling of boundary vertices\footnote{That is, if $\rho(i) \leadsto \rho(j)$ is a trip of $G^\rho$, then $\pi(\rho(i)) = \rho(j) $, and one puts the value $\rho(j)$ in $\tI(F)$ for every face $F \in G^\rho$ to the left of this trip. Again, $\targetF(G^\rho): = \{\Delta_{\tI(F)} \colon F \in G^\rho\} $.} (see \cref{fig:badRelabeled}). The \emph{target seed} is $\Sigma_{G^\rho}^T = (\Delta(\targetF(G^\rho)),Q(G))$, with $\Delta_{\tI(F)}$ declared frozen when $F$ is a boundary face.
\end{defn} 

Although we use the terminology ``target seed,'' we are not yet viewing the data $\Sigma_{G^\rho}^T$ as a seed on a particular positroid variety. 

\cref{fig:fourplabics} shows three examples of relabeled plabic graphs with their target collections. Each of these graphs has trip permutation $465213$. The 
trip permutations of the underlying graphs are $456312$, $564123$, and $546132$ respectively (in the order top center, bottom center, right). 

\begin{figure}
	\includegraphics[width=0.8\textwidth]{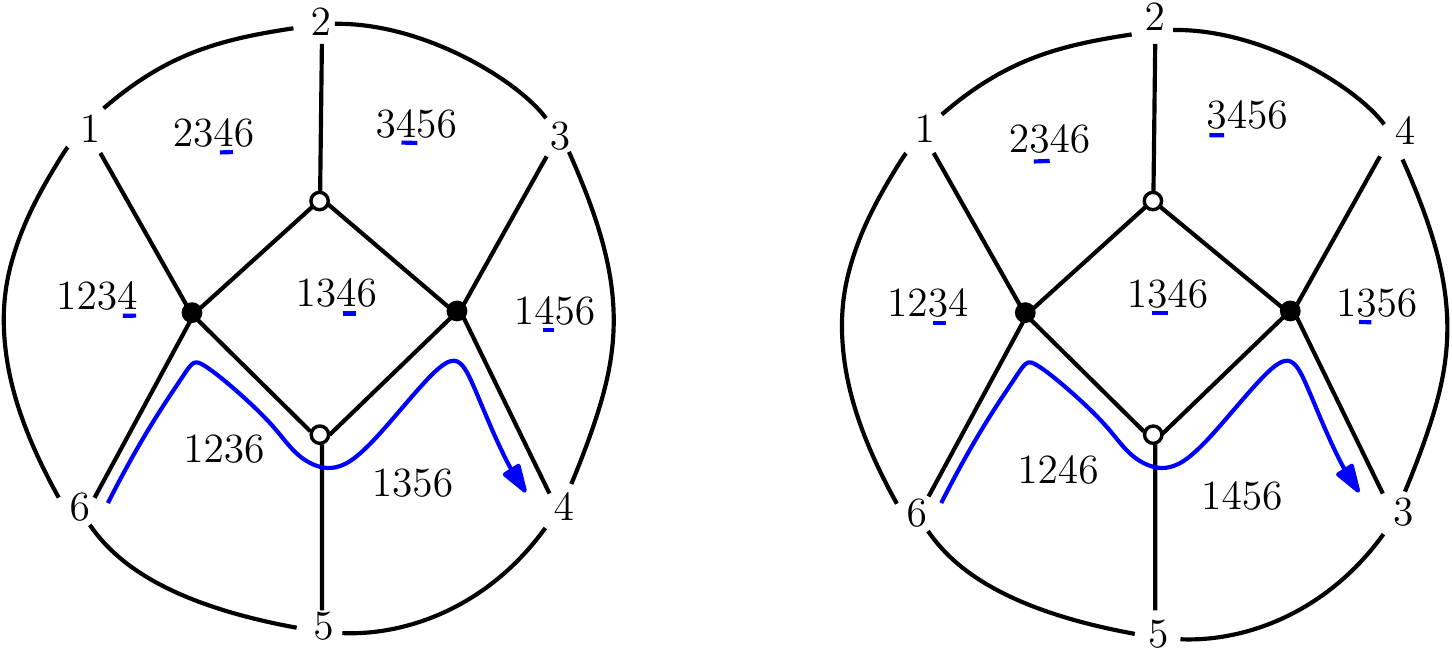}
	\caption{\label{fig:badRelabeled}Left: A reduced plabic graph $G$ with trip permutation $\mu=651324$, shown with target face labels. The trip $6 \leadsto 4$ is in blue. Right: the relabeled plabic graph $G^{s_3}$ with trip permutation $\pi=654123$, shown with target face labels. The trip $6 \leadsto 3$ is in blue.}
\end{figure}

\begin{rmk}[$G$ vs. $G^\rho$]\label{rmk:relationships}
We have the following relationships between $G$ and $G^\rho$. If $G$ has trip permutation $\mu$ then $G^\rho$ has trip permutation $\pi(G^\rho) = \rho \mu \rho^{-1}$. The face collections are related by a permutation of the ground set $\targetF(G^\rho) = \rho(\targetF(G))$. In particular the boundary faces of $G^\rho$ are given by $\rho(\rmci_\mu)$. 
\end{rmk}

\begin{example} \label{eg:sourceisrelabeled}
Let $G$ be a reduced plabic graph with trip permutation $\pi$. The relabeled graph $G^{\pi^{-1}}$ also has trip permutation $\pi$, and the face labels $\targetF(G^{\pi^{-1}})$ are $\pi^{-1}(\targetF(G))=\sourceF(G)$. Thus, the source seed $\Sigma_G^S$ is equal to the target seed $\Sigma_{G^{\pi^{-1}}}^T$, and so \cref{defn:permutedgraph} includes both the target and source seeds of usual plabic graphs.
\end{example}

We adopt the following setup throughout the rest of the paper. Let $\pi$ be a permutation of type $(k,n)$ with open positroid variety $\tPio_\pi$. Let $\bbp_\pi \subset \bbc[\tPio_\pi]$ denote the abelian group in the frozen variables $\Delta(\rmci_\pi)$. Let $f \in {\rm Bound}(k,n)$ be the lift of $\pi$.

Let $G^\rho$ be a relabeled plabic graph with trip permutation $\pi$, whose underlying plabic graph $G$ therefore has trip permutation $\mu = \rho^{-1}\pi\rho$. The goal of the present paper is to investigate conditions under which the target seed $\Sigma_{G^\rho}^T$ determines a cluster structure on the cone~$\tPio_\pi$ over the open positroid variety.  The following conditions are clearly necessary:  
\begin{itemize}

	\item[(P0)][$k$-subsets] The graph $G$, or equivalently the permutation $\mu$, has type $(k,n)$. In particular, the lift of $\mu$ is an affine permutation $m \in {\rm Bound}(k,n)$. 
	\item[(P1)] [Units] If a boundary face of $G^\rho$ has target label $I$, then $\Delta_I$ is a unit in $\bbc[\tPio]$.
	\item[(P2)][Seed size] The underlying graph $G$ has $\dim \tPio_\pi$ many faces. Equivalently by \eqref{eq:lengthiscodim},  $\ell(m) = \ell(f)$. 	

\end{itemize}

The conditions (P0), (P1), and (P2) are certain compatibility conditions between permutations $\pi,\rho \in S_n$. In \cref{subsec:unitlabels} we show that (P0) and (P1) hold when $\pi\rho \ler \pi$. In \cref{subsec:correctlength} we completely characterize when (P2) holds, assuming that $\pi\rho \ler \pi$. 

\begin{example}[Failure of (P0)]\label{eg:P0fails}
To illustrate that the condition (P0) does not always hold, consider the permutation $\pi = 654123$ which has type $(3,6)$ and determines an open positroid variety $\tPio_\pi \subset \tGr(3,6)$. Let $\rho$ be the simple transposition $s_3 = (34)$. Then $\mu = \rho^{-1}\pi \rho = 651324$ is a permutation of type $(4,6)$. That is, the target collection $\targetF(G^\rho)$ of a relabeled plabic graph $G^\rho$ with trip permutation~$\pi$ (e.g. the collection on the right in \cref{fig:badRelabeled}) is contained in $\binom{[6]}4$ rather than in $\binom{[6]}3$. Such a target collection does not determine a set of Pl\"ucker coordinates for $\tPio_\pi$. 
\end{example}

\subsection{Grassmannlike necklaces}
The following combinatorial objects generalize forward and reverse Grassmann necklaces, and axiomatize the possible boundary faces of a relabeled plabic graph (cf. \cref{gnecFaceLabels}).

\begin{defn}[Grassmannlike necklace]\label{defn:gendnecklace}
	A \emph{Grassmannlike necklace} of type $(k,n)$
	is an $n$-tuple $\gnec = (I_1,\dots,I_n)$ of subsets $I_j \in \binom {[n]} k$, with the property that for some permutation $\rho \in S_n$, we have 
\begin{equation} \label{eq:necklace}
	I_{a+1} = I_a \setminus \rho_a \cup \iota_a \text{ for all } a \in [n]
\end{equation}
where $\rho_a \in I_a$ for all $a$\footnote{the index $a$ is considered modulo $n$ here and throughout.}.

The permutation $\rho: a \mapsto \rho_a$ is the \emph{removal permutation} of $\gnec$. It follows that the map $\iota:a \mapsto \iota_a$ is also a permutation of $[n]$, called the {\sl insertion permutation}. We define the {\sl trip permutation} of $\gnec$ as $\pi = \iota\rho^{-1}$, which maps $\rho_a \to \iota_a$ for all $a \in [n]$.  

We write $\gnec = \gnec_{\rho,\iota,\pi}$ to summarize that a Grassmannlike necklace $\gnec$ has removal, insertion, and trip permutations $\rho,\iota,\pi$. Since any two of these permutations determine the third, we sometimes write $\gnec_{\bullet, \iota, \pi}$, $\gnec_{\rho, \bullet, \pi}$ or $\gnec_{\rho, \iota, \bullet}$ for this necklace. Of the three permutations, the trip permutation $\pi$ is the ``most important,'' because we ultimately aim to view $\gnec$ as a set of frozen variables for a cluster structure on $\tPio_\pi$.
\end{defn}

\begin{rmk}
	Our \cref{defn:gendnecklace} is closely related to the cyclic patterns of Danilov, Karzanov and Koshevoy \cite{DKK} and also to Grassmann-like necklaces as defined by  Farber and Galashin \cite{FG}. We have borrowed the latter terminology, although we stress that \cref{defn:gendnecklace} {\sl does not} require that $\gnec$ is a weakly separated collection, as was required in \cite{DKK,FG}. 
\end{rmk}

We depict Grassmannlike necklaces by writing 
\begin{equation}\label{eq:chordsnotation}
\gnec = I_1 \, \substack{\iota_1 \\ \rightleftarrows \\ \rho_1}  \, I_2 \, \substack{\iota_2 \\ \rightleftarrows \\ \rho_2} \, I_3 \, \substack{\iota_3 \\ \rightleftarrows \\ \rho_3}\cdots \substack{\iota_{n-1} \\ \rightleftarrows \\ \rho_{n-1}}\, I_n \, \substack{\iota_n \\ \rightleftarrows \\ \rho_n} I_1,
\end{equation}
i.e., by indicating the removal and insertion permutations in the picture. It is helpful to think of this picture wrapping around cyclically. The trip permutation can be read by reading up the ``columns''  of this picture. 

\begin{example}
	A forward Grassmann necklace 
	\[\rmci_\pi = \rightI_1 \, ~\substack{\pi_1 \\ \rightleftarrows \\ 1}~  \, \rightI_2 \, ~\substack{\pi_2 \\ \rightleftarrows \\ 2} \cdots \substack{\pi_{n-1} \\ \rightleftarrows \\ {n-1}}\, ~\rightI_n \, ~\substack{\pi_n \\ \rightleftarrows \\ n}~ \rightI_1\]
	 is a Grassmannlike necklace with removal permutation the identity and insertion permutation $\pi$. A reverse Grassmann necklace
	 \[\lmci= \leftI_1 \, ~\substack{1 \\ \rightleftarrows \\ \pi^{-1}_1}  \, ~\leftI_2 \, ~\substack{2 \\ \rightleftarrows \\ \pi^{-1}_2} \, \cdots \substack{{n-1} \\ \rightleftarrows \\ \pi^{-1}_{n-1}}\, ~\leftI_n~ \, \substack{n \\ \rightleftarrows \\ \pi^{-1}_n}~ \leftI_1\]
	is a Grassmannlike necklace with insertion permutation the identity and removal permutation $\pi^{-1}$. Both necklaces have trip permutation $\pi$.
\end{example}

\begin{example}\label{ex:shiftedReverseGN}	
	Let $\lmci_\pi=(\leftI_1, \dots, \leftI_n)$ be a reverse Grassmann necklace. It will sometimes be convenient for us to consider instead the Grassmannlike necklace	
\begin{equation}\label{eq:revGrneckconvention}
	\lmci_\pi[k]=(\leftI_{k+1}, \leftI_{k+2}, \dots,\leftI_n,\leftI_1,\dots,\leftI_{k}).
	\end{equation}
which is a rotation of the reverse Grassmann necklace. Rotation does not affect trip permutation, so the trip permutation of $\lmci_\pi[k]$ is $\pi$. The insertion permutation of $\lmci_\pi[k]$ is the Grassmannian permutation $\iota = \eps_k$ and the removal permutation is $\rho = \pi^{-1}\eps_k$.

	\end{example}

Our next three lemmas collect basic properties of Grassmannlike necklaces. 

\begin{lem}
	Let $\gnec_{\rho, \iota, \pi}=(I_1, \dots, I_n)$ be a Grassmannlike necklace. Then $\gnec_{\rho, \iota, \pi}$ is uniquely determined by $\rho$ and $\iota$. In particular, 
	\begin{equation}\label{eq:typekn}
	I_1 = \{a \in [n] \colon \rho^{-1}(a) \leq \iota^{-1}(a) \}
	\end{equation}
	 and the remaining elements are determined from $I_1$ using $\rho$, $\iota$, and \eqref{eq:necklace}.
\end{lem}

\begin{proof}
	The only thing to show is \eqref{eq:typekn}. Consider $a\in [n]$. Then $a$ is removed from $I_{\rho^{-1}(a)}$ only and inserted into $I_{\iota^{-1}(a)+1}$ only. This means that $a$ is in $I_j$ for $j=\iota^{-1}(a)+1, \iota^{-1}(a)+2, \dots, \rho^{-1}(a)$. This cyclic interval includes 1 exactly when $\rho^{-1}(a)\leq \iota^{-1}(a)$.
\end{proof}

Besides the three above permutations 
$\rho$, $\iota$, and $\pi = \iota \rho^{-1}$, there is a fourth permutation 
$\mu$ which is also naturally associated to any Grassmannlike necklace:

\begin{defn}[Underlying permutation]
	Let $\gnec_{\rho, \iota, \pi}$ be a Grassmannlike necklace. The permutation
		\begin{equation}\label{eq:threewaysofwriting}
		\mu:=\rho^{-1}\iota = \rho^{-1}\pi\rho = \iota^{-1}\pi \iota
	\end{equation}
 is the {\sl underlying permutation} of $\gnec_{\rho, \iota, \pi}$.
\end{defn}

As discussed in \cref{rmk:relationships}, if $G^\rho$ is a relabeled plabic graph whose trip permutation is $\pi$, then $\mu$ is the trip permutation of the underlying graph $G$, justifying its name.  The interplay between the trip permutation $\pi$ and the underlying permutation $\mu$ is key in our subsequent results.

We have an $S_n$-action on Grassmannlike necklaces by permuting the ground set: 
$\sigma(\gnec):=(\sigma(I_1), \dots, \sigma(I_n))$ for $\sigma\in S_n$ and a necklace $\gnec=(I_1, \dots, I_n)$. 

\begin{lem}\label{lem:gnecPermuted}
	Let $\gnec_{\rho, \iota, \pi}$ be a Grassmannlike necklace with underlying permutation $\mu$. Then $\gnec=\rho(\rmci_\mu)=\iota(\lmci_\mu)$. In particular, $\gnec_{\rho, \iota, \pi}$ is of type $(k, n)$ if and only if $\rmci_\mu$ (and hence $\mu$ itself) is of type $(k, n)$.
\end{lem}
Thus any Grassmannlike necklace is related to a forward Grassmann necklace (likewise a reverse Grassman necklace) by a permutation of the ground set. 

\begin{proof}
	Let $\gnec_{\rho, \iota, \pi}=(I_1, \dots, I_n)$ and $\rmci_\mu=(\rightI_1, \dots, \rightI_n)$.	Setting $a=\rho(b)$ in \eqref{eq:typekn}, we obtain $I_1=\{\rho(b) \in [n]: b \leq (\rho^{-1}\iota)^{-1}(b) \}$. This is $\rho(\rightI_1)$. Now, $\rho(\rightI_{j+1})$ is $\rho(\rightI_j) \setminus \rho(a) \cup \rho(\mu(a))$. Since $\mu=\rho^{-1} \iota$, this is exactly the necklace condition, and $\gnec= \rho(\rmci_\mu)$. The second equality is similar. 
	
	
\end{proof}

We next connect Grassmannlike necklaces with relabeled plabic graphs. 

\begin{lem}[Grassmannlike necklaces as face labels]\label{gnecFaceLabels}
	Let $G^\rho$ be a relabeled plabic graph of type $(k, n)$ with trip permutation $\pi$. Let $F_1, \dots, F_n$ be the boundary faces of $G^\rho$ in clockwise order with $F_1$ the face immediately before vertex $\rho(1)$.
	
	Then $(\tI(F_1), \dots, \tI(F_n))$ is the Grassmannlike necklace $\gnec=\gnec_{\rho, \bullet, \pi}$. Moreover, every Grassmannlike necklace arises in this way as the boundary face labels of a relabeled plabic graph $G^\rho$, read clockwise.
\end{lem}

\begin{proof}
	The underlying graph $G$ of $G^\rho$ has trip permutation $\mu:=\rho^{-1}\pi\rho$ (\cref{rmk:relationships}). Let $(\rightI_1, \dots, \rightI_n)$ be the forward Grassmann necklace with permutation $\mu$.
	
	By \cref{rmk:relationships}, $\tI(F_j)$ is equal to $\rho(\rightI_j)$. By \cref{lem:gnecPermuted}, this is equal to the $j$th subset in $\gnec$. So $\gnec=(\tI(F_1), \dots, \tI(F_n))$. 
	
	Conversely, if $\gnec=\gnec_{\rho, \iota, \pi}$ is a Grassmannlike necklace, consider any plabic graph $G$ with trip permutation $\rho^{-1}\pi\rho$.  The boundary face labels of the relabeled plabic graph $G^\rho$ (which has trip permutation $\pi$) will be $\gnec$.
\end{proof}

\begin{example}\label{eg:morePofailure}
The type of a Grassmannlike necklace need not match the type of its trip permutation. Indeed, this occurs whenever the $k$-subsets condition (P0) fails. Continuing \cref{eg:P0fails}, the boundary faces of the relabeled graph $G^\rho$ (see \cref{fig:badRelabeled}) are the Grassmannlike necklace 
$$ 1234 \, ~\substack{6 \\ \rightleftarrows \\ 1}~  \, 2346 \, ~\substack{5 \\ \rightleftarrows \\ 2}~  \, 3456 \, ~\substack{1 \\ \rightleftarrows \\ 4}~  \, 1356 \, ~\substack{4 \\ \rightleftarrows \\ 3}~  \, 1456 \, ~\substack{2 \\ \rightleftarrows \\ 5}~  \, 1246\, ~\substack{3 \\ \rightleftarrows \\ 6} 1234.$$
This is a necklace of type (4,6) whose trip permutation $\pi = 654123$ has 
type (3,6). 
\end{example}

\section{When relabeled plabic graphs give seeds}\label{sec:thms}
We begin by identifying a condition which is sufficient to guarantee that a Grassmannlike necklace could serve as frozen variables in a cluster structure, i.e. that the units condition (P1) holds. Then, in the cases that this sufficient condition holds, we state our main theorem characterizing when a seed from a relabeled plabic graph gives rise to a cluster structure on the open positroid variety. 

\subsection{Toggles and the units condition}\label{subsec:unitlabels} We define \emph{toggling}, a local move on Grassmannlike necklaces, and then use toggles to produce a large, well-behaved class of necklaces satisfying the units condition.

\begin{defn}[Unit necklace]
	Let $\gnec$ be a Grassmannlike necklace with trip permutation $\pi$ and let $\bbp_\pi \subset \bbc[\tPio_\pi]$ be the free abelian group of Laurent monomials in the target frozen variables $\Delta(\rmci_\pi)$. We say $\gnec$ is a {\sl unit necklace} if $\Delta(\gnec) \subset \bbp_\pi$.\end{defn}

We are interested in such necklaces because their corresponding Pl\"ucker coordinates could serve as frozen variables for a cluster structure on $\tPio_\pi$. 
\begin{conj} 
The group of units of the algebra $\bbc[\tPio_\pi]$ coincides with the group $\bbp_\pi$. 
\end{conj}

By definition, the forward Grassmann necklace $\rmci_\pi$ is a unit necklace. We 
will construct many more examples of unit necklaces, starting with $\rmci_\pi$ and repeatedly applying the following basic move.
\begin{defn}[Toggling a necklace]\label{defn:toggle}
	Let 
	$\gnec = \gnec_{\rho,\iota,\pi}$ be a Grassmannlike necklace satisfying $\rho_{a-1} \neq \iota_a$ and $\rho_{a} \neq \iota_{a-1}$ for some $a \in [n]$. The operation of {\sl toggling} $\gnec$ {\sl at position}~$a$ yields a new necklace $\gnec'$ whose permutations are given by $(\rho',\iota',\pi') = (\rho \cdot s_{a-1},\iota \cdot s_{a-1},\pi)$. 
\end{defn}

In other words, if 
$$\gnec = I_1 \, \substack{\iota_1 \\ \rightleftarrows \\ \rho_1}  \, I_2 \, \substack{\iota_2 \\ \rightleftarrows \\ \rho_2} \cdots \, \substack{\iota_{a-1} \\ \rightleftarrows \\ \rho_{a-1}} \, I_a \, \substack{\iota_a \\ \rightleftarrows \\ \rho_a}\cdots \substack{\iota_{n-1} \\ \rightleftarrows \\ \rho_{n-1}}\, I_n \, \substack{\iota_n \\ \rightleftarrows \\ \rho_n} I_1,$$ 
then toggling at $a$ produces the Grassmannlike necklace 

$$\gnec' = I_1 \, \substack{\iota_1 \\ \rightleftarrows \\ \rho_1}  \, I_2 \, \substack{\iota_2 \\ \rightleftarrows \\ \rho_2} \cdots \, \substack{\iota_{a} \\ \rightleftarrows \\ \rho_{a}} \, I'_a \, \substack{\iota_{a-1} \\ \rightleftarrows \\ \rho_{a-1}}\cdots \substack{\iota_{n-1} \\ \rightleftarrows \\ \rho_{n-1}}\, I_n \, \substack{\iota_n \\ \rightleftarrows \\ \rho_n} I_1,$$

where $I'_a = I_{a-1} \setminus \rho_a \cup \iota_a = I_{a+1} \setminus \iota_{a-1} \cup \rho_{a-1}$.

\begin{rmk}
	Toggling does not affect the trip permutation or the type of a Grassmannlike necklace. However, it changes the underlying permutation $\mu$ via conjugation by $s_{a-1}$. Toggling at position $a$ is an involution.
\end{rmk}

\begin{defn}[Aligned chords]\label{defn:noncrossing}
	Let $w \neq z ,y \neq x \in [n]$ and consider a pair of chords $w \mapsto x$ and $y \mapsto z$ drawn in the circle with boundary vertices $1,\dots,n$. These chords are called {\sl noncrossing} if they do not intersect (including at the boundary). Two noncrossing chords $w \mapsto x$ and $y \mapsto z$ are  
	{\sl aligned} if we either have $w <_w y <_w z <_w x$ or $w <_w x <_w z <_w y$ (or, if $w=x$, we have $w<_w y <_w z$). See \cref{fig:strandsEx} for an example.
	We say that toggling in position $a$
	is {\sl noncrossing} (resp. \emph{aligned}) if the chords $\rho_{a-1} \mapsto \iota_{a-1}$ and $\rho_{a} \mapsto \iota_{a}$ are noncrossing (resp. aligned). 
\end{defn}

\begin{figure}
	\includegraphics[width=0.7\textwidth]{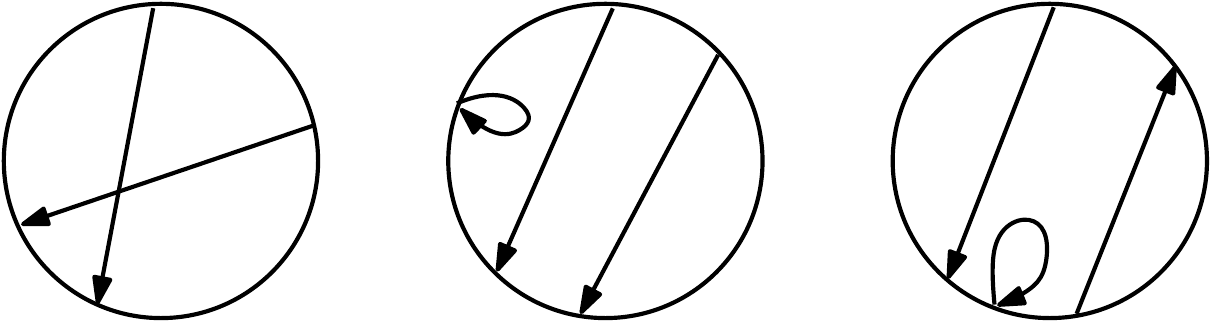}
	\caption{\label{fig:strandsEx} Left: Two crossing chords. Center: Three noncrossing chords, any two of which are aligned. Right: Three noncrossing chords, any two of which are not aligned.}
\end{figure}

\begin{example}\label{ex:easyToggles}
	Consider the Grassmann necklace of type (3, 6)
	\begin{equation}\label{eg:pi}
	\rmci_\pi= 
	123 \, \substack{4\\ \rightleftarrows \\ 1}  \, 234\substack{6 \\ \rightleftarrows \\ 2} \,  
	346\substack{5 \\ \rightleftarrows \\ 3} \,
	456 \, \substack{2\\ \rightleftarrows \\ 4}  \, 256\substack{1 \\ \rightleftarrows \\ 5} \, 126\substack{3\\ \rightleftarrows \\ 6}
	\end{equation}
	whose trip permutation and insertion permutation are $\pi  = \iota =465213$ (see the left of \cref{fig:fourplabics} for a reduced plabic graph with this trip permutation).

	The toggles of $\rmci$ at $3$ and $5$ are aligned, and all other toggles are crossing. Toggling $\rmci$ at $3$ yields the Grassmannlike necklace $\gnec_1$ pictured below while toggling $\rmci$ at $5$ yields the necklace $\gnec_2$.	
	\begin{equation}\label{eq:bottomneck}
	\gnec_1:= 
	123 \, \substack{4\\ \rightleftarrows \\ 1}  \, 234\substack{5 \\ \rightleftarrows \\ 3} \,  
	245\substack{6 \\ \rightleftarrows \\ 2} \,
	456 \, \substack{2\\ \rightleftarrows \\ 4}  \, 256\substack{1 \\ \rightleftarrows \\ 5} \, 126\substack{3\\ \rightleftarrows \\ 6}.
	\end{equation}
	\begin{equation}\label{eq:topneck}
	\gnec_2:= 
	123 \, \substack{4\\ \rightleftarrows \\ 1}  \, 234\substack{6 \\ \rightleftarrows \\ 2} \,  
	346\substack{5 \\ \rightleftarrows \\ 3} \,
	456 \, \substack{1\\ \rightleftarrows \\ 5}  \, 146\substack{2 \\ \rightleftarrows \\ 4} \, 126\substack{3\\ \rightleftarrows \\ 6}.
	\end{equation}

	Relabeled plabic graphs whose boundaries are these necklaces are shown in the top center and bottom center of \cref{fig:fourplabics}.
\end{example}

Oh's Theorem \eqref{eq:recoverpositroidfromnecklace} describes the positroid $\mcm_\pi$ in terms of the Grassmann necklace $\rmci_\pi$ (and also in terms of $\lmci_\pi$). Our next lemma is a weak version of this: for a Grassmannlike necklace $\gnec = \gnec_{\bullet, \iota, \pi}$ obtained from $\rmci_{\pi}$ by a sequence of noncrossing toggles, we can obtain some information about $\mcm_\pi$ directly from the necklace $\gnec$. This somewhat technical lemma is used in the proof of \cref{prop:monomialtransform}. We defer its proof to \cref{secn:rightweakproof}.

\begin{lem}\label{lem:notinpositroid}
	Let $\gnec = (I_1,\dots,I_n)$ be a Grassmannlike necklace that can be obtained from the forward Grassmann necklace $\rmci_\pi$ by a finite sequence of noncrossing toggles.

If $y<_z \pi(z)$ and $y \notin I_{\rho^{-1}(z)}$, then $I_{\rho^{-1}(z)} \setminus z \cup y \notin \mcm_\pi$. Likewise, if $\pi(z)<_z y$ and $y \in I_{\rho^{-1}(z)}$, then $I_{\rho^{-1}(z)} \setminus y \cup \pi(z) \notin \mcm_\pi$.
\end{lem}

\begin{rmk}\label{rmk:conclusionstillholds}
Suppose that $\gnec = (I_1,\dots,I_n)$ and let $\gnec[r] = (I_{r},\dots,I_n,I_1,\dots,I_{r})$ be a cyclic shift of this necklace. The conclusion of \cref{lem:notinpositroid} is invariant under cyclic shift. Thus, if the conclusion holds for $\gnec$, it holds for its cyclic shift  $\gnec[r]$. We use this in the proof of \cref{prop:twistComposeId}. 
\end{rmk}

The following observation underpins the relationship between toggles and unit necklaces. 

\begin{rmk} \label{rmk:togglePluckerRel}
	Toggling is related to three-term Pl\"ucker relations as follows. Consider a Grassmannlike necklace $\gnec=(I_1, \dots, I_n)$ and a position~$a$ at which a toggle can be performed, involving two chords which are not loops. Let $S:= I_{a-1} \setminus \{\rho_{a-1},\rho_a\} \in \binom{[n]}{k-2}$.
Nearby the toggle, the subsets $I_{a-1},I_a,I_{a+1}$ take the form 
$$S\rho_{a-1}\rho_a \, \substack{\iota_{a-1} \\ \rightleftarrows \\ \rho_{a-1}} \, S\iota_{a-1}\rho_a \substack{\iota_a \\ \rightleftarrows \\ \rho_a} \, S\iota_{a-1}\iota_a.$$
Let $I_a' = S\rho_{a-1}\iota_a$ be the result of toggling, and let $S_1 = S\iota_{a-1}\rho_{a-1}$ and $S_2 = S\iota_{a}\rho_{a}$. We have the following Pl\"ucker relation in $\bbc[\Gr(k, n)]$:
\begin{equation}\label{eq:signs}
\Delta_{I_a}\Delta_{I_a'} = 
\begin{cases}
\Delta_{I_{a-1}}\Delta_{I_{a+1}} + \Delta_{S_1}\Delta_{S_2} & \text{ if the toggle at $a$ is aligned} \\
\Delta_{I_{a-1}}\Delta_{I_{a+1}} - \Delta_{S_1}\Delta_{S_2} & \text{ if the toggle at $a$ is noncrossing and nonaligned} \\
\Delta_{S_1}\Delta_{S_2}-\Delta_{I_{a-1}}\Delta_{I_{a+1}} & \text{ if the toggle at $a$ is crossing.}
\end{cases}
\end{equation}
\end{rmk}

\begin{prop}[Noncrossing toggles and unit necklaces]\label{prop:monomialtransform}
	Suppose that by a finite sequence of noncrossing toggles, we move from the forward Grassmann necklace $\rmci_\pi$ to a Grassmannlike necklace $\gnec = (I_1,\dots,I_n)$. Let $\gnec' = (I'_1
	,\dots,I_n')$ be the result of performing a noncrossing toggle to $\gnec$ in position~$a$. Then 
	\begin{equation}\label{eq:monomialtransform}
	\Delta(I_a') = \frac{\Delta(I_{a-1})\Delta(I_{a+1})}{\Delta(I_{a})} \in \bbc[\tPio_\pi]
	\end{equation}
	 and $\gnec'$ is a unit necklace.
\end{prop}

\begin{proof}
	From \eqref{eq:signs}, it suffices by induction to show that when we perform a noncrossing toggle on $\gnec$, either $S_1 \notin \mcm$ or $S_2 \notin \mcm$ (using the notation of \cref{rmk:togglePluckerRel}). 
	
	Suppose we wish to perform a noncrossing toggle at the necklace $\gnec = \gnec_{\rho,\iota,\pi}$ reachable from $\rmci$ by a sequence of noncrossing toggles. Let $\substack{\pi(a) \\ \rightleftarrows \\ a}$
	be the insertion and removal values to the left of the subset which is going to be toggled, i.e. we are toggling at the subset $I_{\rho^{-1}(a)+1} \in \gnec$. Let $L=I_{\rho^{-1}(a)}$ and $R =I_{\rho^{-1}(a)+1}$ so that we are toggling at $R$, and locally the necklace looks like 
	$L \substack{\pi(a) \\ \rightleftarrows \\ a}R \substack{t \\ \rightleftarrows \\ \pi^{-1}(t)}X$ for some $t \in [n]$ and $X \in \binom{[n]}k$.
	
	Since the toggle is noncrossing, we either have that 
	$\{t,\pi^{-1}(t)\} \subset (a,\pi(a))$ or 
	$\{t,\pi^{-1}(t)\} \subset (\pi(a),a)$, where $(a, \pi(a))$ denotes the cyclic interval $a<_a a+1 <_a \dots, <_a \pi(a)$ and similarly for $(\pi(a),a)$.

	In the first situation the subset $S_2$ can be written as $R \setminus \pi(a) \cup t$ with $t <_a \pi(a)$. Hence $S_2 \notin \mcm_\pi$ via \cref{lem:notinpositroid}. 
	
	In the second situation we can write $S_1 = L \setminus \pi^{-1}(t) \cup \pi(a)$. Since we are in the second situation, we have $\pi(a)<_a \pi^{-1}(t)$ so that $S_1 \notin \mcm_\pi$ by \cref{lem:notinpositroid}. 
\end{proof}

\begin{example}[Toggle as monomial transformation]
	Consider the Grassmann necklace $\rmci_\pi$ from \cref{ex:easyToggles}. By Oh's Theorem \eqref{eq:recoverpositroidfromnecklace}, the positroid corresponding to $\pi $ is $\mcm_\pi = \binom{[6]}{3} \setminus \{345,156\}$. A reduced plabic graph with trip permutation $\pi$ is on the left of \cref{fig:fourplabics}.
	
	Performing an aligned toggle on $\rmci_\pi$ at 3 replaces 
	$346$ with $245$. By a 3-term Pl\"ucker relation we have 
	$$
	\Delta_{346}\Delta_{245} = \Delta_{234}\Delta_{456}+ \Delta_{246}\Delta_{345}.
	$$
	However, $\Delta_{345}$ vanishes on $\tPio_\pi$, so in $\bbc[\tPio_\pi]$ we have the relation
	$$\Delta_{346}\Delta_{245} =\Delta_{234}\Delta_{456}.$$
	In other words, the new variable $\Delta_{245}$ is the Laurent monomial $\frac{\Delta_{234}\Delta_{456}}{\Delta_{346}}$ in $\bbc[\tPio_\pi]$, which is the monomial given in 
	\eqref{eq:monomialtransform}. Similarly, toggling $\rmci_\pi$ at 5 replaces $256$ with $146$, and we have $\Delta_{146} = \frac{\Delta_{456}\Delta_{126}}{\Delta_{256}}$ in $\bbc[\tPio_\pi]$. 
\end{example}

\cref{prop:monomialtransform} gives us a large class of unit Grassmannlike necklaces---obtained from $\rmci_\pi$ by noncrossing toggles---and thus many candidates for frozen variables of a cluster structure on $\tPio$. However, for the rest of the paper we restrict our attention necklaces obtained from $\rmci_{\pi}$ by \emph{aligned} toggles. As justification for this we have:

\begin{lem}\label{lem:muchisgained}
Suppose that $\gnec = (I_1,\dots,I_n)$ is a weakly separated Grassmannlike necklace whose trip permutation $\pi$ is a derangement. Then any noncrossing toggle is an aligned toggle. 
\end{lem}

\begin{proof}
Let $\substack{x \\ \rightleftarrows \\ w} I_a \substack{z \\ \rightleftarrows \\ y}$ be the chords nearby a noncrossing toggle. By definition of toggle and the derangement assumption, we have $\{w,y\} \cap \{x,z\} = \emptyset$. Thus, $x,z \in I_{a+1} \setminus I_{a-1}$ while $w,y \in I_{a-1} \setminus I_{a+1}$. If the toggle is not aligned, then the numbers $w,x,y,z$ have cyclic order either $w<x<y<z $ or $w<z<y<x$. So $I_{a+1}$ and $I_{a-1}$ are not weakly separated. 
\end{proof}

Thus, in the world of weakly separated necklaces, there is no difference between aligned and noncrossing toggles. As a second justification, the set of necklaces that can be reached from $\rmci_\pi$ by a sequence of aligned toggles is easy to describe: they are the necklaces $\gnec_{\bullet, \iota, \pi}$ with $\iota \ler \pi$, as the next lemma shows.

\begin{lem}[Aligned toggles and weak order]\label{lem:alignedMovesDown}
	Let $\iota,\pi$ be permutations of type $(k, n)$ with lifts $i, f \in \bd(k, n)$. Consider the Grassmannlike necklace $\gnec=\gnec_{\bullet,\iota,\pi}$.
	
	Suppose $\iota \ler \pi$. Then the toggle of $\gnec$ at $a$ is aligned if and only if $is_{a-1} \leq_R f$, or equivalently, if $\iota \overline{s_{a-1}} \ler \pi$. 
	
\end{lem}

\begin{proof} Let $r:=f^{-1}i$, so that $f: r(a) \mapsto i(a)$. We are assuming that $i \leq_R f$.

	$(\Leftarrow):$ Suppose that $i s_{a-1} \leq_R f$. Since the toggle of $\gnec$ at $a$ and the toggle of $\gnec_{\bullet, \iota \overline{s_{a-1}}, \pi}$ at $a$ involve the same chords, we may exchange $i$ and $is_{a-1}$ if necessary. So without loss of generality, we have $i(a-1)<i(a)$ and $\ell(is_{a-1})> \ell(i)$. 
	
	By assumption $f=(i s_{a-1})(s_{a-1}r^{-1})$ is length-additive. So we have $\ell(s_{a-1}r^{-1})<\ell(r^{-1})$ or, equivalently, $\ell(rs_{a-1})<\ell(r)$. This means that $r(a-1)>r(a)$. Combining this with the boundedness of $f$, we have $r(a)<r(a-1)<i(a-1)<i(a)\leq r(a)+n$. Reducing modulo $n$, we see the chords $\rho_{a-1} \mapsto \iota_{a-1}$ and $\rho_a \mapsto \iota_a$ are aligned.
	
	$(\Rightarrow):$ Now, suppose the chords $\rho_{a-1} \to \iota_{a-1}$ and $\rho_{a} \to \iota_{a}$ are aligned. If $i(a-1)>i(a)$, then we have $i s_{a-1} \lessdot_R i \leq_R f$, so we assume $i(a-1)<i(a)$. We would like to show that $f=(i s_{a-1})(s_{a-1}r^{-1})$ is length-additive, which is equivalent to $r(a-1)>r(a)$.
	
    By the boundedness of $f$ and $i$, we have the following situation in $\bbz$:
    \begin{center}
    	\includegraphics[width=0.7\textwidth]{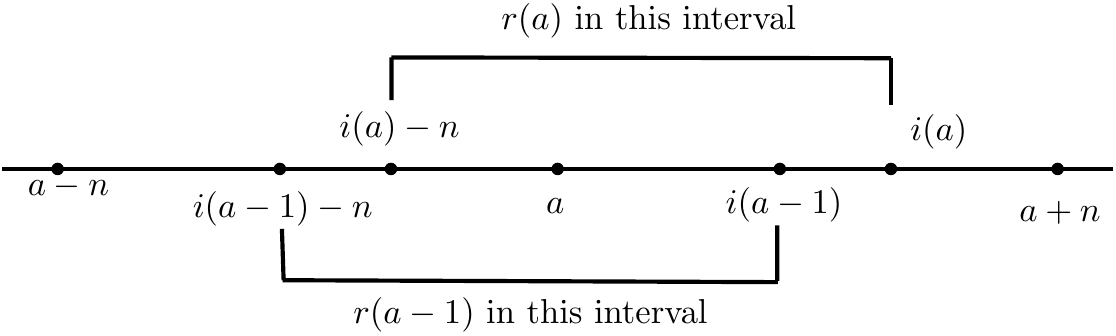}
    \end{center}
    
    Consider the chords in the lower half plane from $r(a-1)$ to $i(a-1)$ and from $r(a)$ to $i(a)$. If the endpoints of these chords are contained in the interval $[i(a)-n, i(a)]$, we can obtain $\rho_{a-1} \mapsto \iota_{a-1}$ and $\rho_a \mapsto \iota_a$ in the disk by gluing $i(a)-n$ and $a$ (see figure below). With this in mind, suppose that $r(a-1)\geq i(a)-n$. The assumption that $\rho_{a-1} \mapsto \iota_{a-1}$ and $\rho_a \mapsto \iota_a$ are aligned forces $r(a-1)>r(a)$ (see figure below).
    
    \begin{center}
    	\includegraphics[width=\textwidth]{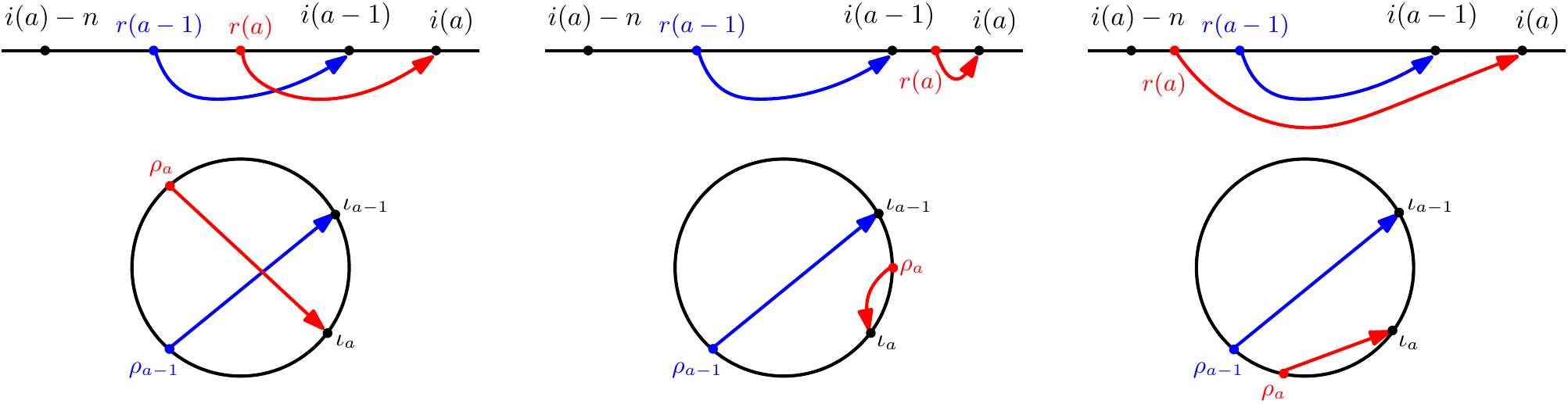}
    \end{center}
    
    Now suppose $r(a-1)<i(a)-n$. We will instead consider the chords in the lower half plane from $r(a-1)+n$ to $i(a-1)$ and from $r(a)$ to $i(a)$; again, gluing $i(a)-n$ and $i(a)$, we obtain the chords $\rho_{a-1} \mapsto \iota_{a-1}$ and $\rho_a \mapsto \iota_a$. Because $\rho_{a-1} \mapsto \iota_{a-1}$ and $\rho_a \mapsto \iota_a$ are aligned, we must have $i(a-1)< r(a-1)+n < r(a)$. Now, consider $f$ written in one-line notation. In positions $r(a-1)$ and $r(a)$, we see the values $i(a-1)$ and $i(a)$, respectively. Because positions $r(a-1)$ and $r(a)$ are more than $n$ apart and $f$ is bounded, the value $r(a)$ occurs somewhere between $i(a-1)$ and $i(a)$ and we see
   \[f= \dots i(a-1) \dots r(a) \dots i(a)\dots.\]
Because $i(a-1)<r(a)<i(a)$, after right-multiplying $f$ by a sequence of length-decreasing simple transpositions, we will always see
\[i(a-1)\dots r(a)\dots i(a)\]
in that order.
 But by assumption, $i$ can be obtained from $f$ by right-multiplication by such a sequence, and the values $i(a-1), i(a)$ are adjacent in $i$. This is a contradiction.
\end{proof}

Combining \cref{prop:monomialtransform,lem:alignedMovesDown}, we obtain the following result on unit necklaces.

\begin{thm}[Unit necklace theorem]\label{thm:rightweak}
Let $\pi, \iota$ be permutations of type $(k, n)$ such that $\iota \ler \pi$. Then the Grassmannlike necklace $\gnec = \gnec_{\bullet,\iota,\pi}$ is of type $(k, n)$ and is a unit necklace in $\tPio_\pi$. 	
	Moreover, $\Delta(\gnec)$ is a basis for the free abelian group $\bbp_\pi \subset \bbc[\tPio_\pi]$. 
\end{thm}

That is, if $\iota \leq_\circ \pi$ and if we set $\rho = \pi^{-1}\iota$ as usual, then we get a pair of permutations $\rho,\pi$ satisfying the $k$-subsets condition (P0) and the units condition (P1).

\begin{rmk}\label{rmk:exhibiting}
\cref{thm:rightweak} provides us with {\sl many} $n$-tuples of Pl\"ucker coordinates which are bases for the abelian group $\bbp_\pi \subset \bbc[\tPio_\pi]$. We get one such $n$-tuple for each element in the lower order ideal beneath $f$ in $({\rm Bound}(k,n),\leq_R)$. In particular, we obtain an explicit construction of {\sl many} Pl\"ucker coordinates which are units in $\bbc[\tPio]$. Any such Pl\"ucker coordinate cannot be a mutable cluster variable in any cluster structure on $\tPio_\pi$. 
\end{rmk}

We have the following corollary of \cref{thm:rightweak}\footnote{This was known, but we do not think it has been stated explicitly previously.}:

\begin{cor}[Source vs. target frozens]\label{eg:groupscoincide}
	Let $\rmci=\rmci_\pi$ be a forward Grassmann necklace and let $\lmci$ be the reverse Grassmann necklace with permutation $\pi^{-1}$. Then $\Delta(\lmci)$ is a basis for $\bbp_\pi$. That is, the group of Laurent monomials in the target frozens $\Delta(\rmci)$ coincides with group of Laurent monomials in the source frozens $\Delta(\lmci)$ inside $\bbc[\tPio]$.
\end{cor}

\begin{proof}
Recall that $\eps_k=\overline{e_k}$ is the Grassmannian permutation $k+1 \cdots n 1 \cdots k$. Because $e_k$ is the minimal element of $\bd(k, n)$, $\eps_k$ is the minimal element of $\leq_\circ$, and in particular $\eps_k\ler \pi$. The Grassmannlike necklace $\gnec_{\bullet,\eps_k,\pi}$ is the shifted reverse Grassmann necklace $\lmci_\pi[k]$. By \cref{thm:rightweak}, we conclude that $\lmci$ is a unit necklace and that $\Delta(\lmci)$ is a basis for $\bbp_\pi$.
\end{proof}

\subsection{The seed size condition and the main theorem}\label{subsec:correctlength} 
Continuing our general setup, consider a permutation $\pi$ of type $(k,n)$ and a permutation $\rho$ with the property that $\pi \rho \ler \pi$. Consider a relabeled plabic graph~$G^\rho$ with trip permutation $\pi$. The boundary face target labels are the Grassmannlike necklace $\gnec=\gnec_{\rho,\bullet,\pi}$ whose corresponding Pl\"ucker coordinates $\Delta(\gnec) \subset \mathbb{C}[\tPio_\pi]$ are a basis for $\mathbb{P}_\pi$ by the unit necklace theorem. Thus the conditions (P0) and (P1) are satisfied. In this context, it turns out that the seed size condition (P2) is also a sufficient condition for $\Sigma_{G^\rho}^T$ to give a cluster structure on $\tPio_\pi$. In this section, we discuss the seed size condition and then state our main theorem characterizing when $\Sigma_{G^\rho}^T$ gives a cluster structure.

Now, let $\mu:= \rho^{-1}\pi\rho= \iota^{-1}\pi \iota$ be the underlying permutation of $\gnec$ (i.e., the trip permutation of the underlying graph $G$). It also has type $(k,n)$ by \cref{lem:gnecPermuted}. Let $m, f \in {\rm Bound}(k,n)$ be the lifts of $\mu$ and $\pi$. The seed size condition (P2) is that these two bounded affine permutations $m$ and $f$ have the same length.

To more easily analyze this condition, we give a more direct formula for this lift~$m$.

\begin{lem}[Lifting $\mu$] \label{lem:conjugationLift}
	Let $\pi$ and  $\iota$ be permutations of type $(k, n)$ with $\iota \ler \pi$.  Set $\mu:=\iota^{-1} \pi \iota$. Let $f, i, m \in \bd(k,n)$ be the respective lifts of $\pi$, $\iota$, and  $\mu$. Then $m=i^{-1}fi$.
\end{lem}
\begin{proof}
	Because reducing modulo $n$ is a group homomorphism, it's clear that $\overline{i^{-1}fi}=\mu$. It is also clear that $i^{-1}fi \in \tilde{S}_n^k$, so all that remains to prove is that $i^{-1}fi$ is bounded.
	
	Let $r :=f^{-1}i \in \tilde{S}^0_n$ so that $i^{-1}fi=r^{-1}i$.
	
	We start with the two line notation for $f$; that is, the numerator is $f$ and the denominator is the identity permutation $e_0 \in \tilde{S}^0_n$. Since $i \leq_R f$, we can obtain $i$ in the numerator of this array by repeatedly swapping adjacent numbers $b>a$ in the numerator. We obtain $r$ in the denominator by applying the same sequence of swaps.
	
	Focusing on any particular value $x \in \bbz$, using 
	boundedness of $f$, before performing any swaps we see 
	$$\begin{matrix}
	\cdots & \cdots & x & \cdots & \cdots \\
	\cdots & x-n & \cdots & \cdots & x
	\end{matrix} $$
	appearing in $f$ in that order (specifically, $x$ in the numerator appears strictly left of $x$ and weakly right of $x-n$ in the denominator). Also, note that $f(x-n) \leq x < f(x)$. Thus, we will never swap $f(x)$ and $x$ in the numerator, or $x$ and $f(x)$. So the relative order of these three symbols ($x$ in the numerator, $x, x-n$ in the denominator) is preserved for all $i$ and $r$.
	
	Let $a$ be given. 
	Let $x = i(a)$ be the value in the numerator at position $a$. By the argument above, the $x$ in the denominator is in a position strictly right of position $a$. The position of $x$ in the denominator is $r^{-1}(x)$, so we have $a<r^{-1}i(a)$. And the 
	$x-n$ in the denominator is weakly to the left of $a$, so we have 
	$r^{-1}(i(a)-n)=r^{-1}i(a)-n \leq a$. So $r^{-1}i$ is bounded.
\end{proof}

\begin{rmk}\label{rmk:mshorterf}
In the situation of \cref{lem:conjugationLift}, we have 
$$\ell(m) = \ell(i^{-1}f i) \leq \ell(i^{-1}f)+\ell(i) = \ell(f) $$
where in the last step we have used $i \leq_R f$. Our main theorem characterize the cases in which the inequality $\ell(m) 
\leq \ell(f)$ is in fact an equality. In our running setup, the inequality $\ell(m) 
\leq \ell(f)$ means that the relabeled plabic graphs $G^{\rho}$ we consider always have {\sl at least} $\dim \tPio_\pi$ many faces. 
\end{rmk}

Before stating our main theorem characterizing when the seed size condition (P2) holds, we recall the following result of Farber and Galashin.

\begin{thm}[{\cite[Theorem 6.3]{FG}}]\label{thm:FG}
	Let $G^\rho$ be a relabeled plabic graph with trip permutation $\pi$, and let $\gnec=\gnec_{\rho, \bullet, \pi}$ be the Grassmannlike necklace of target labels of boundary faces of $G^\rho$. 
	If $\gnec$ is weakly separated, then so is the target collection $\targetF(G^\rho)$. 
\end{thm}

Recall that $\targetF(G^\rho) = \rho(\targetF(G))$. Farber and Galashin worked directly with the right hand side of this equality, i.e. they did not frame their results in terms of relabeled plabic graphs. 

\begin{rmk}\label{rmk:FGinside} In the setting of \cref{thm:FG}, let $\mu$ be the trip permutation of the underlying graph $G$. Recall that the permutation $\rho$ sends 
$\targetF(G) \to \targetF(G^\rho)$ and sends $\rmci_\mu$ to $\gnec$. Farber and Galashin prove moreover that $\rho$ induces a bijection from
	\[\mathscr{C}= \{ \mcc  \colon \rmci_{\mu} \subset \mcc \subset \mcm_\mu \text{ and } \mcc \text{ is weakly separated and maximal by inclusion}\}
	\]
	to
		\[\mathscr{C}'= \{ \mcc' : \gnec \subset \mcc \subset D^{\rm in}_\gnec\text{ and } \mcc \text{ weakly separated, maximal by inclusion}\}
	\]
where $D^{\rm in}_\gnec \subset \binom{[n]}k$ denotes the $k$-subsets lying inside the closed curve defined by $\gnec$ in the plabic tiling (cf.~\cite[Definition 4.3]{FG}). We expect that $\mathscr{C}'$ consists of all maximal weakly separated collections $\mcc'$ with $\gnec \subset \mcc' \subset \mcm_\pi$.
\end{rmk}

With these preparations in hand, we restate our main theorem:
\begin{thm}\label{thm:main}
	Let $G^\rho$ be a relabeled plabic graph with trip permutation $\pi$. Let $\mu$ be the trip permutation of the underlying plabic graph $G$. Let $\gnec$ be the Grassmannlike necklace given by the target labels of boundary faces of $G^\rho$.
	
	Suppose that $\pi \rho \ler \pi$. 
	Then the following are equivalent:
	\begin{enumerate}
\item The seed $\Sigma_{G ^\rho}^T$ determines a cluster structure on $\tPio_\pi$.
\item The number of faces of $G^\rho$ is $\dim \tPio_\pi$. Equivalently, $\dim \tPio_\mu = \dim \tPio_\pi$.
\item The Grassmannlike necklace $\gnec_{\rho, \pi \rho, \pi}$ (equivalently, the target collection $\targetF(G^\rho)$) is a weakly separated collection. 
\item The open positroid varieties $\tPio_\pi$ and $\tPio_{\mu}$ are isomorphic.  
\end{enumerate}
\end{thm}

\begin{proof}[Outline of proof]
The equivalence of the two formulations of (2) is Equation \eqref{eq:lengthiscodim}. The equivalence of the two formulations of (3) is \cref{thm:FG}. 

We give a direct combinatorial proof that the Coxeter-theoretic condition (2) and the combinatorial condition (3) are equivalent in \cref{secn:conjugationproof}. 
The algebro-geometric equalities (4) and  (1) both clearly imply the dimension condition (2). 

Condition (1) implies (4): condition (1) says that $\bbc[\tPio_\pi] = \mca(\Sigma^T_{G^\rho})$. By \cref{thm:GL}, $\bbc[\tPio_\mu] = \mca(\Sigma^S_G)$. Since these seeds have the same quiver (including the frozen vertices!), the cluster algebras are isomorphic as rings, so (4) holds. 

We show that (2) implies (3) in \cref{thm:twistsareinverses} by identifying an explicit ``Grassmannlike'' {\sl twist isomorphism} $\tPio_\pi \to \tPio_\mu$ generalizing the twist automorphisms of Muller and Speyer \cite{MSTwist}. 

To show that (2) implies (1), we first establish that $\Sigma_{G^\rho}^T$ is a seed in $\bbc(\tPio_\pi)$ in \cref{prop:faceLabelsTorus} and then deduce that $\mca(\Sigma_{G^\rho}^T)= \bbc[\tPio_\pi]$ in \cref{thm:genPlabicClusterStruc}.
\end{proof}

When $\rho$ is the identity permutation, \cref{thm:main} says that target seeds $\Sigma^T_G$ determine cluster structures in $\tPio_\pi$ as claimed in \cref{rmk:targetpreference}.

\begin{rmk}\label{rmk:whenapplies}
When $\gnec$ is not weakly separated, we see no good way of creating a seed whose frozen variables are $\gnec$ and whose mutable variables are Pl\"ucker coordinates. Perhaps there is some kind of construction of seeds with frozen variables $\gnec$ and with non-Pl\"ucker mutable variables. 
\end{rmk}

\begin{example}\label{eg:differentbases}
Consider the bottom center graph $G^\rho$ from \cref{fig:fourplabics}. Its boundary face labels are the Grassmannlike necklace $\gnec_2$ from \cref{ex:easyToggles}. The underlying plabic graph $G$ has trip permutation $\mu=564123$ with forward Grassmann necklace 
$$\rmci_\mu=(123, 235, 356, 456, 156, 126).$$
By Oh's theorem, the corresponding positroid is $\mcm_\mu = \binom {[6]} 3 \setminus \{134,234,345,346\}$, with $16$ bases in total. The open positroid variety $\tPio_\mu$ consists exactly of points where columns 3 and 4 are parallel in any matrix representative and the Pl\"ucker coordinates $\Delta(\rmci_{\mu})$ are nonvanishing. 

The trip permutation of $\gnec_2$ is the permutation $\pi = 465213$ whose forward Grassmann necklace is given in \eqref{eg:pi}. Using Oh's theorem again, the corresponding positroid is $\mcm_\pi = \binom{[6]}3 \setminus \{345,156\}$ which has 18 bases. 

There does not seem to be any relationship between the positroids $\mcm_\mu$ and $\mcm_\pi$. On the other hand, one can easily check that the necklace $\gnec_2$ is weakly separated. The removal permutation is $s_4$ and one can check that $\pi s_4 \leq_\circ \pi$. Thus the implication (2)$\implies$(3) from \cref{thm:main} says that the open positroid varieties $\tPio_\pi$ and $\tPio_\mu$  are isomorphic. 
\end{example}

\begin{example}\label{eg:someeasytoggles2}
Continuing the previous example, the lift of the trip permutation $\pi$ of $\gnec_2$ is $f=[4, 6, 5, 8, 7, 9]$ in window notation. The lift of the underlying permutation $\mu$ is $m = [5,6,4,7,8,9]$. One can check that these two affine permutations have the same length (each has two inversions), in agreement with the implication (2)$\implies$(1) from \cref{thm:main}. 
\end{example}

Each of the red necklaces in \cref{fig:notreachable} is an example of a unit necklace (for the open positroid variety indexed by $\pi = 5761432$) 
which is not weakly separated. 

\begin{rmk}
Let $f \in \bd(k,n)$ with corresponding $\pi = \overline{f} \in S_n$. Notice that $e^{-1}_kfe_k \in {\rm Bound}(k,n)$ 
has the same length as $f$ because right or left multiplication by the elements $\{e_b \colon b \in \bbz\}$ does not affect length. Thus, the implication (1)$\implies$(2) from \cref{thm:main} says that the Grassmannlike necklace $\gnec_{\bullet, \eps_k, \pi}$ is weakly separated. But $\gnec_{\bullet, \eps_k, \pi}$ is just the shifted reverse Grassmann necklace $\lmci_{\pi}[k]$ (cf. \cref{ex:shiftedReverseGN}). So this special case of \cref{thm:main} is the well-known statement that reverse Grassmann necklaces are weakly separated.
\end{rmk}

%% file: twistsecnB.tex
We explain in this section how, in the setting of Theorem~\ref{thm:rightweak}, a Grassmannlike necklace $\gnec$ encodes a ``Grassmannlike twist map''
{\sl between} two open positroid varieties. These varieties need not be of the same dimension. The definition is an extension of the twist {\sl automorphisms} given by Muller and Speyer \cite{MSTwist}. 
We show that if $\gnec$ satisfies \cref{thm:main} (2), then our twist map is an isomorphism. Moreover, this isomorphism fits into a commutative diagram which is in the spirit of the chamber ansatz \cite{BFZ} and of the main commutative diagram from \cite{MSTwist}. Using this commutative diagram, we reach our conclusions concerning cluster structures and total positivity. 

\subsection{Grassmannlike twist maps}
Endow $\bbc^k$ with an inner product $\langle \cdot , \cdot\rangle$, and let $\Mato(k, n)$ denote the space of full-rank $k \times n$ matrices. Let $\gnec = (I_1,\dots,I_n)$ be a Grassmannlike necklace. 
We use the notation 
$$D(\gnec) = \{M \in \Mato(k,n)\colon \Delta_I(M) \neq 0 , \text{ for all $I \in \gnec$}\}$$ to denote the Zariski-open subset of $\Mato(k,n)$ defined by the non-vanishing of Pl\"ucker coordinates $\Delta(\gnec)$. We use the same notation $D(\gnec) \subset \Gr(k,n)$ to denote the image of these matrices in the Grassmannian. 

For $M\in \Mato(k, n)$, we use $M_a$ to denote the $a$th column of $M$. 
\begin{defn}[$\gnec$-twist maps]\label{defn:gentwist}
Let $\gnec = (I_1,\dots,I_n)$ be a Grassmannlike necklace of type $(k,n)$ with removal permutation $\rho$ and insertion permutation $\iota$. Let $M \in D(\gnec) \subset \Mato(k, n)$ have columns $M_1,\dots,M_n$. Then the {\sl right $\gnec$-twist of $M$} is the matrix $\rt_\gnec(M) \in \Mat^\circ(k,n)$ whose $a$th column $\rt_\gnec(M)_a$ is the unique vector such that for all $b \in I_a$,

\begin{equation}\label{eq:twistdefn}
\langle (\rt_\gnec(M))_a,M_b \rangle  = \begin{cases}
1 & \text{ if } \rho(a)=b\\
0 & \text{ else}.
\end{cases}
\end{equation}

Similarly, the \emph{left $\gnec$-twist of $M$} is the matrix $\lt_\gnec(M)$ such that for all $a$, $\lt_\gnec(M)_a$ is the unique vector such that for all $b \in I_{a+1}$,  

\begin{equation}\label{eq:lefttwistdefn}
\langle (\lt_\gnec(M))_a,M_b \rangle  = \begin{cases}
1 & \text{ if } \iota(a)=b\\
0 & \text{ else}.
\end{cases}
\end{equation}
 
Both \cref{eq:twistdefn,eq:lefttwistdefn} do define unique vectors, since by assumption the columns of $M$ indexed by $I_a$ form a basis of $\bbc^k$.

\end{defn}

An identical argument as the proof of \cite[Prop. 6.1]{MSTwist} shows that the $\gnec$-twist maps descend from ${\rm Mat}^\circ(k,n)$ to $\Gr(k,n)$. 

For a permutation $\pi$, we abbreviate $\rt_\pi := \rt_{\rmci_\pi}$ where the latter denotes the right $\gnec$-twist in the case $\gnec = \rmci_\pi$ is a forward Grassmann necklace. Similarly we set $\lt_\pi := \lt_{\lmci_\pi}$.

\begin{rmk}
Marsh and Scott defined a version of a twist map for the top-dimensional open positroid variety \cite{MarshScott}. Muller and Speyer slightly modified this definition and extended it to all positroid varieties. More specifically, they introduced right and left twist maps $\rt$ and $\lt$, which they viewed as piecewise-continuous maps on $\Gr(k,n)$ whose domains of continuity are the open positroid varieties. 
The right and left twist maps are mutual inverses. 

The restriction of the Muller-Speyer right twist map to a given open positroid variety $\tPio_\pi$ is the map $\rt_\pi$ we have defined above, and similarly for the left twist\footnote{Our definition looks slightly different because of our slightly different indexing conventions for reverse Grassmann necklaces.}.  Unlike Muller and Speyer, however, we view the maps $\rt_\pi$ and $\lt_\pi$ as rational maps on $\tGr(k,n)$ whose domains of definition are $D(\rmci_\pi)$ and $D(\lmci_\pi)$ respectively. These maps are {\sl not} mutual inverses when viewed on the larger domain $D(\rmci_\pi) \cap D(\lmci_\pi)$ on which they are defined. Despite this, a key technical step in our forthcoming arguments is to show that 
$ \Delta_I(\lt_\pi \circ \rt_\pi(x))=\Delta_I(x)$ for specific $I \in \binom{[n]}k$ which happen to arise as face labels of certain plabic graphs, see \cref{lem:twistisalmostinverse}. 
\end{rmk}

\begin{example}\label{eg:thetwotwists}
Let $\gnec= \gnec_2$ be the necklace from \cref{ex:easyToggles} and \cref{eg:differentbases}. Let $M_1,\dots,M_6 \in \mathbb{C}^3$ be columns of a $3 \times 6$ matrix representative for a point in $\tGr(3,6)$. Let $\times \colon \mathbb{C}^3 \otimes \mathbb{C}^3 \to \mathbb{C}^3$ denote cross product of vectors. Then 
\begin{equation}
\rt_\gnec(M) = \begin{bmatrix}
	\frac{M_2 \times M_3}{\Delta_{123}(M)} &
	\frac{M_3 \times M_4}{\Delta_{234}(M)} &
	\frac{M_4 \times M_6}{\Delta_{346}(M)}&
	\frac{M_4 \times M_6}{\Delta_{456}(M)}&
	\frac{M_1 \times M_6}{\Delta_{146}(M)}&
	\frac{M_1 \times M_2}{\Delta_{126}(M)}
\end{bmatrix}
\end{equation}
while 
\begin{equation}
\lt_\gnec(M) = \begin{bmatrix}
	\frac{M_2 \times M_3}{\Delta_{234}(M)}&
	\frac{M_3 \times M_4}{\Delta_{346}(M)}&
	\frac{M_4 \times M_6}{\Delta_{456}(M)}&
	\frac{M_4 \times M_6}{\Delta_{146}(M)}&
	\frac{M_1 \times M_6}{\Delta_{126}(M)}&
	\frac{M_1 \times M_2}{\Delta_{123}(M)}
\end{bmatrix}.
\end{equation}
\end{example}

Our next result describes the image of the right and left twist maps $\rt_\pi$ and $\lt_\pi$. It is based on \cite[Prop. 6.6]{MSTwist}.
\begin{thm}[Based on Muller-Speyer]\label{thm:MStwist}
Let $\pi$ be a permutation of type $(k, n)$. Then the right twist map $\rt_\pi$ descends to a regular map $D(\rmci_\pi) \twoheadrightarrow \tPio_\pi$. Similarly, the left twist map $\lt_\pi$ descends to a regular map $D(\lmci_\pi) \twoheadrightarrow \tPio_\pi$.
\end{thm}

\begin{proof}
	Let $\rmci=\rmci_\pi$ and $\lmci=\lmci_\pi$.
	
	We already know that $\rt_\pi: D(\rmci)\to \widetilde{\Gr(k, n)}$ is a regular map.
	
To show that $\rt_\pi$ lands in $\tPio_\pi$ we need to show that all coordinates in $\Delta(\binom{[n]}k \setminus \mcm_\pi)$ vanish on the image and show also that all coordinates in $\Delta(\rmci)$ do not vanish. We have the same determinantal formula \cite[Lemma 6.5]{MSTwist} describing Pl\"ucker coordinates of $\rt_\pi(x)$. Using this formula we see 
that $\Delta(\rmci)$ is non-vanishing on the image (cf.~\cite[Equation (9)]{MSTwist}). 

We use the same formula to see that $\Delta_J$ vanishes on the the image of $\rt_\pi$ when $J \notin \mcm_\pi$. Let us clarify a confusing point. In the proof of \cite[Proposition 6.6]{MSTwist}, Muller and Speyer argue an implication $i_d \in \rightI_{j_c}$, by taking a representative matrix~$A$ of $x \in \tPio$ and making an argument about linear independence of columns of $A$. The assumption $x \in \tPio$ is more restrictive than the assumption $x \in D(\rmci)$. However, the implication $i_d \in \rightI_{j_c}$ is a property of the necklace $\rmci_\pi$ and the positroid $\mcm_\pi$, i.e. it is not a property of matrix representative~$A$.  
The rest of the proof of \cite[Proposition 6.6]{MSTwist} proceeds without change. 

Finally, the map is surjective because it is an automorphism when restricted to 
$\tPio \subset D(\rmci)$. 

The claims about $\lt_\pi$ follow by a symmetric argument.
\end{proof}

A permutation $\rho \in S_n$ determines an automorphism of $\Mato(k,n)$, and likewise $\Gr(k,n)$, by column permutation:
$$ \begin{bmatrix}
A_1 & \cdots & A_n
\end{bmatrix}
\mapsto 
 \begin{bmatrix}
A_{\rho(1)} & \cdots & A_{\rho(n)}
\end{bmatrix}.
$$
We denote these automorphisms by the same symbol $\rho$.  
This automorphism acts on Pl\"ucker coordinates via $\Delta_I(\rho(X)) = \pm\Delta_{\rho(I)}(X)$ where the extra sign $\pm$ is the sign associated with sorting the values $\rho(i_1),\dots,\rho(i_k)$.

The $\gnec$-twist maps can be described completely in terms of the right and left twists $\rt_\pi$ and $\lt_\pi$ together with column permutations:
\begin{lem}[$\gnec$-twists and column permutation]\label{lem:gentwist}	Let $\pi$ be a permutation of type $(k, n)$, and consider the Grassmannlike necklace $\gnec=\gnec_{\rho, \iota, \pi}$. Let $\mu:= \rho^{-1}\iota$ be the underlying permutation.%

We have 	
$$\rt_\gnec = \rt_\mu \circ \rho \text{ and }  \lt_\gnec = \lt_\mu \circ \iota$$ 
as rational maps on ${\rm Mat}(k,n)$ or $\widetilde{\Gr(k,n)}$. In particular, the image of $\rt_\gnec$ and $\lt_\gnec$ is contained in $\tPio_\mu$.
\end{lem}

As an example, the matrices $\rt_\gnec(M)$ and $\lt_\gnec(M)$ computed in \cref{eg:thetwotwists} are elements of $\Pio_\mu$ (with $\mu$ as in \cref{eg:differentbases}): columns $3$ and $4$ are parallel and the minors $\Delta(\rmci_{\mu})$ are nonzero.

\begin{proof}
	First we discuss the equality $\rt_\gnec = \rt_\mu \circ \rho$. Let $\rmci_\mu = (\rightI_1,\dots,\rightI_n)$, and
	$\gnec = (I_1,\dots,I_n)$. By \cref{lem:gnecPermuted}, $\rho(\rightI_a) = I_a$.
	
	The domain of $\rt_\mu \circ \rho$ is $\rho^{-1} (D(\rmci_\mu))$. We have that $\Delta_{\rightI_a}(\rho(x)) = \pm \Delta_{\rho(\rightI_a)}(x)= \pm\Delta_{I_a}(x)$ for any $x \in \Gr(k,n)$. Thus 
	$\rho^{-1}(D(\rmci_\mu)) = D(\gnec)$, so $\rt_\mu \circ \rho$ and $\tau_\gnec$ have the same domain of definition. 
	
	Let $x \in D(\gnec)$ be represented by the matrix $M \in \Mat(k, n)$, so $\rho(x)$ is represented by the matrix $[M_{\rho(1)}\cdots M_{\rho(n)}]$.
	
	Let $\delta$ denote the Kronecker delta. The definition of $\rt_\mu\circ\rho(x)$ is that 
	\begin{align}
	\langle (\rt_\mu\circ\rho(M))_a,\rho(M)_b \rangle  &= \delta_{a,b} \text{ for all } b \in \rightI_a, \text{ i.e. } \\
	\langle (\rt_\mu\circ\rho(M))_a,M_{\rho(b)} \rangle  &= \delta_{a,b} \text{ for all } b \in \rightI_a, \text{ i.e. } \\
	\langle (\rt_\mu\circ\rho(M))_a,M_{s} \rangle  &= \delta_{a,\rho^{-1}(s)} \text{ for all } s \in \rho(\rightI_a), \text{ i.e. } \\
	\langle (\rt_\mu\circ\rho(M))_a,M_{s} \rangle  &= \delta_{a,\rho^{-1}(s)} \text{ for all } s \in I_a, \text{ i.e. } \\
	\langle (\rt_\mu\circ\rho(M))_a,M_{s} \rangle  &= \delta_{\rho(a),s} \text{ for all } s \in I_a, \label{eq:preindex}
	\end{align}
	so that the condition defining $(\rt_\mu\circ\rho(M))_a$ is exactly the condition \eqref{defn:gentwist} defining 
	$\rt_\gnec(M)_a$. So $\tau_\gnec$ and $\rt_\mu \circ \rho$ are the same map.
	
	The second equality holds by a similar argument, noting that $\iota(\leftI_{a+1})=I_{a+1}$ by \cref{lem:gnecPermuted}.
\end{proof}

To show that our twist maps are invertible, we introduce the following notion.

\begin{defn}[Dual necklace]
	Let $\gnec=\gnec_{\rho, \iota, \pi}$ be a Grassmannlike necklace. The \emph{dual necklace} $\gnec^*$ is the Grassmannlike necklace with removal permutation $\iota^{-1}$ and with insertion permutation $\rho^{-1}$. 
\end{defn}

\begin{example}
	The dual necklace of a Grassmann necklace $\rmci_\pi$ is the reverse Grassmann necklace $\lmci_\pi$. The dual of the necklace $\gnec_2$ from \cref{ex:easyToggles} and \cref{eg:differentbases} is 
	$$\gnec_2^* = 	(456, 156, 126, 123, 235,245).$$ 
\end{example}

Clearly taking the dual is an involutive operation on Grassmannlike necklaces. One can check that $\gnec^*$ has trip permutation $\rho^{-1}\iota$ and has underlying permutation $\iota\rho^{-1}$, i.e. the duality operation interchanges the trip and underlying permutations. In particular $\gnec^*$ has the same type as the trip permutation $\pi$ of $\gnec$. Using \cref{lem:gentwist}, we see that the images of the $\gnec^*$-twists are contained in $\tPio_\pi$.

We can now invert the $\gnec$-twist maps.

\begin{thm}[Inverting the $\gnec$-twist map]\label{thm:twistsareinverses}
	Let $\gnec=\gnec_{\rho,\iota,\pi}$ be a Grassmannlike necklace with $\iota 
	\ler \pi$. Let $\gnec^*$ be the dual necklace with trip permutation $\mu = \rho^{-1}\iota$. 
	
	Suppose that $\dim(\tPio_\pi)=\dim(\tPio_\mu)$. Then $\rt_\gnec \colon \tPio_\pi \to \tPio_\mu$ is an isomorphism of open positroid varieties with inverse $\lt_{\gnec^*}: \tPio_\mu \to \tPio_\pi$. 
\end{thm}

One has similarly that $\lt_\gnec \colon \tPio_\pi \to \tPio_\mu$ is an isomorphism of open positroid varieties with inverse $\rt_{\gnec^*}: \tPio_\mu \to \tPio_\pi$.

\begin{rmk}
	In the situation of \cref{thm:twistsareinverses}, let $f, m \in \bd(k, n)$ be the respective lifts of $\pi$ and $\mu$. We always have $\ell(m) \leq \ell(f)$ (c.f. \cref{rmk:mshorterf}), and hence $\dim \tPio_\mu \geq \dim \tPio_\pi$, but this inequality may be strict. When this inequality is strict, there is no chance for the $\gnec$-twist to determine an isomorphism $\tPio_\pi \to \tPio_\mu$. \cref{thm:twistsareinverses} asserts that this dimension constraint is the only obstruction to $\gnec$-twists being isomorphisms.
\end{rmk}

The invertability of $\gnec$-twist maps in \cref{thm:twistsareinverses} will be deduced from the following proposition. 
Recall the permutation $\eps_r\in S_n$ sends $a \mapsto a+r$ (taken modulo $n$). The awkward appearance of $\eps_r$ below is because, if $\gnec=\gnec_{\rho, \iota, \pi}$ satisfies $\iota \ler \pi$, the dual necklace $\gnec^*=\gnec_{\iota^{-1}, \rho^{-1}, \mu}$ will not satisfy $\rho^{-1} \ler \mu$ on the nose.

\begin{prop} \label{prop:twistComposeId}
Let $\pi$ be a permutation of type $(k, n)$, and let $\gnec=\gnec_{\rho, \iota, \pi}$ be a Grassmannlike necklace with underlying permutation $\mu:= \rho^{-1}\iota$ and dual $\gnec^*$. Suppose that for some $r \in [n]$, we have $\iota \circ \eps_r \ler \pi$, and further that $\gnec^*$ is a unit necklace in $\tPio_\mu$.

Then on $\tPio_\pi \subset D(\gnec)$, both compositions
 $$\lt_{\gnec^*} \circ \rt_\gnec \text{ and } \rt_{\gnec^*} \circ \lt_\gnec
 $$  
 are the identity map ${id}_{\tPio_\pi}$.
\end{prop}

\begin{proof}	
	Suppose $\gnec=(I_1, \dots, I_n)$. The Grassmannlike necklace with insertion permutation $\iota\circ \eps_r$ and removal permutation $\rho\circ \eps_r$ is a rotation $\gnec[r] = ( I_{r+1}, \dots, I_n, I_1, \dots, I_{r})$ of $\gnec$. The two necklaces have the same trip permutation, $\pi$. By \cref{thm:rightweak}, since $\iota \circ \eps^r \ler \pi$, $\gnec$ is a unit necklace in $\tPio_\pi$ (since this is true of the shifted necklace).

	This means that $\tPio_\pi$ is indeed a subset of $D(\gnec)$. By \cref{lem:gentwist}, the image of $\rt_\gnec$ and $\lt_\gnec$ are both contained in $\tPio_\mu$. By assumption, $\gnec^*$ is a unit necklace, so $\rt_{\gnec^*}$ and $\lt_{\gnec^*}$ are defined on $\tPio_\mu$. In particular, the compositions are well-defined.
	
	We focus on the composition $\lt_{\gnec^*} \circ \rt_\gnec$. Note that by \cref{lem:gentwist} and the definition of $\gnec^*$, $\lt_{\gnec^*}= \lt_\pi \circ \rho^{-1}$. So the image of the composition is contained in $\tPio_\pi$.

Let $M \in \tPio_\pi$ be given. We would like to show that $M$ is the image of $\rho^{-1} \circ \rt_\mu\circ\rho(M)$ under $\lt_\pi$.

Rewriting the equality \eqref{eq:preindex} in terms of $b = \rho(a)$, we have 
\begin{align}
\langle (\rho^{-1} \circ \rt_\mu\circ\rho(M))_b,M_{s} \rangle  &= \delta_{b,s} \text{ for all } s \in I_{\rho^{-1}(b)} \textnormal{ and for all $b$.} \label{eq:oneside}
\end{align}

We need to show that $M$ satisfies the defining equalities of $\lt_\pi(\rho^{-1} \circ \rt_\mu\circ\rho(M))$ given in \cite[Section 6.1]{MSTwist}, namely the equalities
\begin{align}
\langle (\rho^{-1} \circ \rt_\mu\circ\rho(M))_b,M_{s} \rangle  &= \delta_{b,s} \text{ for all } b \in \leftI_{s+1} \textnormal{ and for all $s$.} \label{eq:inversecheck}
\end{align}

Noting that $b \in I_{\rho^{-1}(b)}$ from the necklace property, we can set $s = b$ in \eqref{eq:oneside} and conclude that \eqref{eq:inversecheck} holds when $b=s$. \cref{eq:oneside} also implies that \eqref{eq:inversecheck} holds for $b \in \leftI_{s+1} \setminus s$ when $s \in I_{\rho^{-1}(b)}$.

It remains to show that when $b \in \leftI_{s+1} \setminus s$ and $s \notin I_{\rho^{-1}(b)}$, then $(\rho^{-1} \circ \rt_\mu\circ\rho(M))_b$ is perpendicular to $M_s$. By \eqref{eq:oneside}, it would suffice to show that $M_s$ is in the span of $\{M_a \colon a \in I_{\rho^{-1}(b)} \setminus b\}$. 

To show that this is true, we apply \cref{lem:notinpositroid} to $I_{\rho^{-1}(b)}$ with $z= b$ and $y = s$. Note that we can apply this lemma to $\gnec$ since it holds for $\gnec[r]$, cf.~\cref{rmk:conclusionstillholds}. The hypothesis of the lemma requires that that we show $s<_b \pi(b)$. Since $b \in \leftI_s$, 
we have that $M_b \notin {\rm span}(M_{b+1},\dots,M_s)$. On the other hand, from the definition of $\pi(b)$, we have $M_b \in {\rm span}(M_{b+1},\dots,M_{\pi(b)})$. It follows that $s <_b \pi(b)$. So by \cref{lem:notinpositroid}, we have   
$I_{\rho^{-1}(b)}\setminus b\cup s \notin \mcm_\pi$, or in other words, $M_s$ is in the span of $\{M_a \colon a \in I_{\rho^{-1}(b)} \setminus b\}$ (since these vectors are independent).

A symmetric argument shows that the second composition is the identity on $\tPio_\pi$. 

\end{proof}

\begin{proof}[Proof of \cref{thm:twistsareinverses}]
	Let $f, m, i \in \bd(k, n)$ be the lifts of $\pi, \mu$ and $\iota$, respectively, and set $r:=f^{-1}i \in \tilde{S}^0_n$. Our two assumptions imply that $f=ir^{-1}$ is length-additive and that $\ell(m)=\ell(f)$. By \cref{lem:conjugationLift}, $m=i^{-1}fi=r^{-1}i$. The assumption that $\ell(m)=\ell(f)$ implies that $m=r^{-1}i$ is also length-additive.

	We would like to apply \cref{prop:twistComposeId}, so first we need to show that $\gnec^*$ is a unit necklace in $\tPio_\mu$. We consider instead the shifted Grassmannlike necklace $\mathcal{L}$ with insertion permutation $\rho^{-1} \circ \eps_k$ and trip permutation $\mu$. It suffices to show that $\mathcal{L}$ is a unit necklace since $\gnec^*=(L_{n-k+1}, \dots, L_n, L_1, \dots, L_{n-k})$.
	
The lift of the permutation $\rho^{-1} \circ \eps_k$ is $r^{-1}e_k$. Note that $m = (r^{-1}e_k)(e_k^{-1}i)$. By the unit necklace theorem applied to $\gnec^*$ and $\rmci_\mu$, it suffices to show the length-addivitity of this factorization. And indeed, using our two assumptions, we have that
$$\ell(m) = \ell(f) = \ell(i)+\ell(r^{-1}) = \ell(e_k^{-1}i)+\ell(r^{-1}e_k),$$
so that $\mathcal{L}$ hence $\gnec^*$ is indeed a unit necklace for $\tPio_\mu$.
	
	The statements now follow immediately from applying \cref{prop:twistComposeId} to the pair of necklaces $\gnec, \gnec^*$ and to $\gnec^*,( \gnec^*)^*=\gnec.$
		
\end{proof}

\subsection{Inverting boundary measurements}
In this section, we use twist maps along necklaces to deduce that $\Sigma_{G^\rho}^T$ gives a cluster structure on $\tPio_\pi$, relying on the fact that this is true of the source structure as proved in \cite{GalashinLam}.

The following observation is used several times in this section. Consider $f \in {\rm Bound}(k,n)$ with $\overline{f} = \pi$. Consider $a<b \in \bbz$ such that 
$\ell(t_{ab}f) < \ell(f)$, and let $a',b' \in [n]$ be their reductions mod $n$.

\begin{lem}[{\cite[Lemma 4.5]{MSTwist}}]\label{lem:aimpliesb}
With $f,\pi,a',b'$ as above, let  $G$ be a reduced plabic graph with trip permutation $\pi$ and let $I \in \targetF(G)$ be a target label. If $b' \in I$, then $a' \in I$ also. \end{lem}

We can partially order $[n]$ according to whether the conclusion of \cref{lem:aimpliesb} holds. This partial order shows up in several of our proofs in this section.

By \cref{lem:gentwist}, twist maps along necklaces involve column permutation, which introduce unwanted signs in Pl\"ucker coordinates. Our next lemma allows us to compensate for these signs in our constructions. 
\begin{lem}[Taking care of signs]\label{lem:signssigns}
	Let $G^\rho$ be a relabeled plabic graph as in the statement of \cref{thm:main}(2). Then there exists an involutive automorphism $\underline{\eps} \in \Aut(\Gr(k,n))$ with the property that 
	for any $I \in \targetF(G)$ and $y \in \Gr(k,n)$ we have 
	\begin{equation}\label{eq:signssigns}
	\Delta_{\rho(I)}(y) = \Delta_{I}(\rho(\underline{\eps}(y))).  
	\end{equation}
\end{lem}

The automorphism $\underline{\eps}$ rescales the columns of $k \times n$ matrix representatives by appropriately chosen signs. The argument is similar to \cite[Proposition 7.14]{MSTwist}.

\begin{proof}
	Let $f,m,i \in {\rm Bound}(k,n)$ be the lifts of $\pi,\mu,\iota$ and set $r = f^{-1}i$. From the assumption (2) we have length-additivity $\ell(m) = \ell(r^{-1}i) = \ell(r^{-1})+\ell(i)$. 
	
	Consider an infinite matrix $z$ with $k$ rows and with columns $(z_i)_{i \in \bbz}$. Let $r$ act on $z$ by permuting columns, so $r(z)_i = z_{r(i)}$. 
	Then 
	\begin{align}
	\Delta_{I}(r(z)) &= \det(z_{r(i_1)},\dots,z_{r(i_k)}) \\
	&= (-1)^{\# \{a < b \in I \times I \colon \ell(rt_{ab}) < \ell(r) \}}\Delta_{r(I)}(z) \\
	&= (-1)^{\# \{a < b \in I \times I \colon \ell(t_{ab}r^{-1}) < \ell(r^{-1}) \}}\Delta_{r(I)}(z) \\
	&= (-1)^{\# \{a < b \in [n] \times I \colon \ell(t_{ab}r^{-1}) < \ell(r^{-1}) \}}\Delta_{r(I)}(z) \\
	&= \prod_{b \in I}(-1)^{\# \{a \in [n] \colon \ell(rt_{ab}) < \ell(r)\}}\Delta_{r(I)}(z).
	\end{align}
	The second equality is sorting the columns; the third is the statement that inverting affine permutations does not affect length. The fourth equality follows from \cref{lem:aimpliesb} and the length-additivity assumption. Indeed, 	
	$$\ell(t_{ab}m) = \ell(t_{ab}r^{-1}i) \leq \ell(t_{ab}r^{-1})+\ell(i) < \ell(m), $$
    so that the condition that $b \in I$ implies already that $a \mod n$, which is $a$ itself, is also in $I$. The last line is just regrouping terms. We set $\eps_1(b) = (-1)^{\# \{a \in [n] \colon \ell(rt_{ab}) < \ell(r)\}}$; we will combine this sign with other signs momentarily.     
	
	Now we specialize the matrix $z$ to the matrix $y^\infty$ with blocks $\cdots | y | y | y | \cdots$. We compare the sign relating $\Delta_{r(I)}(y^\infty)$ to $\Delta_{\rho(I)}(y)$. Write $r(I) = r_1 < r_2 < \cdots < r_k$. If a given $r_b <1,$ then $r_b$ must sort past $|r(I) \cap [1,r_b+n)|$ many values in order for it to occupy the correct place in $\rho(I)$. By the dual argument when $r_b>n$, we have: 
	$$\Delta_{r(I)}(z) = \prod_{b \colon r_b<1} (-1)^{\# r(I) \cap [1,r_b+n)} \prod_{b \colon r_b > n}(-1)^{\# r(I) \cap (r_b-n, n]}\Delta_{\rho(I)}(y).$$
   We define a sign $\eps_2(b)$ which is 1 when $r_b \geq 1$ and which is otherwise equal to the factor $(-1)^{\# r(I) \cap [1,r_b+n)}$ appearing in the first exponent on the right hand side. We define a third sign $\eps_3(b)$ which is 1 when $r_b \leq n$ and which is otherwise equal to $(-1)^{\# r(I) \cap r(I) \cap (r_b-n, n]}$. Finally we combine these signs to get a sign $\eps(b) = \eps_1(b)\eps_2(b)\eps_3(b)$. 

   We define an automorphism 
	$\underline{\eps}' \in \Aut(\Gr(k,n))$ rescaling the $b$th column by the sign $\eps(b)$ (for $b=1,\dots,n$) and define a similar  
	$\underline{\eps} \in \Aut(\Gr(k,n))$ rescaling the $b$th column by the sign $\eps(\rho(b))$. Thus $\underline{\eps}' \circ \rho = \rho \circ \underline{\eps}$. By the second calculation above we have $\Delta_{\rho(I)}(y) = \prod_{b \in I}\epsilon_2(b)\epsilon_3(b)\Delta_{r(I)}(y^\infty) $. By the first calculation above we have 
   $\Delta_{r(I)}(y^\infty) = \prod_{b \in I}\varepsilon_1(b) \Delta_I(r(y^\infty))$. Clearly $\Delta_I(r(y^\infty)) = \Delta_I(\rho(y))$. 
Putting these together, we have $$\Delta_{\rho(I)}(y) = \Delta_{I}(\underline{\eps}'\rho(y)) = \Delta_{I}(\rho\underline{\eps}(y))$$ as desired. 
\end{proof}

We follow the notation in \cite{MSTwist}, denoting by $(\bbc^*)^E / (\bbc^*)^{V-1}$ the algebraic torus of edge weights on $G^\rho$ modulo restricted gauge transformations at vertices, and by $\bbc^F$ an algebraic torus whose coordinates are indexed by the faces of $G^\rho$. (Neither of these algebraic tori is sensitive to the relabeling of the boundary vertices.)

For a plabic graph $G$, there is a {\sl boundary measurement map} $\meas_G$ from $(\bbc^*)^E/ (\bbc^*)^{V-1} \to \tGr(k,n)$. The Pl\"ucker coordinate $\Delta_I$ of a point in the image is the weight-generating function for matchings of $G$ with boundary $I$. 

For a relabeled plabic graph $G^\rho$, we define $\meas_{G^\rho}:= \rho^{-1} \circ \meas_G$. Up to sign, $\meas_{G^\rho}$ is the weight-generating function for matchings of $G^\rho$ with given boundary (taking into account the relabeling of vertices according to $\rho$).

We also have a rational map $\targetF_{G^\rho}(\bullet) \colon \tGr(k,n) \to (\bbc^*)^F$ given by evaluating the Pl\"ucker coordinates $\Delta(\targetF(G^\rho))$. This evaluation map agrees with the composition $\targetF_{G}(\bullet) \circ \rho$ up to sign. More specifically, by \cref{lem:signssigns} we have $\targetF_{G^\rho}(\bullet) =\targetF_{G}(\bullet)   \circ \rho  \circ \underline{\eps}$.

Muller and Speyer defined an invertible Laurent monomial map $\vec{\partial} \colon (\bbc^*)^F \to (\bbc^*)^E / (\bbc^*)^{V-1}$ whose inverse is denoted $\cev{\mathbb{M}}$. We do not use any properties of either $\vec{\partial}$ or $\cev{\mathbb{M}}$ beyond that they fit into the commutative diagram \cite[Theorem 7.1]{MSTwist} and are monomial maps.

\begin{prop}[Main commutative diagram] \label{prop:faceLabelsTorus}
Suppose $G^\rho$ is a relabeled plabic graph with trip permutation $\pi$ satisfying \cref{thm:main}(1) and let $\mu$ be the trip permutation of the underlying graph $G$. Then we have a commutative diagram
\begin{equation}\label{eq:commutative}
\begin{tikzcd}
(\mathbb{C}^*)^F \arrow[r, "\vec{\partial}"] & (\mathbb{C}^*)^E/(\mathbb{C}^*)^{V-1} \arrow[d, "\tilde{\mathbb{D}}_{G^\rho}"] \\
\tPio_\pi \arrow[u, "\targetF_{G^\rho}(\bullet)", dotted]               & \rho^{-1}(\tPio_\mu) \arrow[l, "\underline{\eps}\circ \lt_\pi"]                              
\end{tikzcd}.
\end{equation}
In particular the domain of definition of $\targetF_{G^\rho}(\bullet)$ is an algebraic torus and $\Delta(\targetF_{G^\rho}) \subset \bbc(\tPio_\pi)$ is a seed. 
\end{prop}

We strengthen the above birational statements (i.e., that $\targetF_{G^\rho}(\bullet)$ determines an algebraic torus chart on $\tPio_\pi$) to a stronger statement about cluster structures in \cref{thm:genPlabicClusterStruc} below. 

Before the proof of \cref{prop:faceLabelsTorus}, we mention a corollary related to total positivity. Recall that $\tPio_{\pi, >0}=\{x \in \tPio_\pi: \Delta_I(x)>0 \text{ for all } I \in \mcm_\pi\}$.

\begin{cor}[$\gnec$-twist respects positivity]\label{cor:twistPositive}
		Suppose $\iota \ler \pi$ are permutations of type $(k, n)$ and the Grassmannlike necklace $\gnec=\gnec_{\rho, \iota, \pi}$ is weakly separated. Let $\mu:=\rho^{-1}\iota$. Then $\rt_\gnec \circ \underline{\eps}(\tPio_{\pi,>0}) =\tPio_{\mu,>0}$. 
\end{cor}

\begin{proof}
	By \cref{prop:faceLabelsTorus}, 
	the map $\rt_\gnec \circ \underline{\eps} \colon \tPio_\pi \to \tPio_\mu$ can also be expressed as a composition $\widetilde{\bbd}_G\circ \vec{\partial}\circ \targetF_{G^\rho}$. Each of the maps in the composition sends $\bbr_{>0}$-points to $\bbr_{>0}$-points and is surjective on such points. 
\end{proof}

\begin{proof}[Proof of \cref{prop:faceLabelsTorus}]
From the commutativity of the left square in \cite[Theorem 7.1]{MSTwist}, we have
\begin{align}
\cev{\mathbb{M}} &= \targetF_G(\bullet) \circ \lt_\mu \circ \tilde{\mathbb{D}}_G
\end{align}
as maps $(\mathbb{C}^*)^E/(\mathbb{C}^*)^{V-1} \to (\bbc^*)^F$.

We seek to prove the commutativity: 
\begin{align}
\cev{\mathbb{M}} &= \targetF_{G^\rho}(\bullet) \circ \underline{\eps} \circ \lt_\pi \circ \tilde{\mathbb{D}}_{G^\rho} = \targetF_{G}(\bullet)\circ \rho \circ \lt_\pi \circ \rho^{-1} \circ \tilde{\mathbb{D}}_G.
\end{align}

Since $\tilde{\mathbb{D}}_G$ is injective, it suffices to prove that 
\begin{align}\label{eq:reducetothis}
\targetF_G(\bullet) \circ \lt_\mu &= \targetF_{G}(\bullet) \circ \rho \circ \lt_\pi \circ \rho^{-1}
\end{align}
as maps $\tPio_\mu \to (\bbc^*)^F$. We caution the reader that $\lt_\mu \neq \rho \circ \lt_\pi \circ \rho^{-1}$ on the domain $\tPio_\mu$ (see \cref{eg:crossproduct} below). Nonetheless, the Pl\"ucker coordinates in which these two maps disagree do not arise as target face labels of plabic graphs $G$ with trip permutation $\mu$, as we presently explain.

Let $F$ be a face of $G$ and $y \in \tPio_\mu$. Set $x' = \lt_\mu(y)$ and $x = \rho \lt_\pi \rho^{-1}(y)$. By \cref{thm:twistsareinverses} we have $\lt_\mu\rt_\mu(x) = x'$. So to establish \eqref{eq:reducetothis} we need to prove that if 
$I = \tI(F)$ is a target label of a plabic graph with trip permutation $\mu$ and if $x \in D(\rmci)$, then $\Delta_I(x) = \Delta_I(\lt_\mu\rt_\mu(x))$. We prove this in \cref{lem:twistisalmostinverse} below. The commutativity of \eqref{eq:commutative} is proved. 

By the commutativity of the diagram, the domain of definition of $\targetF_{G^\rho}(\bullet)$ is the image of $\lt_\pi\tilde{\bbd}_{G^\rho} = \lt_{\gnec^*} \circ \tilde{\bbd}_{G}$. Thus, it is an algebraic torus because $\lt_{\gnec^*}$ is an isomorphism by \cref{thm:twistsareinverses}. Any regular function on $\tPio$ restricts to a regular function on this algebraic torus, hence to a Laurent polynomial in its basis of characters 
$\Delta(\targetF_{G^\rho})$. This shows that every regular function can be expressed as  Laurent polynomial in $\Delta(\targetF_{G^\rho})$. Thus, $\Delta(\targetF_{G^\rho})$ generates the function field $\bbc(\tPio_\pi)$. Since $\dim \tPio_\pi = \# \targetF_{G^\rho}$ we conclude that $\Delta(\targetF_{G^\rho})$ is algebraically independent and thus $\Sigma_{G ^\rho}^T$ is a seed in $\bbc(\tPio_\pi)$. 
\end{proof} 

The proof of \cref{prop:faceLabelsTorus} promised the following lemma which we now state and prove. 
\begin{lem}[Triangularity lemma]\label{lem:twistisalmostinverse}
If $I = \tI(F)$ is the target label of a reduced plabic graph with trip permutation $\pi$, and if $z \in D(\rmci_\pi)$, then $\Delta_I(z) = \Delta_I(\lt_\pi\rt_\pi(z))$. 
\end{lem}

That is, while the maps $\rt_\pi$ and $\lt_\pi$ are not inverse rational maps, the composition $\lt_\pi \rt_\pi$ preserves {\sl certain} Pl\"ucker coordinates. We call the lemma a triangularity lemma because, as we show below, the passage from $z$ to $\lt_\pi \rt_\pi(z)$ amounts to a triangular change of basis (in columns $I$). 

\begin{proof}
Since we always deal with $\lt_\pi$ and $\rt_\pi$ in this proof, we omit the subscript $\pi$. As discussed above, the conclusion of \cref{lem:aimpliesb} endows $[n]$ with a partial order which we denote by $\prec$ during this proof. Thus, if $a \prec b$, then $b$ appears in the target label $I$ of a plabic graph whenever $a$ does.

Let $M$ be a matrix with columns $M_1,\dots,M_n$, representing a point $z \in D(\rmci_\pi)$ . The key claim which underlies our argument is the following. For any $a \in [n]$, we have 
 \begin{align}\label{eq:triangularity}
(\lt \rt M)_{\pi(a)} &\in {\rm span}\{M_{\pi(b)} \colon \pi(a) 
\prec \pi(b\}.
\end{align}
We first explain that this key claim implies the desired statement. Clearly $\pi(a) \in \rightI_{\pi(a)}$ and hence $\{\pi(b) \colon \pi(a) \prec \pi(b)\} \subset \rightI_{\pi(a)}$ using \cref{lem:aimpliesb}. Since $z \in D(\rmci_\pi)$, it follows that the column vectors on the right hand side of \eqref{eq:triangularity} are linearly independent. By the definition of $\lt$ we have $\langle (\lt \rt M)_{\pi(a)}, (\rt M)_{\pi(a)}\rangle = 1$. Assuming \eqref{eq:triangularity}, and using the perpendicularity statement in the definition of $\rt$, it follows that when we expand 
$(\lt \rt M)_{\pi(a)}$ in terms of the vectors on the right hand side of \eqref{eq:triangularity}, the coefficient of $M_{\pi(a)}$ equals~1.  Choosing an ordering of columns that refines the partial order $\prec$, it follows that the matrices $(M_a)_{ a \in \tI(F)}$ and $((\lt \rt M)_a)_{ a \in \tI(F)}$ are related by a triangular matrix with ones on the diagonal. Thus the two matrices have the same determinant. That is, we have the desired equality of  Pl\"ucker coordinates 
$$\Delta_I(z) =: \det \left((M_a)_{ a \in \tI(F)}\right) \overset{\eqref{eq:triangularity}}{=}  \det \left(((\lt \rt M)_a)_{ a \in \tI(F)}\right) := \Delta_I(\lt_\pi\rt_\pi(z)).$$

What remains then is to argue the key claim \eqref{eq:triangularity}. Using the definition of $\lt$ in column $\pi(a)+1$ and then using the definition of $\rt$ for each column in $\leftI_{\pi(a)+1} \setminus \pi(a)$, we have 
\begin{align}
(\lt \rt(M))_{\pi(a)} \in \left({\rm span}(\rt(M)_x)_{x \in \leftI_{\pi(a)+1} \setminus \pi(a)}\right)^\perp  
&\subseteq \bigcap_{x \in \leftI_{\pi(a)+1} \setminus \pi(a)} {\rm span}(M_{y})_{y \in \rightI_x \setminus x} \\
& \subseteq \bigcap_{x \in \leftI_{\pi(a)+1} \cap(a,\pi(a))} {\rm span}(M_{y})_{y \in \rightI_x \setminus x}
\label{eq:anyx}.
\end{align}
In the last step we simply took an intersection over a smaller indexing set. 

We list the elements $x$ in $\leftI_{\pi(a)+1} \cap(a,\pi(a))$ in cyclic order as $a< x_s < \cdots < x_1 < \pi(a)$. Let $w$ be an arbitrary number in the cyclic interval $(a,\pi(a))$. It follows by comparing the definitions of $\leftI_{\pi(a)+1}$ and $\rightI_w$ that 
\begin{equation}\label{eq:zisgood}
\pi(w) \in (w,\pi(a)) \text{ if and only if } w \notin \{x_1,\dots,x_s\}.
\end{equation}

We will now compute the right hand side of \eqref{eq:anyx} directly. We claim inductively, for $0 \leq t \leq s$, that 
\begin{align*}
\bigcap_{j=1}^t {\rm span}(M_{y})_{y \in \rightI_{x_j} \setminus x_j}  &= {\rm span}(M_y)_{y \in \rightI_{\pi(a)} \setminus \{\pi(x_1),\dots,\pi(x_t)\}}. 
\end{align*}
When $t = 0$ we interpret the left hand side as an empty intersection (hence, as all of $\bbc^k$) and the base case holds since $M \in D(\rmci)$. Evaluating the inductive claim when $t = s$ and using \eqref{eq:zisgood} we see
that the right hand side is spanned by $M_{\pi(a)}$ as well as various $M_{\pi(b)}$'s where $b<a< \pi(a) < \pi(b)$ which is the key claim \eqref{eq:triangularity}. 

Assuming the inductive claim for a given $t \in [0,s)$, we have 
\begin{align*}
\bigcap_{j=1}^{t+1} {\rm span}(M_{y})_{y \in \rightI_{x_j} \setminus x_j}  &= {\rm span}(M_y)_{y \in \rightI_{\pi(a)} \setminus \{\pi(x_1),\dots,\pi(x_t)\}} \cap 
{\rm span}(M_{y})_{y \in \rightI_{x_{t+1}} \setminus x_{t+1}} \\
&= {\rm span}(M_y)_{y \in \left(\rightI_{\pi(a)} \setminus \{\pi(x_1),\dots,\pi(x_t)\}\right) \cap \left(\rightI_{x_{t+1}} \setminus x_{t+1}\right)}\\
&= {\rm span}(M_y)_{y \in \rightI_{\pi(a)} \setminus \{\pi(x_1),\dots,\pi(x_{t+1})\}}.
\end{align*}
The first equality is the inductive assumption. To establish the second equality, we claim that $(\rightI_{x_{t+1}} \setminus x_{t+1}) \cup (\rightI_{\pi(a)} \setminus \{\pi(x_1),\dots,\pi(x_t)\})) \subseteq \rightI_{x_{t+1}+1}$. The first containment is the definition of Grassmann necklace and the second containment follows from \eqref{eq:zisgood}. Since $M \in D(\rmci)$, the vectors $\{M_y \colon y \in \rightI_{x_{t+1}+1}\}$ are independent. Thus, we can replace the intersection of spans with the span of the intersections, which is what we have done in the second equality. The third equality is another instance of \eqref{eq:zisgood}, noting that $\pi(x_{t+1}) \in \rightI_{\pi(a)}$ but is not in $\rightI_{x_t}$. This completes the inductive proof, establishing \eqref{eq:triangularity}. 
\end{proof}

\begin{example}[$\lt_\mu$ vs. $\rho\lt_\pi\rho^{-1}$]\label{eg:crossproduct}
Continuing \cref{eg:differentbases}, we demonstrate our assertion from the proof of \cref{prop:faceLabelsTorus}, namely that the maps $\lt_\mu$ and $\rho\lt_\pi\rho^{-1}$ do not coincide on $\tPio_\mu$. The map $\lt_\mu$ restricts to an automorphism of $\tPio_\mu$. On the other hand, we will show that $\rho\lt_\pi\rho^{-1}(\tPio_\mu)$ is not contained in $\tPio_\mu$. Recall that any point of $\tPio_\mu$ has the property that any matrix representative has parallel 3rd and 4th columns.

We have $\rho = 123546$ in one line notation and $\lmci_\pi = (456, 146, 126, 123, 234, 245)$. If $M_1,\dots,M_6$ are the columns of a matrix $M$ representing a point in $\tPio_\mu$, with similar notation as in \cref{eg:thetwotwists}, we have 
$$\rho \lt_\pi \rho^{-1}(M) = 
\begin{bmatrix}
\frac{M_5 \times M_6}{\Delta_{156}(M)} & \frac{M_1 \times M_6}{\Delta_{126}(M)}& \frac{M_1 \times M_2}{\Delta_{123}(M)} & \frac{M_2 \times M_5}{\Delta_{245}(M)} & \frac{M_2 \times M_3}{\Delta_{235}(M)} & \frac{M_4 \times M_5}{\Delta_{456}(M)}
\end{bmatrix}.$$
Note that the Pl\"ucker coordinates appearing in the denominators of these formulas are nonvanishing using the fact that $M \in \tPio_\mu$. Provided $\Delta_{125}(M) \neq 0$, which is generically true on $\tPio_\mu$, one can see that the cross product of the third and fourth columns of the above matrix is a nonzero scalar multiple of $M_2$. Thus, 
one has $\rho \lt_\pi \rho^{-1}(M) \notin \tPio_\mu$  for generic $M \in \tPio_\mu$.
\end{example}

Now we extract the stronger statement about cluster structures from the commutative diagram \eqref{eq:commutative}. 
\begin{thm}[$G^\rho$ gives a cluster structure]\label{thm:genPlabicClusterStruc}
Suppose $\pi \rho \ler \pi$ and let $G^\rho$ be a reduced plabic graph with trip permutation $\pi$. Suppose $G^\rho$ satisfies \cref{thm:main}(1). Then we have the equality of cluster algebras $\mca(\Sigma^T_{G^\rho}) = \bbc[\tPio_\pi]$. 
\end{thm}

\begin{proof}
First we claim that the analogue of the double twist formula \cite[Proposition 7.13]{MSTwist} holds. Let 
$\varphi := \rt_\mu \circ \rt_\gnec \circ \underline{\eps} = \rt_\mu^2 \circ \rho \circ \underline{\eps} \colon \tPio_\pi \to \tPio_\mu$. We claim that $\varphi^*(\Sigma^S_G) \sim \Sigma^T_{G^\rho}$, i.e. that the seed obtained by pulling back 
the source collection $\Sigma^S_G \subset \bbc[\tPio_\mu]$
along $\varphi$ is quasi-equivalent to the target seed $\Sigma^T_{G^\rho}$. Indeed, we can rewrite our commutative diagram \eqref{eq:commutative} and combine it with the right diagram in \cite[Theorem 7.1]{MSTwist} to obtain a commutative diagram
\begin{equation}\label{eq:commutative2}
\begin{tikzcd}
(\mathbb{C}^*)^F \arrow[r, "\vec{\partial}"]         & (\mathbb{C}^*)^E/(\mathbb{C}^*)^{V-1} \arrow[d, "\tilde{\mathbb{D}}_G"] \arrow[r, "\vec{\bbm}"] & (\mathbb{C}^*)^F                \\
\tPio_\pi \arrow[u, "\targetF(G^\rho)(\bullet)", dotted] \arrow[r, "\rt_{\gnec} \circ \underline{\eps}"] & \tPio_\mu \arrow[r, "\vec{\tau}_\mu"]                                                         & \tPio_\mu \arrow[u, "\sourceF(G)", dotted].
\end{tikzcd}\end{equation}
Repeating the proof of \cite[Proposition 7.13]{MSTwist} using this diagram we obtain the formula 
\begin{equation}\label{eq:doubletwist}
\Delta_{\sI(F)}(\varphi(y)) = \Delta_{\rho(\tI(F))}(y)\prod_{i \in \sI(F)}\frac{\Delta_{I_{i}}(y)}{\Delta_{I_{i+1}}(y)}.
\end{equation}

The Pl\"ucker coordinate on the left is in $\sourceF(G)$ whereas the Pl\"ucker coordinate $\rho(\tI(F)) \in \targetF(G^\rho)$ and likewise for the necklace variables $I_i,I_{i+1} \in \gnec$. Since $\gnec$ is a unit necklace, the multiplicative factor on the right hand side of \eqref{eq:doubletwist} is in $\bbp_\pi$ as required by \cref{defn:seedorbit}. To complete the proof that $\varphi^*(\Sigma^S_G) \sim \Sigma^T_{G^\rho}$, we need to check that the $\hat{y}$'s in these two seeds coincide. That is, we need to show that the multiplicative $\bbp_\pi$-factors on the right hand side of \eqref{eq:doubletwist} cancel out when we compute the Laurent monomial $\hat{y}$. This follows from the  well-known fact that each of the $\hat{y}$ monomials is homogeneous with respect to the $\bbz^n$-grading on $\bbc[\tPio_\pi]$ given by the degree in the column vectors: the number of times that a given $i$ appears in the numerator of each $\hat{y}$ cancels with the number of times it appears in the denominator, and thus the same is true of each $\Delta_{I_{i}}/\Delta_{I_{i+1}}$. 

Since $\varphi^*(\Sigma^S_G) \sim \Sigma^T_{G^\rho}$, we have then that 
$$\mca(\Sigma^T_{G^\rho}) = \mca(\varphi^*(\Sigma^S_G)) = \varphi^*(\mca(\Sigma^S_G)) = \varphi^*(\bbc[\tPio_\mu]) = \bbc[\tPio_\pi].$$
\end{proof}

This completes the proof of \cref{thm:main} (leaving the equivalence of the combinatorial conditions (2) and (3) for \cref{secn:proofs}) . We also have the following corollary about the positive parts of $\tPio_\pi$ determined by relabeled plabic graph seeds.

\begin{cor}[$G^\rho$ gives a positivity test]\label{cor:posPartsCoincide}
	Suppose $\pi\rho \ler \pi$. Suppose $G^\rho$ is a relabeled plabic graph with trip permutation $\pi$ and satisfies \cref{thm:main}(1). Then the positive part of $\tPio_\pi$ determined by $\Sigma_{G ^\rho}^T$ is equal to $\tPio_{\pi,>0}$. That is, 
	$$\{x \in \tPio_\pi: \Delta_{\tI(F)}(x)>0 \text{ for all faces } F \text{ of } G^\rho\}=\{x \in \tPio_\pi: \Delta_{I}(x)>0 \text{ for all } I \in \mcm_\pi\}.
	$$
\end{cor}

\begin{proof}
	
	Let $\gnec=\gnec_{\rho, \iota, \pi}$. By going up the left hand side of the main commutative diagram \eqref{eq:commutative}, we see that $\targetF(G^\rho) \subset \mcm_\pi$, thus the 
right left hand set above is contained in the left hand set. 
	
	The containment of the left hand side in the right hand side follows from \cref{cor:twistPositive} and its dual statement, which is that $\underline{\eps}\circ\lt_{\gnec^*}(\tPio_{\mu,>0}) = \tPio_{\pi,>0}$.  Indeed, if $x$ is in the left hand side, then by \cref{cor:twistPositive}, $\rt_{\gnec}\circ \underline{\eps}(x) \in \tPio_{\mu,>0}$. From the dual statement, we have that $$\underline{\eps} \circ \lt_{\gnec^*}(\rt_{\gnec}\circ \underline{\eps} (x))=x\in \tPio_{\pi,>0}.$$
\end{proof}

%% file: reachableB.tex
In this section, we investigate the quasi-equivalence of the different cluster structures on $\tPio_\pi$ given by \cref{thm:main}. We verify the quasi-equivalence of source and target cluster structures (\cref{conj:sourceistarget}) for a class of positroids we call ``toggle-connected.'' We also prove the quasi-equivalence of all cluster structures given by \cref{thm:main} for (open) Schubert and opposite Schubert varieties, verifying a special case of \cref{conj:samepattern}.

First we recall our main conjecture on quasi-equivalence:
\begin{conj:samepattern}[Quasi-equivalence conjecture]
Let $G^\rho$ be a relabeled plabic graph satisfying the conditions of \cref{thm:main}, determining a cluster structure on $\tPio_\pi$. Let $H$ be a plabic graph with trip permutation $\pi$. Then the seeds $\Sigma^T_{G^\rho}$ and $\Sigma^T_H$ are related by a quasi-cluster transformation.\footnote{In the language of \cite{Fra}, this conjecture says that the map $\rt_\gnec \circ \underline{\eps}$ from \cref{sec:twist} is a {\sl quasi-isomorphism} of the target structures on $\tPio_\pi$ and $\tPio_\mu$.}
 \end{conj:samepattern}

\begin{rmk}\label{rmk:sourceisrelabeled}
In \cref{eg:sourceisrelabeled} we explained that $\Sigma^S_G = \Sigma^T_{G^{\pi^{-1}}}$. However, the seed $\Sigma^T_{G^{\pi^{-1}}}$ does not fit into the framework of Theorem \cref{thm:main} and \cref{conj:samepattern} because when $\rho = \pi^{-1}$, $\iota := \pi \rho$ does not have type $(k,n)$. This is not a problem: if we instead choose $\rho=\pi^{-1} \eps_k$ and $\iota= \eps_k$, then $\iota \ler \pi$ as desired. If $H^\rho$ has trip permutation $\pi$, then the target labels of the boundary faces of $H^\rho$ give the shifted reverse Grassmann necklace $\lmci_{\pi}[k]$.  It is not hard to see that $\Sigma_{H ^\rho}^T$ is equal to the source seed $\Sigma_{G}^S$ for some plabic graph\footnote{in fact, for a rotation of the underlying graph $H$} $G$ with trip permutation $\pi$. So the source-labeled seed of $G$ can be realized as a relabeled plabic graph seed satisfying the assumptions of \cref{thm:main} and \cref{conj:samepattern}.
\end{rmk}

\begin{rmk}
If indeed \cref{conj:samepattern} holds, then the two cluster structures on $\tPio_\pi$ given by the seeds $\Sigma_{G ^\rho}^T$ would and $\Sigma_H^T$ would give rise to the same notion of positive part of $\tPio_\pi$. \cref{cor:posPartsCoincide} confirms that this is indeed the case, supporting \cref{conj:samepattern}.
\end{rmk}

We now state our main result in the direction of \cref{conj:samepattern}. As preparation, let $G^\rho$ and $H^\sigma$ be reduced plabic graphs with trip permutation $\pi$, both satisfying 
\cref{thm:main}(1) (i.e., satisfying both the hypotheses of that theorem and the equivalent conditions (1) through (4)). Thus the seeds $\Sigma_{G^\rho}^T$ and 
$\Sigma_{H^\sigma}^T$ give rise to cluster structures on $\tPio_\pi$. The boundary faces of these graphs yield Grassmannlike necklaces $\gnec_{\rho,\bullet,\pi}$ and $\gnec_{\sigma,\bullet,\pi}$.

\begin{thm}[Toggling as quasi-cluster transformation] \label{thm:alignedIsQuasiEquiv}  Suppose, in the setting of the previous paragraph, that the Grassmannlike necklaces $\gnec_{\rho,\bullet,\pi}$ and $\gnec_{\sigma,\bullet,\pi}$
are related by an aligned toggle. Then the seeds  $\Sigma^T_{G^\rho}$ and 
$\Sigma^T_{H^\sigma}$ are related by a quasi-cluster transformation. 
\end{thm}

The proof of \cref{thm:alignedIsQuasiEquiv} is in \cref{sec:reachablepf}. 
Informally, the argument is as follows. By a sequence of square moves (at interior faces), one can pass from the given graph $G^\rho$ to another graph $(G')^\rho$ with the same trip permutation $\pi$, and with the property that the aligned toggle relating $\gnec_1$ and $\gnec_2$ is carried out by ``performing a square move" at the corresponding boundary face of $(G')^\rho$. By a simple argument, performing square moves at boundary faces in this way is a quasi-cluster transformation. 

We now define a \emph{toggle graph} to summarize the various quasi-equivalences which follow from \cref{thm:alignedIsQuasiEquiv}.

\begin{defn}[Toggle graph]\label{defn:TransitiveClosure}
Fix $f \in {\rm Bound}(k,n)$ and let $\sep_f$ be the set of $i \leq_R f$ such that $\ell(i^{-1}fi)= \ell(f)$. Define an (undirected) graph $TG_f$ on $\sep_f$ by putting an edge between $i$ and $w$ if $w=i s_a$ for some $a$. That is, $TG_f$ is obtained from the Hasse diagram of the lower order ideal of $f$ in $(\bd(k, n), \leq_R)$ by deleting all elements $i \leq_R f$ with $\ell(i^{-1}fi) \neq \ell(f)$ (see \cref{fig:notreachable}).

We define an analogous graph for permutations of type $(k, n)$ by applying the map $f \mapsto \overline{f}$ everywhere, and use the same notation.

We say that $f$ is \emph{toggle-connected} if $f$ and $e_k$ are in the same connected component of $TG_f$.
\end{defn}

\begin{rmk}\label{rmk:connectedComp}
	Each vertex of $TG_f$ corresponds to a Grassmannlike necklace satisfying condition (2) of \cref{thm:main}, thus to a cluster structure on $\tPio_\pi$. The edges of $TG_f$ record when two such necklaces are related by an aligned toggle. By \cref{thm:alignedIsQuasiEquiv}, any two necklaces in the same connected component of $TG_f$ determine quasi-equivalent cluster structures. 
\end{rmk}

\begin{example} For the affine permutation $f=[4, 6, 5, 8, 7, 9]$ appearing in \cref{fig:fourplabics}, \cref{ex:easyToggles,eg:someeasytoggles2}, the weak order lower order ideal beneath $f$ consists of 4 permutations, namely $f$ itself together with $[4, 5, 6, 8, 7, 9]$, $[4, 6, 5, 7, 8, 9]$, and $e_3$. Each of these permutations satisfies the length condition $\ell(i^{-1}fi) = \ell(f)$, i.e. $\sep_f$
coincides with the lower order ideal in this case. The toggle graph $TG_f$ is a 4-cycle, a connected graph, so the source and target cluster structures are quasi-equivalent in this case. On the other hand, \cref{fig:notreachable} illustrates a positroid whose toggle graph has two connected components. While it is simple to check in this particular case that one can pass from one of these connected components to the other by a quasi-cluster transformation, our general setup does not prove statements of this sort. 
\end{example}

\begin{example} We analyze in greater detail the cluster variables in the four quasi-equivalent cluster structures on $\tPio_\pi$ coming from the relabeled graphs in \cref{fig:fourplabics}. Each of these is a finite type cluster algebra of type $A_2$.

The leftmost graph is the the target structure. Three of the five clusters come from plabic graphs. The cluster variables are
\begin{equation}\label{eq:listallfive}
\Delta_{124},\Delta_{246},\Delta_{236},\Delta_{356},\Delta_{346}\Delta_{125},
\end{equation}
listed so that adjacent cluster variables form a cluster. 
The last cluster variable in \eqref{eq:listallfive} is a product of two Pl\"ucker coordinates hence is not the target label of a plabic graph.

A similar story holds for the rightmost graph, i.e. the source structure, with cluster variables 
$\Delta_{236},$ $\Delta_{246},$ $\Delta_{124},$ $\Delta_{145},$ and $\Delta_{146}\Delta_{235}.$ 

For the two intermediate cluster structures, every cluster comes from a relabeled plabic graph and every cluster variable is a Pl\"ucker coordinate. The cluster variables for the top center graph are
$\Delta_{124},$ $\Delta_{246}$, $\Delta_{236}$, $\Delta_{235}$, and $\Delta_{125}$. Those for the bottom center graph are  
$\Delta_{124}$, $\Delta_{246}$, $\Delta_{236}$, $\Delta_{136}$, and $\Delta_{134}.$

The following 8 Pl\"ucker coordinates appear as a boundary face label in one of the four plabic graphs in \cref{fig:fourplabics}:
$$\Delta_{123},\Delta_{234},\Delta_{346},\Delta_{456},\Delta_{256},\Delta_{126},\Delta_{146},\Delta_{245}.$$ 
Thus each of these Pl\"ucker coordinates is a unit in $\tPio_\pi$.

The following 9 Pl\"ucker coordinates appear as mutable variables in one of the cluster structures:
$$\Delta_{124},\Delta_{246},\Delta_{236},\Delta_{356} \equiv \Delta_{136} \equiv \Delta_{235}, \Delta_{134} \equiv \Delta_{125} \equiv \Delta_{145}.$$ 
Here we use the notation $\equiv$ to denote equality up to multiplication by an element of $\bbp_\pi$, working inside $\mathbb{C}[\tPio_\pi]$. One may check that $\Delta_{135} = \frac{\Delta_{356}\Delta_{125}}{\Delta_{256}} \in \mathbb{C}[\tPio_\pi]$ is a nontrivial cluster monomial (in any of the four cluster structures). Altogether, this accounts for 18 Pl\"ucker coordinates. The remaining two Pl\"ucker coordinates are $345,156 \notin \mcm_\pi$. 

Thus, at least in this small example, we can find every Pl\"ucker coordinate which is a mutable variable up to units as a mutable variable in {\sl some} seed arising from a relabeled plabic graph. 
\end{example}


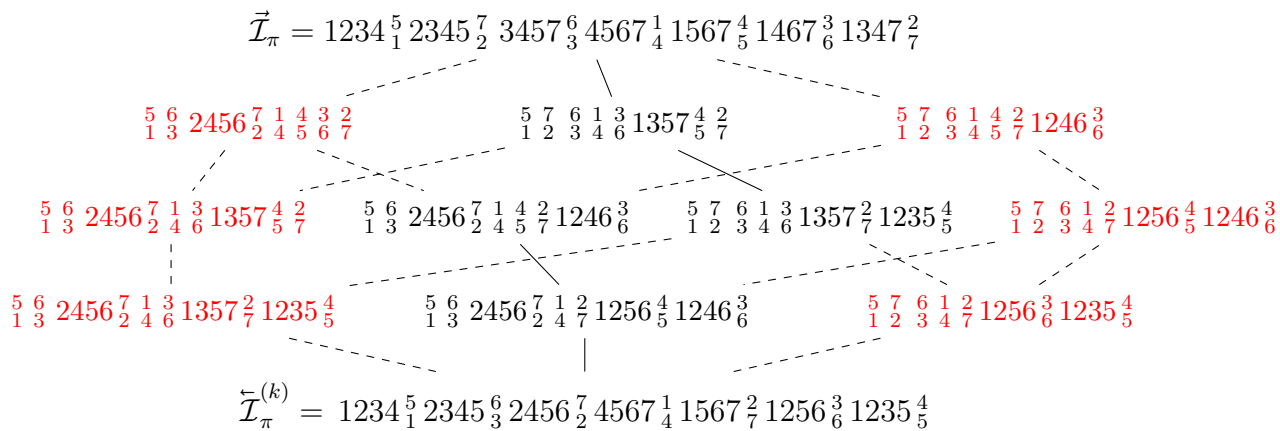
\begin{figure}
	\begin{tikzpicture}
	\node (A) at (-2.5,0) { $\rmci_\pi = 1234 \,^5_1 \, 2345\,^7_2\, \,3457\,^6_3 \,4567\,^1_4 \,1567 \,^4_5 \,1467\,^{3}_6\,1347 \,^{2}_7 $};
	
	\node (B) at (-7,-1.25) {\small \textcolor{red}{$ \,^5_1 \, \,^6_3\, \,2456\,^7_2 \,\,^1_4 \, \,^{4}_5 \,\,^{3}_6\, \,^{2}_7 $}};

	\node (C) at (-2,-1.25) {\small $ \,^5_1 \, \,^7_2\, \, \,^6_3 \,\,^1_4 \, \,^{3}_6 \,1357\,^{4}_5\, \,^{2}_7 $};
	
	\node (D) at (3,-1.25) {\small \textcolor{red}{$ \,^5_1 \, \,^7_2\, \,\,^6_3 \,\,^1_4 \,\,^{4}_5 \,\,^{2}_7\,1246 \,^{3}_6 $}};

	\node (E) at  (-8,-2.5) {\small \textcolor{red}{$ \,^5_1 \, \,^6_3\, \,2456\,^7_2 \,\,^1_4 \, \,^{3}_6 \, 1357 \,^{4}_5\, \,^{2}_7 $}};
	
	\node (F) at (-3.7,-2.5) {\small $ \,^5_1 \, \,^6_3\, \,2456\,^7_2 \,\,^1_4 \, \,^{4}_5 \,  \,^{2}_7\, 1246 \,^{3}_6 $};
	
	\node (H) at (4.9,-2.5) {\small \textcolor{red}{$ \,^5_1 \, \,^7_2\, \,\,^6_3 \,\,^1_4 \, \,^{2}_7 \, 1256 \,^{4}_5\, 1246 \,^{3}_6 $}};
	
	\node (G) at (.6,-2.5) {\small $ \,^5_1 \, \,^7_2\, \,\,^6_3 \,\,^1_4 \, \,^{3}_6 \, 1357 \,^{2}_7\, 1235 \,^{4}_5 $};
	
	\node (I) at (-8,-3.75) {\small \textcolor{red}{$ \,^5_1 \, \,^6_3\, \,2456\,^7_2 \,\,^1_4 \, \,^{3}_6 \, 1357 \,^{2}_7\,1235 \,^{4}_5 $}};
	
	\node (J) at (-2.5,-3.75) {\small $ \,^5_1 \, \,^6_3\, \,2456\,^7_2 \,\,^1_4 \, \,^{2}_7 \, 1256 \,^{4}_5\, 1246 \,^{3}_6 $};
	
	\node (K) at (3,-3.75) {\small \textcolor{red}{$ \,^5_1 \, \,^7_2\, \,\,^6_3 \,\,^1_4 \, \,^{2}_7 \, 1256 \,^{3}_6\, 1235 \,^{4}_5 $}};
	
	\node (L) at (-2.5,-5) { $\lmci_\pi^{(k)} = \,1234 \,^5_1 \,2345 \,^6_3\, 2456 \,^7_2 \, 4567\,^1_4 \,1567 \,^{2}_7 \, 1256 \,^{3}_6\, 1235 \,^{4}_5 $};
	
	\draw[dashed] (A)--(B);
	\draw (A)--(C);
	\draw[dashed] (A)--(D);
	\draw[dashed] (B)--(E);
	\draw[dashed] (B)--(F);
	\draw[dashed] (E)--(I);
	\draw[dashed] (C)--(E);
	\draw[dashed] (D)--(F);
	\draw (C)--(G);
	\draw[dashed] (D)--(H);
	\draw (F)--(J);
	\draw[dashed] (K)--(H);
	\draw[dashed] (K)--(G);
	\draw[dashed] (K)--(L);
	\draw[dashed] (I)--(L);
	\draw (J)--(L);
	\draw[dashed] (G)--(I);
	\draw[dashed] (H)--(J);

	\end{tikzpicture}
	\caption{The $\ler$ lower order ideal of $\pi = 5761 4 3 2$, a permutation of type $(4, 7)$. For each $\iota \ler \pi$, the Grassmannlike unit necklaces $\gnec_{\bullet,\iota,\pi}$ is displayed (to save space, elements of $\rmci_\pi$ are omitted from intermediate necklaces). The weakly separated necklaces, which have insertion permutation $\iota \in \sep_\pi$, are in black. For example, any necklace containing $2456,1347$ is not weakly separated. Edges are cover relations in $\ler$: solid edges are edges in $TG_\pi$, while dashed edges are not. Since there is no solid path from the top to the bottom, $\pi$ is not toggle-connected.
	}\label{fig:notreachable}
\end{figure}

We saw in \cref{rmk:sourceisrelabeled} that the source cluster structure on $\tPio_\pi$ corresponds to the vertex $\eps_k$ of $TG_\pi$. If $\pi$ is toggle-connected, then $\pi$ and $\eps_k$ are in the same connected component of $TG_\pi$, which by \cref{rmk:connectedComp} immediately implies the following corollary.

\begin{cor}[Source and target are quasi-equivalent]\label{cor:quasiCoinciding}
If $\pi$ is toggle-connected, then the source and target cluster structures on $\tPio_\pi$ are quasi-equivalent.
\end{cor}

\begin{rmk}
It is an unfortunate fact of life that not every $\pi \in {\rm Bound}(k,n)$ is toggle-connected. We do not see a way of constructing a quasi-cluster transformation from the target structure to the source structure purely within the world of Pl\"ucker coordinates and square moves. On the other hand, our results break up this problem into smaller subproblems (namely, the subproblem of finding a sequence of quasi-cluster transformations between connected components of $TG_\pi$). 
\end{rmk} 

We end this section by investigating $TG_\pi$ for open Schubert and opposite Schubert varieties, and showing that \cref{conj:samepattern} holds for these classes.

\begin{defn}
	A (loopless) open positroid variety $\tPio_\pi \subseteq \tGr(k, n)$ is an \emph{open Schubert variety} if $\pi$ has a single descent and no fixed points before the descent. It is an \emph{open opposite Schubert variety} if the numbers $1, \dots, k$ and $k+1, \dots, n$ appear in increasing order in $\pi$ and none of $k+1, \dots, n$ are fixed points.\footnote{Open Schubert varieties correspond to Le-diagrams that are filled entirely with pluses. Open opposite Schubert varieties correspond to Le-diagrams whose shape is a $k \times (n-k)$ rectangle and whose zeros form a partition.}
\end{defn}

\begin{prop}\label{prop:SchubWholeInterval}
	Let $\tPio_\pi$ be an open Schubert or opposite Schubert variety. Then for all $\iota \ler \pi$, $\iota \in \sep_\pi$.
\end{prop}

\begin{proof}
	First, suppose $\tPio_\pi \subset \tGr(k, n)$ is an open Schubert variety, so $\pi$ has a single descent. That is, there is a single $a$ such that $\pi(a) > \pi(a+1)$. Since $\pi$ has no fixed points in $[a]$, all of $\pi(1), \dots, \pi(a)$ are not anti-excedences of $\pi$. On the other hand, all of $\pi(a+1), \dots, \pi(n)$ are anti-excedences of $\pi$, so $a=n-k$.

	 The bounded affine permutation $f$ corresponding to $\pi$ satisfies $f(b)=\pi(b)$ for $b\in [n-k]$ and $f(b)=\pi(b)+n$ for $b=n-k+1, \dots, n$. If $a<b$ and $f(a)>f(b)$ with $a \in [n]$, then $b> n$, since the window notation of $f$ consists of an increasing sequence. Additionally, we have $f(a), f(b) \in [n+1, 2n]$. This means that the right associated reflections of $f$ all have the form $t_{ab}$ where $a \in [n]$ and $b>n$ and the left associated reflections of $f$ all have the form $t_{ab}$ where $a, b \in [n]$. Thus, $T_L(f) \cap T_R(f)= \emptyset$.
	 
	 Now, consider any $i \leq_R f$, so $f=iw$ is length-additive. It is not hard to see that $T_L(i) \subseteq T_L(f)$ and $T_R(w) \subseteq T_R(f)$, so in particular, $T_L(i) \cap T_R(w)$ is empty. By \cref{lem:lengthAddAndReflections}, this means that $wi$ is length-additive. Since $i^{-1}fi=wi$, we have that $\ell(i^{-1}fi)=\ell(w)+\ell(i)= \ell(f)$, so $i \in \sep_f$. This implies that $\iota := \overline{i}$ is in $\sep_\pi$. Since the choice of $i$ was arbitrary, this completes the proof.
	 
	 The proof for opposite Schubert varieties is similar.
\end{proof}

As an immediate corollary, we obtain the following. 

\begin{thm}[Quasi-equivalence for Schuberts]\label{thm:SchubGood}
	Let $\tPio_\pi$ be an open Schubert or opposite Schubert variety. Then each relabeled plabic graph $G^\rho$ with trip permutation $\pi$ whose boundary satisfies $\pi \rho \ler \pi$ gives rise to a cluster structure on $\tPio_\pi$. Moreover, all of these cluster structures are quasi-equivalent. In particular, \cref{conj:sourceistarget} holds for $\tPio_\pi$.
\end{thm}

\begin{rmk}
	Open skew-Schubert varieties are a subclass of positroid varieties indexed by skew-shapes contained in a rectangle. Relabeled plabic graphs with a particular boundary were shown to give a cluster structure on open skew-Schubert varieties in \cite{SBSW}. In fact, it is not difficult to show that open skew-Schubert varieties are toggle-connected, and moreover that the cluster structure given in \cite{SBSW} is quasi-equivalent to the target and source cluster structures.
\end{rmk}

%% file: proofsC.tex
We prove the weak analogue \cref{lem:notinpositroid} of Oh's theorem used in the proof of the unit necklace theorem. Then we prove the equivalence of the Coxeter-theoretic condition (2) and the weak-separation theoretic condition (3) from \cref{thm:main}.  Finally, we prove 
\cref{thm:alignedIsQuasiEquiv} which says that aligned toggles on Grassmannlike necklaces induce quasi-equivalences of cluster structures.

\subsection{Proof of {\cref{lem:notinpositroid}}}\label{secn:rightweakproof}
Recall the definition of noncrossing and aligned chords and toggles from \cref{defn:noncrossing}.

By Remark~\ref{rmk:togglePluckerRel}, if we perform an aligned toggle at a necklace $\gnec$ satisfying $\gnec \subset \mcm_\pi$ then the new necklace $\gnec' \subset \mcm_\pi$ as well.

Since $\rho$ and $\iota$ determine $\gnec$, we will frequently omit the subsets $I_i$, writing the removal and insertion values in the following {\sl two-line notation}:
\begin{equation}\label{eq:chordsnotationii}
\gnec = \begin{matrix}\iota_1 \\ \rho_1\end{matrix}
\begin{matrix}\iota_2 \\ \rho_2\end{matrix} \cdots 
\begin{matrix}\iota_{n-1} \\ \rho_{n-1}\end{matrix}
\begin{matrix}\iota_n \\ \rho_n\end{matrix}.
\end{equation}

Now we prove Lemma~\ref{lem:notinpositroid} which is needed in the proof of Theorem~\ref{thm:rightweak}.

\begin{proof}[Proof of Lemma~\ref{lem:notinpositroid}]

 Let $\rmci_{\pi}$ be a forward Grassmann necklace. Suppose $\gnec=(I_1, \dots, I_n)$ is a Grassmannlike necklace which can be obtained from $\rmci_\pi$ by a sequence of noncrossing toggles. 
 
 We prove the following more specific claim which readily implies the desired statement. We abbreviate $L = I_{\rho^{-1}(a)}$ and $R = I_{\rho^{-1}(a)+1}$.

\textit{Claim: There exist sets 
$\mcs,\mct \subset [n] \setminus \{\pi(a),a\}$, 
with 
\begin{align}\label{eq:LtRt}
L&= \rightI_{a} \setminus (\pi^{-1}\mct \cup \mcs)  \coprod \left(\mct  \cup \pi^{-1}\mcs \right) \\
R &= \rightI_{a+1} \setminus (\pi^{-1}\mct \cup \mcs)  \coprod \left(\mct  \cup \pi^{-1}\mcs \right) 
\end{align}
such that the pair of chords $\pi^{-1}(s) \mapsto s$ and $a \mapsto \pi(a)$ are noncrossing for all $s \in \mcs$, and likewise the chords $\pi^{-1}(t) \mapsto t$ and $a \mapsto \pi(a)$ are noncrossing for all $t \in \mct$.
}

In \eqref{eq:LtRt}, let us clarify that the use of $\setminus$ implies that the second set is contained in the first. (We do not adopt that convention in most other parts of the paper.) However, it not important that the sets $\pi^{-1}(\mct), \mcs$ are disjoint, and it is also not important that the sets $\mcs$ and $\mct$ are disjoint (i.e., we allow for removing an element that is in $\mcs$ and then adding it back in if it is in $\mct$).

We will establish this claim by induction on $\ell(f)-\ell(i)$ and then explain why it implies the statement in the lemma. 

The base case of \eqref{eq:LtRt} holds with $\mcs = \mct = \emptyset$. The subsets $L$ and $R$ only change meaning when we toggle at either $L$ or $R$. If we toggle at $R$, things look locally like
\begin{equation}
L \hspace{.3cm} \substack{\pi(a) \\ \rightleftarrows \\ a} \hspace{.3cm}  R \hspace{.3cm}  \substack{t \\ \rightleftarrows \\ \pi^{-1}(t)} \hspace{.3cm}  X.
\end{equation}

Let us denote by $L',R'$ the new versions of $L$ and $R$ after the toggle. Then $R' = X = R \setminus \pi^{-1}(t) \cup t$. The subset $L'$ is obtained by toggling at $R$; we clearly also have $L' = L\setminus \pi^{-1}(t)\cup t$. So $L',R'$ both evolve according to the formula \eqref{eq:LtRt} in this case. The claimed statement that these chords are noncrossing holds by assumption. Note also that $t \neq \pi(a)$. The argument in the case that we toggle at $L$ rather than at $R$ is similar, with the local picture looking like 
\begin{equation}
X \hspace{.3cm} \substack{s \\ \rightleftarrows \\ \pi^{-1}(s)} \hspace{.3cm}  L \hspace{.3cm}  \substack{\pi(a) \\ \rightleftarrows \\ a} \hspace{.3cm}  R
\end{equation}
and the subsets evolving according to the formula $L' = L \setminus s \cup \pi^{-1}(s)$ and  
$R' = R \setminus s \cup \pi^{-1}(s)$. Note again that $s \neq a,\pi(a)$. The claim holds by induction.

Since the chords $\pi^{-1}(s) \mapsto s$ and $a \mapsto \pi(a)$ are noncrossing, we either have that  $\{s,\pi^{-1}(s)\} \subset (a,\pi(a))$ or that 
$\{s,\pi^{-1}(s)\} \subset (\pi(a),a)$ (with both of these considered as cyclic subintervals of $[n]$). And similarly for $\{\pi^{-1}(t),t\}$. From \eqref{eq:LtRt}, it follows that 
$$\#\left(R \cap (a,\pi(a))\right) = \#\rightI_{a+1} \cap (a,\pi(a)).$$
If $J = L \setminus a \cup y= R \setminus \pi(a) \cup y$ where $y<_a \pi(a)$, then 
$$\#\left(J \cap (a,\pi(a))\right) = \#\left(R \cap (a,\pi(a))\right)+1 >\#(\rightI_{a+1} \cap (a,\pi(a)))$$
so that $J \notin \mcm_\pi$ using Oh's Theorem. 

Likewise, let $J = L \setminus y \cup \pi(a)$ in the situation $\pi(a)<_a y$. Then 
$$\#\rightI_{a} \cap [a,\pi(a)] = \#\left(L \cap [a,\pi(a)]\right) = \#(J \cap [a,\pi(a)])-1 ,$$
so that $J \notin \mcm_\pi$ using Oh's Theorem.

\end{proof}

\subsection{Proof of \cref{thm:main}: (1) $\iff$ (2) }\label{secn:conjugationproof} 
By \cref{rmk:togglePluckerRel}, aligned toggles correspond to those in which the Pl\"ucker relation has signs $I_iI_i'=I_{i-1}I_{i+1}+S_1S_2$. Such a relation ``looks like'' 
the three-term Pl\"ucker relation encoding a square move on weakly separated collections. However, such a Pl\"ucker relation does not necessarily correspond to a square move on weakly separated collections; that is, performing an aligned toggle does not preserve weak separation. So not all of the necklaces $\gnec_{\bullet, \iota, \pi}$ with $\iota \ler \pi$ are weakly separated.

Recall our general setup: we have $i \leq_R f$ which are lifts of permutations $\iota,\pi$ of type $(k,n)$. We define $\mu = \iota^{-1}\pi\iota$ which is a permutation of type $(k,n)$ with lift $m \in {\rm Bound}(k,n)$. By \cref{lem:conjugationLift}, $m=i^{-1}fi$. We always have $\ell(m) \leq \ell(f)$ and we want to characterize when $\ell(m) = \ell(f)$. 

To simplify statements, let $r=f^{-1}i$, so $m=r^{-1}i$.

\begin{lem}\label{lem:conjugation}
We have $\ell(m) < \ell(f)$ if and only if there exists a transposition $t \in T$ satisfying both $\ell(tr) < \ell(r)$ and 
$\ell(ti) < \ell(i)$.
\end{lem}

In other words, if $t = t_{ij}$, then the values $i$ and $j$ are ``out of order'' in both $\rho$ and $\iota$ (when both permutations are appropriately upgraded to affine permutations).  

\begin{proof}
Since $i \leq_R f$, the factorization $f = ir^{-1}$ is length-additive: $ \ell(f) = \ell(i)+
\ell(r^{-1})$. So $\ell(m) = \ell(f)$ if and only if the factorization $m = r^{-1}i$ is length-additive:
$\ell(r^{-1}i) = \ell(r^{-1})+\ell(i)$. 

By \cref{lem:lengthAddAndReflections}, $\ell(r^{-1}i) = \ell(r^{-1})+\ell(i)$ if and only if $T_R(r^{-1}) \cap T_L(i)= \emptyset$. Also $T_R(r^{-1})=T_L(r)$, so we are done.

\end{proof}

We now prove (1) implies (2) in Theorem~\ref{thm:main}, namely that $\ell(m) = \ell(f)$ is sufficient to guarantee that $\gnec_{\bullet,\iota,\pi}$ is weakly separated.
\begin{proof}[Proof of sufficiency]
We will suppose that $\gnec$ is not weakly separated and show that there exists a transposition $t \in T$ as in the statement of Lemma~\ref{lem:conjugation}. 

First we rephrase weak separation of $\gnec$ as a condition on the removal and insertion permutations $\rho$ and $\iota$. A 4-tuple of circularly ordered numbers $a < b < c < d$ are a witness for nonseparation of $\gnec$ if and only if there are values $x,y \in [n]$, such that 
\begin{align}
\{a,c\} \subset \iota([y,x)) &\text{ and } \{b,d\} \subset \rho([y,x))  \label{eq:wxyz}\\ 
\{a,c\} \subset \rho([x,y)) &\text{ and } \{b,d\} \subset \iota([x,y)). \label{eq:wxyzb}
\end{align}

Visually, we can ``chop'' $\gnec$ in the positions $I_x$ and $I_y$, breaking $[n] = \bb_1 \coprod \bb_2$ in two cyclic intervals $\bb_i$. In $\bb_1$ we see 
$\{a,c\}$ in the insertion permutation and $\{b,d\}$ in the removal permutation while in $\bb_2$ we see the opposite.

Now we switch from thinking about permutations $\rho,\iota$ to thinking about affine permutations.  We consider the two-line notation \eqref{eq:chordsnotationii} (extended bi-infinitely and $n$-periodically in both directions) whose numerator is $i$ and whose denominator is $r=f^{-1}i$. Reducing values modulo $n$ yields the permutations $\iota,\rho$ respectively.

As in the proof of Lemma~\ref{lem:conjugationLift}, we can reach this two-line notation by starting with the two-line notation whose numerator is $f$ and whose denominator is the identity $e_0 \in \tilde{S}^0_n$, and $n$-periodicallly performing swaps of adjacent columns. In particular, 
any column vector $\begin{matrix} \bb \\ \aa \end{matrix}$ appearing in the two-line notation satisfies $\aa \leq \bb \leq \aa+n $.

As in the proof of Lemma~\ref{lem:conjugationLift}, the appearance of any $x \in \bbz$ in the top row is weakly to the left of $x$ in the bottom row.

Suppose $a<b<c<d$ are a witness against weak separation as 
in \eqref{eq:wxyz}. We can uniquely lift these to linearly ordered numbers $a<b'<c'<d' \in \bbn$ such that $b' \in \{b,b+n\}$ and so on. Initially, the numbers $a,b',c',d'$ appear sorted in the order $a,\dots,b',\dots,c',\dots,d'$ in the denominator of the two-line notation. To reach $\gnec$, we perform a sequence of column swaps so that $\{a,c'\}$ and $\{b',d'\}$ are adjacent to each other in the bottom row.

For example, this might happen by starting with $a,\dots,b',\dots,c',\dots,d'$ in the bottom row, performing swaps until we reach $a,\dots,b',c',\dots,d$ with $b',c'$ adjacent, and then performing the swap that switches $b',c'$ yielding $a,\dots,c',b',\dots,d$. Once we have done this, the values $c',b'$ henceforth remain out of order in the bottom row, and in particular we would have $\ell(t_{(b',c')}r) < \ell(r)$. If we perform 
a further sequence of swaps and arrive at the picture 
$$\begin{matrix}\{b,d\} \\ \{a,c\}
\end{matrix} \cdots \begin{matrix}\{a,c\} \\ \{b,d\}
\end{matrix} $$
modulo $n$, we conclude that the picture in fact looks like
$$\begin{matrix}\{b',d'\} \\ \{a,c'\}
\end{matrix} \cdots \begin{matrix}\{a+n,c'+n\} \\ \{b',d'\}
\end{matrix} \cdots \begin{matrix}\{b'+n,d'+n\} \\ \{a'+n,c'+n\}
\end{matrix}, 
$$
using the fact $a$ in the top row appears left of $a$ in the bottom row, etc. The values $b',c'$ are also out of order in the numerator, i.e. we have $\ell(t_{b'c'}i) < \ell(i)$, as desired. 

We have been considering the special case where $b'$ swaps past $c'$, but it is straightforward to see that it is necessary to perform at least one of the swaps ($d$ past $a$, $a$ past $b'$, $b'$ past $c'$, or $c'$ past $d'$) and the argument is identical.  
\end{proof}

Now we prove (2) implies (1) in Theorem~\ref{thm:main}, i.e. that the condition $\ell(m) = \ell(f)$ is necessary for the necklace $\gnec$ to be weakly separated. 
  
\begin{proof}[Proof of necessity]

We will show that if there exists $a<b$ such that $\ell(t_{ab}r) < \ell(r)$ and $\ell(t_{ab}i) < \ell(i)$, then we can chop the two-line notation
\eqref{eq:chordsnotationii} in two pieces as in \eqref{eq:wxyz} and \eqref{eq:wxyzb}. As in the above proof of sufficiency, we work with two-line notation for affine permutations. By assumption, the two-line notations 
looks like 
\begin{equation}
\begin{matrix}\cdots b \cdots a \cdots \cdots 
\cdots \\
\cdots \cdots \cdots \cdots b \cdots a \cdots  \end{matrix}.
\end{equation}
The relative positions of the $a$ in the numerator and the $b$ in the denominator are not important for our argument.

We chop the necklace 
as indicated by vertical bars
\begin{equation}
\begin{matrix}
\cdots b  \\
\cdots 
\end{matrix}
\begin{vmatrix} \cdots a \cdots \cdots \\\ \cdots \cdots b \cdots 
\end{vmatrix}
\begin{matrix}  \cdots \\
a \cdots
\end{matrix}
\end{equation}
(that is, just {\sl after} the $b$ in the top row and just before the $a$ in the bottom row). 
Let $\mcb^- \subseteq (-\infty,a)$ be those values lying in within the vertical bars and in the bottom row of the two-line notation. Similarly, we let 
$\mct^- \subseteq (-\infty,a)$ be those values lying within the vertical bars and in the top row of the two-line notation. 

We claim that $\mcb^- \setminus \mct^-$ is nonempty. We have that $f^{-1}(a) \in \mcb^-$. Then the claim follows from noting that if $z \in \mct^-$ then there exists an element of $\mcb^-$ which is strictly less than $z$ (namely, the element $f^{-1}(z)$). 

We can likewise set $\mcb^+:= (b,\infty) \cap \textnormal{ bottom row}$ and 
$\mct^+:= (b,\infty) \cap \textnormal{ top row}$ (again, only considering those values within the vertical bars). Then $f(b) \in \mct^+$ and we claim that $\mct^+ \setminus \mcb^+$ is nonempty. This follows similarly from as above, noting that 
if $z \in \mcb^+$ then there exists an element of $\mct^+$ which is strictly greater than $z$ (namely, the element $f(z)$).

Letting $z_- \in \mcb^- \setminus \mct^-$, 
$z_+ \in \mct^+ \setminus \mcb^+$, we have 
$z_- < a < b < z_+$ are a witness against weak separation. 
\end{proof}

\subsection{Proof of \cref{thm:alignedIsQuasiEquiv}} \label{sec:reachablepf}
We prove \cref{thm:alignedIsQuasiEquiv} using the technology of plabic tilings introduced in \cite[Section 9]{OPS}. We briefly review the definition here.

Let $p_1, \dots, p_n \in \mathbb{R}^2$ be the vertices of a regular $n$-gon listed in clockwise order. For $I \subset [n]$, we use the notation $p(I):=\sum_{i \in I} p_i$ (where the summation here is summation of vectors in $\mathbb{R}^2$). As usual, we abbreviate $X \cup \{a\}$ as $Xa$.

Given a weakly separated collection $\mcc \subset \binom{[n]}{k}$, the associated plabic tiling $\mct(\mcc)$ is a 2-dimensional CW-complex embedded in $\bbr^2$. The vertices are $\{p(I): I \in \mcc\}$. Faces correspond to nontrivial black and white cliques in $\mcc$. For $X \in \binom{[n]}{k-1}$, the white clique $\mcw(X)$ consists of all $I \in \mcc$ which contain $X$. Similarly, for $X \in \binom{[n]}{k+1}$, the black clique $\mcb(X)$ consists of all $I \in \mcc$ which are contained in $X$. A clique is nontrivial if it has more than two elements. The elements of a white clique, for example, are $Xa_1, \dots, Xa_r$, where $a_1, \dots, a_r$ are cyclically ordered; we have edges between $p(Xa_i)$ and $p(Xa_{i+1})$ in $\mct(\mcc)$. The edges between vertices in a black clique are similar.

\begin{lem} \label{lem:dkkCondition}
Consider a pair of permutations $\iota \ler \pi$ and let 
$\gnec = \gnec_{\bullet, \iota, \pi}$ be a Grassmannlike necklace with terms $(I_1, \dots, I_n)$. Let $u \to \pi(u)$ and $v \to \pi(v)$ be a pair of crossing chords and set $w:=\pi(u)$ and $x:=\pi(v)$. Then $\gnec$ does not contain a quadruple $I_a, I_{a+1}, I_b, I_{b+1}$ such that either $(I_a, I_{a+1})=(Xu, Xw)$ and $(I_b, I_{b+1})=(Xv, Xx)$, or $(I_a, I_{a+1})=(X\setminus u, X\setminus  w)$ and $(I_b, I_{b+1})=(X \setminus v, X \setminus x)$. \end{lem}

\begin{proof}
	We argue by contradiction. Suppose we have $(I_a, I_{a+1})=(Xu, Xw)$ and $(I_b, I_{b+1})=(Xv, Xx)$ for crossing chords $u \to w$ and $v \to x$. Since $u$ is removed from $I_a$, we have that $a=\rho^{-1}(u)$. We also have that $w=\pi(u)$. Since the chords $u \to w$ and the chords $v \to x$ cross, we know that one of $v, x$ is in the cyclic interval from $u$ to $\pi(u)$. Let's say it's $x$ (the other case is identical). That is, we have $x <_u \pi(u)$. By \cref{lem:notinpositroid}, this means that $I_{\rho^{-1}(u)} \setminus u \cup x= Xx$ is not in the matroid $\mcm_\pi$. But by \cref{thm:rightweak}, $\gnec$ is a unit necklace and so in particular, all terms are in $\mcm_\pi$, a contradiction.
	
	The other case is identical.
\end{proof}

For the remainder of this section, we fix a Grassmannlike necklace $\gnec=\gnec_{\bullet,\iota, \pi}$ with $\iota \ler \pi$. To this necklace $\gnec = (I_1,\dots,I_n)$ we associate the polygonal curve $\zeta(\gnec)$ with vertices $p(I_1),\dots,p(I_n)$ in that order. Since $\gnec$ can have repeated terms, the curve $\zeta(\gnec)$ can have self-intersections at vertices. However, the conclusion of \cref{lem:dkkCondition} says that $\zeta(\gnec)$ has no self-intersections involving a pair of crossing edges. Thus,  $\zeta(\gnec)$ is a union of simple closed polygonal curves meeting only at their vertices. 

We have moreover that none of these simple closed curves encloses another. Indeed, if this were true, then using \cref{rmk:FGinside}, one can 
replace each $p(I)$ on these curves by its corresponding $p(\rho^{-1}(I))$ to obtain a similar pair of nested simple curves in the polygonal curve for the Grassmann necklace $\rmci_\mu$. One knows that Grassmann necklaces do not have such nested polygonal curves, see \cite[Section 9]{OPS}. 

Denote by $D^{\rm in}_\gnec$ (resp. $D^{\rm out}_\gnec$) those $k$-element subsets $I$ which are weakly separated from $\gnec$ and whose corresponding point~$p(I)$ are weakly enclosed by (resp. weakly outside of) the curve $\zeta(\gnec)$. 

\begin{lemma}\label{lem:FGintersect}
If $\mathcal{C}$ is a maximal weakly separated collection containing $\gnec$, then 
$\mathcal{C} \cap D^{\rm in}_\gnec$ coincides with $\targetF(G^\rho)$ for 
(reduced) relabeled plabic graph $G^\rho$ whose trip permutation is $\pi$.
In particular, if $I \in D^{\rm in}_\gnec$, then $I \in \mcm_\pi$.
\end{lemma}

\begin{proof}The first assertion follows from the results mentioned in \cref{rmk:FGinside}, working one simple curve at a time. The second statement about positroids follows from our main commutative diagram \cref{prop:faceLabelsTorus}. 
\end{proof}

\begin{proof}[Proof of \cref{thm:alignedIsQuasiEquiv}]
We abbreviate $\gnec = \gnec_{\rho,\bullet,\pi} = (I_1,\dots,I_n)$ and  $\gnec' = \gnec_{\sigma,\bullet,\pi}$. We let $\gnec = (I_1,\dots,I_n)$ and suppose the toggle takes place in position~$j$, with $I_j' \in \gnec'$ the new $k$-subset that arises by performing the toggle.

We first rule out the case that $I_{j-1} = I_{j+1}$. If $\iota_j=\rho_j$, then every relabeled plabic graph with boundary $\rho$ has a white lollipop at $\rho_j$. In this case, it is easy to find relabeled plabic graphs $G^\rho$ and $H^{\rho'}$ whose target seeds are identical (just move the white lollipop appropriately). Similarly, the $\iota_{j-1}=\rho_{j-1}$ case is easy, so we may assume $\iota_{j-1}, \rho_{j-1}, \iota_{j}, \rho_{j}$ are distinct.

For appropriate cyclically ordered indices $u < v < w < x$ and an appropriate subset $S \in \binom{[n]}{k-2}$ we have $I_{j-1}=Suv$, $I_{j}=Svx$, $I_{j+1}=Swx$, and $I'_j = Suw$. As in \cref{rmk:togglePluckerRel}, there is a three-term Pl\"ucker relation involving the Pl\"ucker coordinates $I_j$, $I'_j$, $I_{j-1}$,$I_{j+1}$, and the two ``extra'' terms $Sux, Svw$. By \cite[Lemma 5.1]{OS}, because both $I_j$ and $I'_j$ are weakly separated with $\gnec \setminus \{I_j,I'_j\}$, the collection $\gnec \cup Sux,Svw$ is weakly separated.

We next observe that $I_j$ appears only once in $\gnec$. This follows from the assumption that $\gnec'$ is weakly separated because $I_j$ and $I'_j$ are not weakly separated. Thus the vertex $p(I_j)$ has exactly two neighbors on the polygonal curve $\zeta(\gnec)$ and is not the location of a self-intersection of this curve. It also implies that neither $Sux$ nor $Svw$ resides in $\gnec$.

Next let $\mathcal{C}$ be a maximal weakly separated collection containing $\gnec \cup Sux,Svw$. Note that $p(I_j)$ has exactly four neighbors in the plabic tiling for $\mathcal{C}$. Three of these are included in the picture \eqref{eq:localpicture} below. The missing neighbor is $p(Svw)$ sitting ``above'' $p(I_j)$ and connected to it by a vertical edge. For brevity we replace the symbol $p(I)$ by the symbol $I$ in this picture.

Let $\mathcal{C}' = \mathcal{C} \setminus I_j \cup I'_j$ be the result of performing a square move at $I_j \in \mathcal{C}$. Note that $\gnec' \subset \mathcal{C}'$. By \cref{lem:FGintersect}, one has reduced plabic graphs $(G')^\rho$ and 
$(H')^\sigma$, defined by $D^{\rm in}_\gnec \cap \mathcal{C} = \targetF((G')^\rho)$ and $D^{\rm in}_{\gnec'} \cap \mathcal{C}' = \targetF((H')^\sigma)$. 

We make the following {\sl key claim}: the seeds $\Sigma^T_{(G')^\rho}$ and 
$\Sigma^T_{(H')^\sigma}$ are quasi-equivalent. Since $\Sigma^T_{(G')^\rho}$ is mutation-equivalent to $\Sigma^T_{G^\rho}$ and $\Sigma^T_{(H')^\sigma}$ is mutation-equivalent to  $\Sigma^T_{H^\sigma}$, we would then have that 
$\Sigma^T_{G^\rho}$ and $\Sigma^T_{H^\sigma}$ are related by a quasi-cluster transformation, completing the proof. 

Let us establish the key claim. The coefficients group of these two seeds coincide by the unit necklace theorem. Their sets of mutable variables coincide, since the two plabic graphs differ only in their $j$th boundary face. What needs to be checked is the equality of exchange ratios in the two seeds.

It is not hard to see that every edges of $\zeta(\gnec)$ is either an edge of the plabic tiling for~$\mathcal{C}$ or cuts across a face of this tiling. For topological reasons, it follows that exactly one of the two terms in $Sux$ and $Svw$ resides in $D^{\rm in}_\gnec$. Suppose for concreteness that this is true of $Sux$. By the proof of \cref{prop:monomialtransform}, the other term $Svw \notin \mcm_\pi$. Recall that $Sux \notin \gnec$, thus it is a mutable variable in $\Sigma^T_{(G')^\rho}$.

Performing the square move at $I_j$ only changes the edges of the tiling involving $p(I_j)$ and its four neighbors. Of these, only $Sux$ is a mutable variable and has an exchange ratio. We obtain the seed $\Sigma^T_{(G')^\rho}$ by restricting the edges from the plabic tiling for $\mathcal{C}$ to the elements of $D^{\rm in}_\gnec $ and ignoring any arrows between elements in $\gnec$. The equality of exchange ratios is now a simple local check using the change of coordinates $\Delta_{I_j'} = \frac{\Delta_{I_{j-1}} \Delta_{I_{j+1}}}{\Delta_{I_j}}$ \eqref{eq:monomialtransform} which arises when performing the toggle, see the picture below. 
	
	\begin{equation}\label{eq:localpicture}
		\begin{tikzpicture}
		\node (a)at (0,0) {$I_j$};
		\node (c) at (0,-1.5) {$Sux$};
		\node (d) at (1.5,0) {$I_{j+1}$};
		\node (e) at (-1.5,0) {$I_{j-1}$};
		\draw [->] (e)--(a);
		\draw [->] (d)--(a);
		\draw [->] (a)--(c);
		\draw[dotted] [->] (c)--(d);
		\draw[dashed] [->] (c)--(e);
		\begin{scope}[xshift = 5cm]
		\node (a)at (0,0) {$\frac{I_{j-1}I_{j+1}}{I_j}$};
		\node (c) at (0,-1.5) {$Sux$};
		\node (d) at (1.5,0) {$I_{j+1}$};
		\node (e) at (-1.5,0) {$I_{j-1}$};
		\draw [->] (a)--(e);
		\draw [->] (a)--(d);
		\draw [->] (c)--(a);
		\draw[dotted] [->] (d)--(c);
		\draw[dashed] [->] (e)--(c);
		\end{scope}
		\end{tikzpicture}
	\end{equation}
Here, the dotted ( resp. dashed) arrow is present on the left if and only if it is not present on the right. There may be other arrows between $Sux$ and other vertices, but these arrows in the same in both seeds. The calculation is not affected by simultaneously reversing all arrows in both pictures. 
\end{proof}